\documentclass{article} %
\usepackage[dvipsnames,table,xcdraw]{xcolor} %
\usepackage{iclr2023_conference,times}

\usepackage{hyperref}
\usepackage{url}

\usepackage{etoc}

\newcommand{\rbl}[1]{#1}

\newcommand{\betaparam}{\alpha}

\input{TeX/Preamble.tex} 

\newtheorem{example}{Example}
\newtheorem{remark}{Remark}

\title{Solving stochastic weak Minty variational inequalities without increasing batch size}

\author{%
    Thomas Pethick\thanks{Laboratory for Information and Inference Systems (LIONS), EPFL (\href{mailto:thomas.pethick@epfl.ch}{thomas.pethick@epfl.ch})} \And 
    Olivier Fercoq\thanks{Laboratoire Traitement et Communication d'Information, Télécom Paris, Institut Polytechnique de Paris} \And
    Puya Latafat\thanks{Department of Electrical Engineering (ESAT-STADIUS), KU Leuven} \And 
    Panagiotis Patrinos\footnotemark[3] \And 
    Volkan Cevher\footnotemark[1]
}

\iclrfinalcopy %
\begin{document}

\maketitle
\begin{abstract}
This paper introduces a family of stochastic extragradient-type algorithms for a class of nonconvex-nonconcave problems characterized by the weak Minty variational inequality (MVI). Unlike existing results on extragradient methods in the monotone setting, employing diminishing stepsizes is no longer possible in the weak MVI setting. This has led to approaches such as increasing batch sizes per iteration which can however be prohibitively expensive. In contrast, our proposed methods involves two stepsizes and only requires one additional oracle evaluation per iteration. We show that it is possible to keep one fixed stepsize while it is only the second stepsize that is taken to be diminishing, making it interesting even in the monotone setting.
Almost sure convergence is established and we provide a unified analysis for this family of schemes which contains a nonlinear generalization of the celebrated primal dual hybrid gradient algorithm.
\end{abstract}

\etocdepthtag.toc{mtchapter}
\etocsettagdepth{mtchapter}{subsection}
\etocsettagdepth{mtappendix}{none}

\vspace{-.2em}
\section{Introduction}
\label{sec:introduction}

Stochastic first-order methods have been at the core of the current success in deep learning applications.
These methods are mostly well-understood for minimization problems at this point.
This is even the case in the nonconvex setting where there exists matching upper and lower bounds on the complexity for finding an approximately stable point \citep{arjevani2019lower}. %

The picture becomes less clear when moving beyond minimization into nonconvex-nonconcave minimax problems---or more generally nonmonotone variational inequalities.
Even in the deterministic case, finding a stationary point is in general intractable \citep{daskalakis2021complexity,hirsch1987exponential}. %
This is in stark contrast with minimization where only global optimality is NP-hard.

An interesting nonmonotone class for which we \emph{do} have efficient algorithms is characterized by the so called \emph{weak Minty variational inequality} (MVI) \citep{diakonikolas2021efficient}.
This problem class captures nontrivial structures such as attracting limit cycles and is governed by a parameter $\rho$ whose negativity increases the degree of nonmonotonicity.
It turns out that the stepsize $\gamma$ for the exploration step in extragradient-type schemes lower bounds the problem class through $\rho > -\nicefrac{\gamma}{2}$ \citep{pethick2022escaping}.
In other words, it seems that we need to take $\gamma$ large to guarantee convergence for a large class.
\looseness=-1

This \rbl{reliance on a} large stepsize is at the core of why the community has struggled to provide a stochastic variants for weak MVIs.
The only known results effectively increase the batch size \emph{at every iteration} \cite[Thm. 4.5]{diakonikolas2021efficient}---a strategy that would be prohibitively expensive in most machine learning applications.
\Citet{pethick2022escaping} proposed \eqref{eq:seg+} which attempts to tackle the noise by only diminishing the second stepsize.
This suffices in the special case of unconstrained quadratic games but can fail even in the monotone case as illustrated in \Cref{fig:monotone}.
This naturally raises the following research question:

\begin{center}
\emph{Can stochastic weak Minty variational inequalities be solved without increasing the batch size?}
\end{center}

We resolve this open problem in the affirmative \rbl{when the stochastic oracles are Lipschitz in mean}, with a modification of stochastic extragradient called \emph{bias-corrected stochastic extragradient} (BCSEG+).
The scheme only requires one additional first order oracle call, while crucially maintaining the fixed stepsize.
Specifically, we make the following contributions:
\begin{enumerate}
  \item We show that it is possible to converge for weak MVI \emph{without} increasing the batch size, by introducing a bias-correction term.
    The scheme introduces \emph{no additional hyperparameters} and recovers the maximal range $\rho \in (-\nicefrac{\gamma}{2},\infty)$ of explicit deterministic schemes.
    \rbl{The rate we establish is interesting already in the star-monotone case where only \emph{asymptotic} convergence of the norm of the operator was known when refraining from increasing the batch size \citep[Thm. 1]{hsieh2020explore}}.
    Our result additionally carries over to another class of problem treated in \Cref{app:negF}, which we call \emph{negative} weak MVIs.
  \item We generalize the result to a whole family of schemes that can treat constrained and regularized settings.
  First and foremost the class includes a generalization of the forward-backward-forward (FBF) algorithm of \citet{Tseng2000modified} to stochastic weak MVIs.
  The class also contains a stochastic nonlinear extension of the celebrated primal dual hybrid gradient (PDHG) algorithm \citep{Chambolle2011firstorder}.
  Both methods are obtained as instantiations of the same template scheme, thus providing a unified analysis and revealing an interesting requirement on the update under weak MVI when only stochastic feedback is available.

  \item We prove almost sure convergence under the classical Robbins-Monro stepsize schedule of the second stepsize.
    This provides a guarantee on the last iterate, which is especially important in the nonmonotone case, where average guarantees cannot be converted into a single candidate solution.
    Almost sure convergence is challenging already in the monotone case where even stochastic extragradient may not converge \citep[Fig. 1]{hsieh2020explore}.
\end{enumerate}

\section{Related work}
\label{sec:relatedwork}

\paragraph{Weak MVI}
\citet{diakonikolas2021efficient} was the first to observe that an extragradient-like scheme called extragradient+ \eqref{eq:eg+} converges globally for weak MVIs with $\rho \in (-\nicefrac{1}{8L_F}, \infty)$.
This results was later tightened to $\rho \in (-\nicefrac{1}{2L_F}, \infty)$ and extended to constrained and regularized settings in \citep{pethick2022escaping}.
A single-call variant has been analysed in \citet{bohm2022solving}.
\rbl{Weak MVI is a star variant of cohypomonotonicity}, for which an inexact proximal point method was originally studied in \citet{combettes2004proximal}.
Later, a tight characterization was carried out by \citet{bauschke2021generalized} for the exact case.
It was shown that acceleration is achievable for an extragradient-type scheme even for cohypomonotone problems \citep{lee2021fast}.
Despite this array of positive results the stochastic case is largely untreated for weak MVIs.
The only known result \cite[Theorem 4.5]{diakonikolas2021efficient} requires the batch size to be increasing. %
Similarly, the accelerated method in \citet[Thm. 6.1]{lee2021fast} requires the variance of the stochastic oracle to decrease as $\mathcal O(1/k)$.

\paragraph{Stochastic \& monotone} \rbl{
When more structure is present the story is different since diminishing stepsizes becomes permissible.
In the monotone case rates for the gap function was obtained for stochastic Mirror-Prox in \citet{juditsky2011solving} under bounded domain assumption, which was later relaxed for the extragradient method under additional assumptions \citep{mishchenko2020revisiting}. 
The norm of the operator was shown to asymptotically converge for unconstrained MVIs in \citet{hsieh2020explore} with a double stepsize policy.
There exists a multitude of extensions for monotone problems:
Single-call stochastic methods are covered in detail by \citet{hsieh2019convergence},
variance reduction was applied to Halpern-type iterations \citep{cai2022stochastic}, 
cocoercivity was used in \citet{beznosikov2022stochastic},
and bilinear games studied in \citet{li2022convergence}. 
Beyond monotonicity, a range of structures have been explored such as 
MVIs \citep{song2020optimistic}, 
pseudomonotonicity \citep{kannan2019optimal,boct2021minibatch}, 
two-sided Polyak-\L{}ojasiewicz condition \citep{yang2020global}, 
expected cocoercivity \citep{loizou2021stochastic},
sufficiently bilinear \citep{loizou2020stochastic},
and strongly star-monotone \citep{gorbunov2022stochastic}.}

\paragraph{Variance reduction}
The assumptions we make about the stochastic oracle in \Cref{sec:setup} are similar to what is found in the variance reduction literature (see for instance \citet[Assumption 1]{alacaoglu2021stochastic} or \citet{arjevani2019lower}).
However, our use of the assumption are different in a crucial way.
Whereas the variance reduction literature uses the stepsize $\gamma \propto 1/L_{\hat{F}}$ (see e.g. \citet[Theorem 2.5]{alacaoglu2021stochastic}), we aim at using the much larger $\gamma \propto 1/L_F$.
For instance, in the special case of a finite sum problem of size $N$, the mean square smoothness constant $L_{\hat{F}}$ from \Cref{ass:AsymPrecon:stoch:stocLips} can be $\sqrt{N}$ times larger than $L_F$ (see \Cref{sec:finiteSum} for details).
This would lead to a prohibitively strict requirement on the degree of allowed nonmonotonicity through the relationship $\rho > -\nicefrac{\gamma}{2}$.

\paragraph{Bias-correction}
The idea of adding a correction term has also been exploited in minimization, specifically in the context of compositional optimization \cite{chen2021solving}.
Due to their distinct problem setting it suffices to simply extend stochastic gradient descent (SGD), albeit under additional assumptions such as \citep[Assumption 3]{chen2021solving}.
In our setting, however, SGD is not possible even when restricting ourselves to monotone problems.

\section{Problem formulation and preliminaries}
\label{sec:setup}
We are interested in finding $z \in \R^n$ such that the following inclusion holds,
\begin{equation}\label{eq:StrucIncl}
	0\in Tz := Az + Fz.
\end{equation}
A wide range of machine learning applications can be cast as an inclusion.
Most noticeable, a structured minimax problem can be reduced to \eqref{eq:StrucIncl} as shown in \Cref{sec:pdhg}.
\rbl{
We will rely on common notation and concepts from monotone operators (see \Cref{app:preliminaries} for precise definitions).
}
\begin{ass}\label{ass:AsymPrecon}
In problem \eqref{eq:StrucIncl}, 
\begin{enumerate}
\item \label{ass:AsymPrecon:M:Lip} The operator $F: \R^n\rightarrow\R^n$ is $L_F$-Lipschitz with $L_F \in [0,\infty)$, i.e.,
    \begin{equation}
    \|Fz - Fz'\|
    \leq
    L_F \|z-z'\|
    \quad \forall z,z' \in \R^n.
    \end{equation}
\item\label{ass:A:Struct}
	The operator $A:\R^n\rightrightarrows\R^n$ is a maximally monotone operator.  
\item\label{ass:AsymPrecon:Minty:Struct} Weak Minty variational inequality (MVI) holds, \ie, there exists a nonempty set $\mathcal S^{\star}\subseteq \zer T$ such that for all $z^\star\in \mathcal S^{\star}$ and some $\rho\in(-\tfrac{1}{2L_F}, \infty)$
\begin{equation}
\langle v,z - z^{\star}\rangle \geq \rho\|v\|^2, \quad \text{for all $(z, v)\in \graph T$.}
\end{equation}
\end{enumerate}
\end{ass}
\begin{remark}
In the unconstrained and smooth case ($A\equiv 0$), \Cref{ass:AsymPrecon:Minty:Struct} reduces to $\langle Fz,z - z^{\star}\rangle \geq \rho\|Fz\|^2$ for all $z \in \R^n$.
When $\rho = 0$ this condition reduces \rbl{to the MVI (i.e. star-monotonicity)}, while negative $\rho$ makes the problem increasingly nonmonotone.
Interestingly, the inequality is not symmetric and one may instead consider that the assumption holds for $-F$.
Through this observation, \Cref{app:negF} extends the reach of the extragradient-type algorithms developed for weak MVIs.
\end{remark}

\paragraph{Stochastic oracle}
We assume that we cannot compute $Fz$ easily, but instead we have access to the stochastic oracle $\hat F(z,\xi)$, which we assume is unbiased with bounded variance.
We additionally assume that $z \mapsto \hat F(z,\xi)$ is $L_{\hat F}$ Lipschitz continuous in mean and that it can be simultaneously queried under the same randomness.

\begin{ass}\label{ass:AsymPrecon:stoch}
For the operator $\hat F(\cdot,\xi):\R^n\rightarrow\R^n$ the following holds.
    \begin{enumerate}
        \item \label{ass:AsymPrecon:stoch:multioracle}
            Two-point oracle: 
            The stochastic oracle can be queried for any two points $z,z'\in\R^n$,
            \begin{equation}
            \hat F(z,\xi), \hat F(z',\xi) \quad \text{where} \quad{\xi \sim \mathcal{P}}.
            \end{equation}
        \item \label{ass:AsymPrecon:stoch:unbiased}
            Unbiased:
            \(%
                \mathbb{E}_{\xi}\left[\hat F(z,\xi)\right] = Fz \quad \forall z \in \R^n
            \).%
        \item  \label{ass:AsymPrecon:stoch:boundedvar}
            Bounded variance:
            \(%
                \mathbb{E}_{\xi}\left[\|\hat F(z,\xi)-\hat F(z)\|^2\right] \leq \sigma_F^2  \quad \forall z \in \R^n
            \).%
    \end{enumerate}
\end{ass}
\begin{ass}\label{ass:AsymPrecon:stoch:stocLips} 
The operator $\hat F(\cdot,\xi):\R^n\rightarrow\R^n$ is
Lipschitz continuous in mean with $L_{\hat F} \in [0,\infty)$:
\begin{equation}
\mathbb E_\xi\left[\|\hat F(z,\xi)-\hat F(z',\xi)\|^2\right] \leqslant L_{\hat F}^2\|z-z^{\prime}\|^2 
\quad \text{for all } z,z'\in\R^n.
\end{equation}
\end{ass}

\begin{remark}
\Cref{ass:AsymPrecon:stoch:multioracle,ass:AsymPrecon:stoch:stocLips} are also common in the variance reduction literature \citep{fang2018spider,nguyen2019finite,alacaoglu2021stochastic}, but in contrast with variance reduction we will not necessarily need knowledge of $L_{\hat F}$ to specify the algorithm, in which case the problem constant will only affect the complexity.
Crucially, this decoupling of the stepsize from $L_{\hat F}$ will allow the proposed scheme to converge for a larger range of $\rho$ in \Cref{ass:AsymPrecon:Minty:Struct}.
Finally, note that \Cref{ass:AsymPrecon:stoch:multioracle} commonly holds in machine learning applications, where usually the stochasticity is induced by the sampled mini-batch.
\end{remark}

\section{Method}
\label{sec:method}

To arrive at a stochastic scheme for weak MVI we first need to understand the crucial ingredients in the deterministic setting.
For simplicity we will initially consider the unconstrained and smooth setting, i.e. $A\equiv 0$ in \eqref{eq:StrucIncl}.
The first component is taking the second stepsize $\alpha$ smaller as done in extragradient+ \eqref{eq:eg+},
\begin{equation}
\label{eq:eg+}
\tag{EG+}
\begin{split}
\bar{z}^k &= z^k - \gamma Fz^k \\
z^{k+1} &= z^k - \alpha \gamma F\bar{z}^k
\end{split}
\end{equation}
where $\alpha\in(0,1)$.
Convergence in weak MVI was first shown in \cite{diakonikolas2021efficient} and later tightened by \cite{pethick2022escaping}, who characterized that smaller $\alpha$ allows for a larger range of the problem constant $\rho$.
Taking $\alpha$ small is unproblematic for a stochastic scheme where usually the stepsize is taken diminishing regardless.

However,
\cite{pethick2022escaping} also showed that the extrapolation stepsize $\gamma$ plays a critical role for convergence under weak MVI.
Specifically, they proved that a larger stepsize $\gamma$ leads to a looser bound on the problem class through $\rho>-\gamma/2$. \rbl{While a lower bound has not been established we provide an example in \Cref{fig:SEG+:counterexample} of \Cref{app:experiments} where small stepsize prevents convergence.} %
Unfortunately, picking $\gamma$ large (e.g. as $\gamma=\nicefrac{1}{L_F}$) causes significant complications in the stochastic case where both stepsizes are usually taken to be diminishing as in the following scheme,
\begin{equation}
\label{eq:seg}
\tag{SEG}
\begin{split}
\bar{z}^k &= z^k - \beta_k \gamma \hat F(z^k, \xi_k) \ \quad \text{with}\quad \xi_k \sim \mathcal{P}  \\
z^{k+1} &= z^k - \alpha_k \gamma \hat F(\bar{z}^k, \bar{\xi}_k) \quad \text{with}\quad \bar{\xi}_k \sim \mathcal{P}  \\
\end{split}
\end{equation}
where $\alpha_k = \beta_k \propto \nicefrac{1}{k}$. 
Even with a two-timescale variant (when $\beta_k>\alpha_k$) it has only been possible to show convergence for MVI \rbl{(i.e. when $\rho = 0$)} \citep{hsieh2020explore}.
Instead of decreasing both stepsizes, \cite{pethick2022escaping} proposes a scheme that keeps the first stepsize constant,
\begin{equation}
\label{eq:seg+}
\tag{SEG+}
\begin{split}
\bar{z}^k &= z^k - \gamma \hat F(z^k, \xi_k) \ \quad \text{with}\quad \xi_k \sim \mathcal{P}  \\
z^{k+1} &= z^k - \alpha_k \gamma \hat F(\bar{z}^k, \bar{\xi}_k) \quad \text{with}\quad \bar{\xi}_k \sim \mathcal{P}  \\
\end{split}
\end{equation}
However, \eqref{eq:seg+} does not necessarily converge even in the monotone case as we illustrate in \Cref{fig:monotone}.
The non-convergence stems from the bias term introduced by the randomness of $\bar z^k$ in $\hat F(\bar z^k, \bar{\xi}_k)$.
Intuitively, the role of $\bar z^k$ is to approximate the deterministic exploration step $\tilde{\bar z}^k:=z^k-\gamma Fz^k$.
While $\bar z^k$ is an unbiased estimate of $\tilde{\bar z}^k$ this does not imply that $\hat F(\bar z^k, \bar{\xi}_k)$ is an unbiased estimate of $F(\tilde{\bar z}^k)$.
Unbiasedness only holds in special cases, such as when $F$ is linear and $A\equiv 0$ for which we show convergence of \eqref{eq:seg+} in \Cref{sec:seg+} under weak MVI.
In the monotone case it suffice to take the exploration stepsize $\gamma$ diminishing \citep[Thm. 1]{hsieh2020explore}, but this runs counter to the fixed stepsize requirement of weak MVI.
\looseness=-1

Instead we propose \emph{bias-corrected stochastic extragradient+} (BC-SEG+) in \Cref{alg:WeakMinty:Sto:Struct}.
BC-SEG+ adds a bias correction term of the previous operator evaluation using the current randomness $\xi_k$. 
This crucially allows us to keep the first stepsize fixed.
We further generalize this scheme to constrained and regularized setting with \Cref{alg:WeakMinty:Sto:StructA} by introducing the use of the resolvent, $(\id + \gamma A)^{-1}$.

\begin{algorithm}[t]
    \caption{(BC-SEG+) Stochastic algorithm for problem \eqref{eq:StrucIncl} when $A\equiv 0$}%
    \label{alg:WeakMinty:Sto:Struct}%
    
\begin{algorithmic}[1]
	\Require
		\(z^{-1} = \bar z^{-1} = z^0  \in\R^n\)
		$\alpha_k \in (0,1)$,
		$\gamma \in(\lfloor-2 \rho\rfloor_{+}, 1 / L_F)$

\item[\algfont{Repeat} for \(k=0,1,\ldots\) until convergence]

\State\label{state:FZ:sto:Smooth}%
	Sample $\xi_k\sim \mathcal{P}$
\State\label{state:barzL:sto:Smooth}%
	\(
		\bar z^k = z^k - \gamma \hat F(z^k, \xi_k) + (1-\betaparam_k) \big( \bar z^{k-1} - z^{k-1} + \gamma \hat F(z^{k-1}, \xi_k) \big)
	\) 
\State\label{state:Fbarz:Smooth}%
	Sample $\bar\xi_k\sim \mathcal{P}$
\State\label{state:z+:sto:Smooth}%
	\(
		z^{k+1} = z^k - \alpha_k \gamma \hat F(\bar z^k, \bar \xi_k) 
	\)

\item[\algfont{Return}]
	\(z^{k+1}\)
\end{algorithmic}

\end{algorithm}

\section{Analysis of SEG+}
\label{sec:seg+}
\rbl{In the special case where $F$ is affine and $A \equiv 0$ we can show convergence of \eqref{eq:seg+} under weak MVI up to arbitrarily precision even with a large stepsize $\gamma$.}
\begin{thm} 
\label{thm:SEG+}
Suppose that \cref{ass:AsymPrecon,ass:AsymPrecon:stoch} hold. 
Assume $Fz := Bz+v$ and choose $\alpha_k \in (0,1)$ and $\gamma \in (0,1/L_F)$ such that $\rho \geq \gamma(\alpha_k - 1)/2$.
Consider the sequence $\seq{z^k}$ generated by \eqref{eq:seg+}.
Then for all $z^\star \in \mathcal S^\star$,
\begin{equation}
\begin{aligned}
\sum^{K}_{k=0}\tfrac{\alpha_k}{\sum_{j=0}^K \alpha_j} \mathbb E \|Fz^k\|^{2} &\leq
\tfrac{\|z^{0}-z^{\star}\|^{2} + \gamma^2(\gamma^2L_F^2 + 1) \sigma^2_F \sum_{j=0}^K \alpha_j^2 }{\gamma^2(1-\gamma^2L_F^2) \sum_{j=0}^K \alpha_j}.
\end{aligned}
\end{equation}
\end{thm}
The underlying reason for this positive results is that $\hat F(\bar z^k,\bar\xi_k)$ is unbiased when $F$ is linear.
This no longer holds when either linearity of $F$ is dropped or when the resolvent is introduced for $A \not\equiv 0$,
\rbl{in which case the scheme only converges to a $\gamma$-dependent neighborhood as illustrated in \Cref{fig:monotone}.
This is problematic in weak MVI where $\gamma$ cannot be taken arbitrarily small (see \Cref{fig:SEG+:counterexample} of \Cref{app:experiments}).}

\begin{figure}[t]
\vspace{-1.5em}
\centering
\includegraphics[width=0.5\textwidth]{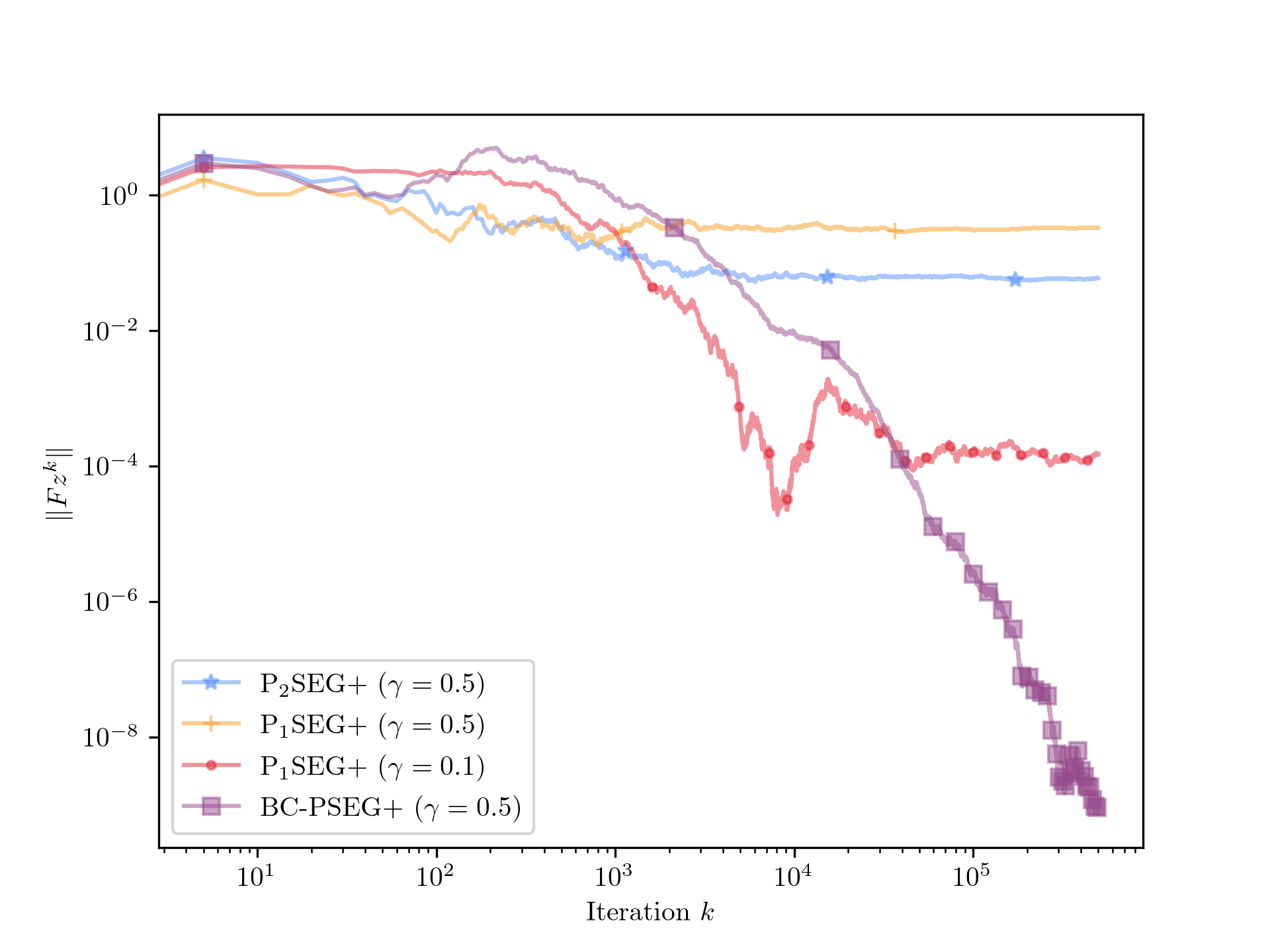}%
\includegraphics[width=0.5\textwidth]{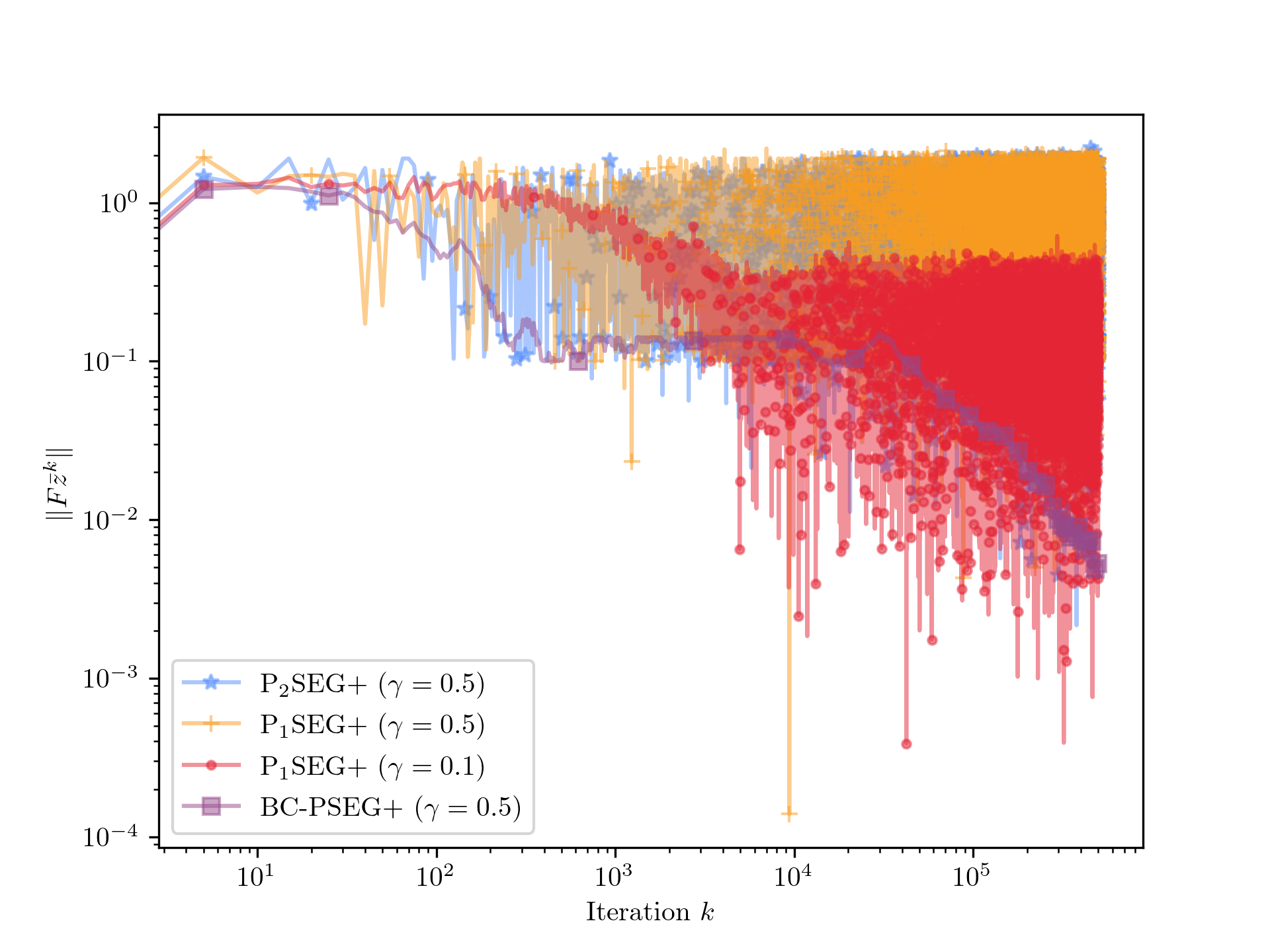}
\caption{Monotone constrained case illustrating the issue for projected variants of \eqref{eq:seg+} (see \Cref{app:algo} for algorithmic details). The objective is bilinear $\phi(x,y)=(x-0.9)\cdot(y-0.9)$ under box constraints $\|(x,y)\|_\infty\leq 1$. The unique stationary point $(x^\star,y^\star)=(0.9,0.9)$ lies in the interior, so even $\|Fz\|$ can be driven to zero.
\rbl{Despite the simplicity of the problem both projected variants of \eqref{eq:seg+} only converges to a $\gamma$-dependent neighborhood.
For weak MVI with $\rho < 0$ this neighborhood cannot be made arbitrarily small since $\gamma$ cannot be taken arbitrarily small (see \Cref{fig:SEG+:counterexample} of \Cref{app:experiments}).}
}
\label{fig:monotone}
\vspace{-1.5em}
\end{figure}

\section{Analysis for unconstrained and smooth case}
\label{sec:analysis}

For simplicity we first consider the case where $A\equiv 0$.
To mitigate the bias introduced in $F(\bar z^k,\bar \xi_k)$ for \eqref{eq:seg+}, we propose \Cref{alg:WeakMinty:Sto:Struct} which modifies the exploration step.
The algorithm can be seen as a particular instance of the more general scheme treated in \Cref{sec:analysis:const}.

\begin{thm}\label{thm:BiasCorr:2}
     Suppose that \cref{ass:AsymPrecon,ass:AsymPrecon:stoch,ass:AsymPrecon:stoch:stocLips} hold.
     Suppose in addition that \(\gamma\in (\lfloor-2\rho\rfloor_+, \nicefrac1{L_F})\) and \(\seq{\alpha_k}\subset(0,1)\) is a diminishing sequence such that 
     \begin{align}\label{eq:condest:Ae0}
         2\gamma L_{\hat{F}}\sqrt{\alpha_{0}}+\Big(1+\big(\tfrac{1+\gamma^2L_F^2}{1-\gamma^2L_F^2}\gamma^{2}L_{F}^{2}\big)\gamma^{2}L_{\hat{F}}^{2}\Big)\alpha_{0}\leq1+\tfrac{2\rho}{\gamma}.
     \end{align}
    Then, the following estimate holds for all $z^\star \in \mathcal S^\star$
\begin{align}
\label{thm:BiasCorr:2:rate}
\mathbb E[\|F(z^{k_\star})\|^2] 
    {}\leq{}
\frac{%
    (1+\eta \gamma^2L_F^2) \|z^{0} - z^\star\|^2 
        {}+{}
    C \sigma_F^2\gamma^2 \sum_{j=0}^K \alpha_j^2
     }{%
        \mu \sum_{j=0}^K \alpha_j
    }
\end{align}
where \(C = 1+2\eta\big((\gamma^{2}L_{\hat{F}}^{2}+1)+2\alpha_{0}\big)\), 
$\eta=\tfrac{1}{2}\tfrac{1+\gamma^2L_F^2}{1-\gamma^2L_F^2}\gamma^{2}L_{F}^{2}+\tfrac{1}{\gamma L_{\hat{F}}\sqrt{\alpha_{0}}}$, $\mu = \gamma^2(1-\gamma^2L_F^2)/2$
and 
$k_{\star}$ is chosen from $\{0,1, \ldots, K\}$ according to probability $\mathcal{P}\left[k_{\star}=k\right]=\frac{\alpha_k}{\sum_{j=0}^K \alpha_j}$.
 \end{thm} 
\begin{rem}
As $\alpha_0 \rightarrow 0$, the requirement \eqref{eq:condest:Ae0} reduces to $\rho > -\nicefrac{\gamma}{2}$ as in the deterministic setting of \citet{pethick2022escaping}.
Letting $\alpha_k = \nicefrac{\alpha_0}{\sqrt{k+r}}$ the rate becomes $\mathcal O(\nicefrac{1}{\sqrt{k}})$, thus matching the rate for the gap function of stochastic extragradient in the monotone case (see e.g. \citet{juditsky2011solving}).
\end{rem}

The above result provides a rate for a \emph{random iterate} as pioneered by \citet{ghadimi2013stochastic}. 
Showing \emph{last iterate} results even asymptotically is more challenging.
Already in the monotone case, vanilla \eqref{eq:seg} (where $\beta_k=\alpha_k$) only has convergence guarantees for the average iterate \citep{juditsky2011solving}.
In fact, the scheme can cycle even in simple examples \citep[Fig. 1]{hsieh2020explore}.

Under the classical (but more restrictive) Robbins-Monro stepsize policy, it is possible to show almost sure convergence for the iterates generates by \Cref{alg:WeakMinty:Sto:Struct}.
The following theorem demonstrates the result in the particular case of \(\alpha_k = \nicefrac{1}{k+r}\).
The more general statement is deferred to \Cref{app:smooth}.

\begin{thm}[almost sure convergence]\label{thm:BiasCorr}
Suppose that \cref{ass:AsymPrecon,ass:AsymPrecon:stoch,ass:AsymPrecon:stoch:stocLips} hold. 
Suppose \(\gamma\in (\lfloor-2\rho\rfloor_+, \nicefrac1{L_F})\), 
\(\alpha_k = \tfrac{1}{k+r}\) for any positive natural number \(r\) and
\begin{equation}\label{eq:BiasCorr:requirements}
    (\gamma L_{\hat{F}}+1)\alpha_{k}+\rbl{2}\left(\tfrac{1+\gamma^2L_F^2}{1-\gamma^2L_F^2}\gamma^{4}\rbl{L_{F}^{2}L_{\hat F}^{2}}\alpha_{k+1}+\gamma L_{\hat{F}}\right)\left(\alpha_{k+1}+1\right)\alpha_{k+1} 
        {}\leq{}
    1 + \tfrac{2\rho}{\gamma}.
\end{equation}
Then, the sequence \(\seq{z^k}\) generated by \Cref{alg:WeakMinty:Sto:Struct} converges almost surely to some \(z^\star\in \zer T\). 

\end{thm}
\begin{rem}
As $\alpha_k \rightarrow 0$ the condition on $\rho$ reduces to $\rho > - \nicefrac{\gamma}{2}$ like in the deterministic case.
\end{rem}

To make the results more accessible, both theorems have made particular choices of the free parameters from the proof, that ensures convergence for a given $\rho$ and $\gamma$.
However, since the parameters capture inherent tradeoffs, the choice above might not always provide the tightest rate.
Thus, the more general statements of the theorems have been preserved in the appendix.

\section{Analysis for constrained case}
\label{sec:analysis:const}

The result for the unconstrained smooth case can be extended when the resolvent is available.
\Cref{alg:WeakMinty:Sto:StructA} provides a direct generalization of the unconstrained \Cref{alg:WeakMinty:Sto:Struct}.
The construction relies on approximating the deterministic algorithm proposed in \cite{pethick2022escaping}, which iteratively projects onto a half-space which is guaranteed to contain the solutions.
By defining $Hz = z - \gamma Fz$, the scheme can concisely be written as,
\begin{equation}\label{eq:ceg+}
\tag{CEG+}
\begin{split}
\bar z^k &= (I + \gamma A)^{-1}(Hz^k) \\
z^{k+1} &= z^k - \alpha_k (Hz^k - H\bar z^k),
\end{split}
\end{equation}
for a particular adaptive choice of $\alpha_k \in (0,1)$.
With a fair amount of hindsight we choose to replace $Hz^k$ with the bias-corrected estimate $h^k$ (as defined in \Cref{state:h:sto:Struct} in \Cref{alg:WeakMinty:Sto:StructA}), such that the estimate is also reused in the second update.

\begin{thm}\label{thm:const:convergence}
    Suppose that \cref{ass:AsymPrecon,ass:AsymPrecon:stoch,ass:AsymPrecon,ass:AsymPrecon:stoch:stocLips} hold. 
    Moreover, suppose that \(\alpha_k\in(0,1)\), \(\gamma\in (\lfloor-2\rho\rfloor_+, \nicefrac1{L_F})\) and the following holds,
    \begin{equation}\label{thm:const:convergence:conditions}
    \mu \coloneqq \tfrac{1-\sqrt{\alpha_0}}{1+\sqrt{\alpha_0}} - \alpha_0(1+2\gamma^2L_{\hat F}^2 \eta) + \tfrac{2\rho}{\gamma} > 0
    \end{equation}
    where $\eta \geq \tfrac{1}{\sqrt{\alpha_0}(1-\gamma^2 L_F^2)} + \tfrac {1-\sqrt{\alpha_0}}{\sqrt{\alpha_0}}$. 
    Consider the sequence \(\seq{z^k}\) generated by \Cref{alg:WeakMinty:Sto:StructA}.
    Then, the following estimate holds for all $z^\star \in \mathcal S^\star$ 
    \begin{align*}
        \rbl{
        \mathbb E[\dist(0,T\z^{k_\star})^2]
        }
            {}\leq{}
        \frac{\mathbb E[\|z^{0}-z^{\star}\|^{2}] + \eta \mathbb E[\|h^{-1}-Hz^{-1}\|^{2}] + C\gamma^2 \sigma_F^2\sum_{j=0}^K \alpha_{j}^{2}}{\gamma^2\mu \sum_{j=0}^K \alpha_{j}}
    \end{align*}
    where $C = 1 + 2\eta(1+\gamma^2L_{\hat F}^2) + 2\alpha_0 \eta$ and
$k_{\star}$ is chosen from $\{0,1, \ldots, K\}$ according to probability $\mathcal{P}\left[k_{\star}=k\right]=\frac{\alpha_k}{\sum_{j=0}^K \alpha_j}$.
\end{thm}

\begin{remark}
The condition on $\rho$ in \eqref{thm:const:convergence:conditions} reduces to $\rho > - \nicefrac{\gamma}{2}$ when $\alpha_0 \rightarrow 0$ as in the deterministic case.
As oppose to \Cref{thm:BiasCorr} which tracks $\|Fz^k\|^2$, the convergence measure of \Cref{thm:const:convergence} reduces to $\dist(0,T\z^{k})^2=\|F\z^{k}\|^2$ when $A\equiv 0$.
Since \Cref{alg:WeakMinty:Sto:Struct} and \Cref{alg:WeakMinty:Sto:StructA} coincide when $A \equiv 0$, \Cref{thm:const:convergence} also applies to \Cref{alg:WeakMinty:Sto:Struct} in the unconstrained case.
Consequently, we obtain rates for both $\|F\bar z^k\|^2$ and $\|Fz^k\|^2$ in the unconstrained smooth case. 
\end{remark}

\begin{algorithm}[t]
    \caption{(BC-PSEG+) Stochastic algorithm for problem \eqref{eq:StrucIncl}}%
    \label{alg:WeakMinty:Sto:StructA}%
    
\begin{algorithmic}[1]
	\Require
		\(z^{-1} = z^0 \in\R^n\),
		$h^{-1}\in\R^n$,
		$\alpha_k \in (0,1)$,
		$\gamma \in(\lfloor-2 \rho\rfloor_{+}, 1 / L_F)$

\item[\algfont{Repeat} for \(k=0,1,\ldots\) until convergence]

\State\label{state:FZ:sto:Struct}%
	Sample $\xi_k\sim \mathcal{P}$
\State\label{state:h:sto:Struct}%
	$h^k = \big( z^k - \gamma \hat F(z^k, \xi_k)\big) + (1-\alpha_k) \Big( h^{k-1} - \big(z^{k-1} - \gamma \hat F(z^{k-1}, \xi_k)\big) \Big)$
\State\label{state:barzL:sto:Struct}%
	$\bar z^k = (\id + \gamma A)^{-1} h_k$
\State\label{state:Fbarz:Struct}%
	Sample $\bar\xi_k\sim \mathcal{P}$
\State\label{state:z+:sto:Struct}%
	$z^{k+1} = z^k - \alpha_k\big(h^k - \bar z^k + \gamma \hat F(\bar z^k, \bar \xi_k)\big)$

\item[\algfont{Return}]
	\(z^{k+1}\)
\end{algorithmic}

\end{algorithm}

\section{Asymmetric \& nonlinear preconditioning}
\label{sec:precon}

{%
In this section we show that the family of stochastic algorithms which converges under weak MVI can be expanded beyond \Cref{alg:WeakMinty:Sto:StructA}. This is achieved by extending \eqref{eq:ceg+} through introducing a nonlinear and asymmetrical preconditioning. 
Asymmetrical preconditioning has been used in the literature  to unify a large range of algorithm in the monotone setting \cite{Latafat2017Asymmetric}. A subtle but crucial difference, however, is that the preconditioning considered here depends nonlinearly on the current iterate. As it will be shown in \Cref{sec:pdhg} this nontrivial feature is the key for showing convergence for primal-dual algorithms in the nonmonotone setting. 

Consider the following generalization of \eqref{eq:ceg+} by introducing a potentially  \emph{asymmetric nonlinear} preconditioning \(P_{z^k}\) that depends on the current iterate \(z^k\). 
\begin{subequations}\label{eq:AFBA} 
    \begin{align}
    \textrm{find} \; \bar z^k \; \textrm{such that} \quad &
    \HC[z^k][z^k]
    {}\in{}
    \PC[z^k][\z^k]+A(\z^k),
    \label{eq:AFBA:Stoc:bar}
       \\
   \textrm{update} \quad&
   z^{k+1} 
   {}={}
    z^k + \alpha \Gamma
    \left(%
        \HC[z^k][\z^k] - \HC[z^k][z^k]
    \right). 
    \label{eq:AFBA:Stoc:plus}
    \qquad 
 \end{align}
\end{subequations}
where \(\HC[u][v]\coloneqq \PC[u][v] - F(v)\) and \(\Gamma\) is some positive definite matrix. 
The iteration independent and diagonal choice \(P_{z^k} = \gamma^{-1} \I\)  and \(\Gamma=\gamma \I\) correspond to the basic \eqref{eq:ceg+}. 
More generally we consider
\begin{equation}\label{eq:AFBA:P:det}
\PC[u][z] {}\coloneqq{}\Gamma^{-1} z + \QC[u][z]
\end{equation} 
where $\QC[u][z]$ captures the nonlinear and asymmetric part, which ultimately enables alternating updates and relaxing the Lipschitz conditions (see \Cref{rem:PDHG:Lips2}).
Notice that the iterates above does not always yield well-defined updates and one must inevitably impose additional structures on the preconditioner (we provide sufficient condition in \Cref{app:AFBA:prelim}).
Consistently with \eqref{eq:AFBA:P:det}, in the stochastic case we define
\begin{equation}
\SPC[u][z] \coloneqq \Gamma^{-1} z + \SQC[u][z].
\end{equation}}
The proposed stochastic scheme, which introduces a carefully chosen bias-correction term, is summarized as 
\begin{subequations}\label{eq:AFBA:Stoc} 
    \begin{align}
    \textrm{compute} \quad
    h^k 
        {}={}&
    \SPC[z^k][z^k][\xi_k]-\hat{F}(z^k,\xi_k) + (1 - \alpha_{k})\Big(h^{k-1} - \SPC[z^{k-1}][z^{k-1}][\xi_k]+\hat{F}(z^{k-1},\xi_k)  \label{eq:AFBA:h}
    \\
        &{}-{}
        \SQC[z^{k-1}][\z^{k-1}][\hxi_{k-1}] + \SQC[z^{k-1}][\z^{k-1}][\hxi_{k}]
        \Big)  \qquad\text{with}\quad \xi_{k},\hxi_{k} \sim \mathcal{P}
         \notag 
    \\
    \textrm{find} \; \bar z^k \; \textrm{such that} \quad 
    &
    h^k
        {}\in{}
    \SPC[z^k][\bar{z}^k][\hxi_k]+A\bar{z}^k 
    \label{eq:AFBA:bar}
       \\
   \textrm{update} \quad
   &
   z^{k+1} 
        {}={}
    z^k + \alpha_k \Gamma
    \left(%
        \SPC[z^k][\z^k][\bar \xi_k]-\hat{F}(\z^k,\bar\xi_k) 
        - h^k 
    \right) 
    \qquad\text{with}\quad \bar\xi_k \sim \mathcal{P}
    \label{eq:AFBA:plus}
    \qquad 
 \end{align}
\end{subequations}

\begin{remark}
The two additional terms in \eqref{eq:AFBA:h} are due to the interesting interplay between weak MVI and stochastic feedback, which forces a change of variables (see \Cref{app:bias-correction}). 
\end{remark}

To make a concrete choice of $\SQC[u][z][\xi]$ we will consider a minimax problem as a motivating example 
\rbl{(see \Cref{app:AFBA:prelim} for a more general setup).}

\subsection{Nonlinearly preconditioned primal dual hybrid gradient}
\label{sec:pdhg}

\begin{algorithm}[t]
    \caption{Nonlinearly preconditioned primal dual extragradient (NP-PDEG) for solving \eqref{eq:prob:PD}}%
    \label{alg:NP-PDHG}%
    
\begin{algorithmic}[1]
	\Require
		\(z^{-1} = z^0 = (x^0, y^0)\) with \(x^0,x^{-1},\hat{x}^{-1},\bar{x}^{-1} \in \R^n\), \(y^0,y^{-1} \in\R^r\), 
		\(\theta \in[0,\infty)\), \(\Gamma_{1} \succ0\), \(\Gamma_{2} \succ 0\)
\item[\algfont{Repeat} for \(k=0,1,\ldots\) until convergence]
\State \(\xi_k \sim \mathcal{P}\)
\State\label{state:NP-PDHG:hatx}%
  \(
    \hat x^k = x^k-\Gamma_{1}\nabla_{x}\hat\varphi(z^k, \xi_k) + (1-\alpha_k) \big(\hat x^{k-1} - x^{k-1} + \Gamma_1 \nabla_x \hat \varphi(x^{k-1}, y^{k-1}, \xi_k)\big)
  \)
\State\label{state:NP-PDHG:barx}%
	\(
		\bar{x}^k  =\prox_{f}^{\Gamma_{1}^{-1}}
		\big(%
      \hat x^k
		\big)
	\)
\State \(\xi_k' \sim \mathcal{P}\)
\State\label{state:NP-PDHG:haty}%
  \(
    \hat y^k = y^k+\Gamma_{2}
			\big(
        \theta\nabla_{y}\hat\varphi(\bar{x}^k,y^k, \xi_k') + (1-\theta)\nabla_{y}\hat\varphi(z^k, \xi_k)
      \big)
	\) \\
	\(
		\qquad \quad + (1-\alpha_k)\Big(\hat y^{k-1} - y^{k-1} - \Gamma_{2}
				\big(
					\theta\nabla_{y}\hat\varphi(\bar{x}^{k-1},y^{k-1}, \xi_k') + (1-\theta)\nabla_{y}\hat\varphi(z^{k-1}, \xi_k)
				\big)
			\Big)
  \)
\State\label{state:NP-PDHG:bary}%
	\(
		\bar{y}^k  =\prox_{g}^{\Gamma_{2}^{-1}}
      \big(%
        \hat y^k
      \big)
	\)
\State \(\bar\xi_k \sim \mathcal{P}\)
\State\label{state:NP-PDHG:x}%
	\(
		x^{k+1}
			{}={}
		x^k + \alpha_k 
		\left(%
			\bar{x}^k
      - \hat x^k
      -\Gamma_{1} \nabla_{x}\hat\varphi(\bar{z}^k, \bar\xi_k)
		\right)
	\)
\State\label{state:NP-PDHG:y}%
	\(
		y^{k+1}
			{}={}
		y^k + \alpha_k
		\left(%
			\bar{y}^k
      - \hat y^k
      +\Gamma_{2} \nabla_{y}\hat\varphi(\bar{z}^k, \bar\xi_k) 
		\right)
	\)
\item[\algfont{Return}]
	\(z^{k+1}=(x^{k+1},y^{k+1})\)
\end{algorithmic}

\end{algorithm}

 We consider the problem of 
 \begin{equation}\label{eq:prob:PD}
     \minimize_{x\in\R^n} \maximize_{y\in\R^r} \quad f(x) + \varphi(x,y) - g(y).
 \end{equation}
 where $\varphi(x,y) := \mathbb E_\xi [\hat \varphi(x,y,\xi)]$.
 The first order optimality conditions may be written as the inclusion
 \begin{equation}\label{eq:AnF}
    0\in Tz \coloneqq A z + Fz,
    \quad \textrm{where} \quad 
    A = (\partial f, \partial g), \quad 
    F(z) = (\nabla_x \varphi(z), -\nabla_y \varphi(z)),
 \end{equation}
 while the algorithm only has access to the stochastic estimates $\hat F(z, \xi) \coloneqq (\nabla_x \hat\varphi(z,\xi), -\nabla_y \hat\varphi(z,\xi))$.
\begin{ass}\label{ass:PD}
For problem \eqref{eq:prob:PD}, let the following hold with a stepsize matrix $\Gamma = \blockdiag(\Gamma_{1}, \Gamma_{2})$ where $\Gamma_1\in \R^n$ and $\Gamma_2\in \R^r$ are symmetric positive definite matrices:%
    \begin{enumerate}
        \item \(f\), \(g\) are proper lsc convex 
        \item \label{ass:PD:Lip:phi} \(\varphi:\R^{n+r} \to \R\) is continuously differentiable and for some symmetric positive definite matrices \(D_{xx},D_{xy}, D_{yx}, D_{yy}\), the following holds for all \(z=(x,y), z^\prime=(x^\prime,y^\prime)\in\R^{n+r}\)
        \begin{align*}
            \|\nabla_{x}\varphi(z^\prime)-\nabla_{x}\varphi(z)\|^{2}_{\Gamma_{1}}    
                {}\leq{}
            L^2_{xx}\|x^\prime - x\|_{D_{xx}}^{2}
                {}+{}
            L^2_{xy}\|y^\prime - y\|_{D_{xy}}^{2},
            \\
            \qquad
            \|\nabla_{y}\varphi(z^\prime)-\theta\nabla_{y}\varphi(x^\prime,y)-(1-\theta)\nabla_{y}\varphi(z)\|_{\Gamma_{2}}^{2} 
                {}\leq{}
            L^2_{yx}\|x^\prime - x\|_{D_{yx}}^{2}
                {}+{}
            L^2_{yy}\|y^\prime - y\|_{D_{yy}}^{2}.
        \end{align*}
        \item Stepsize condition: \label{ass:PD:stepsize} 
             $L^2_{xx}D_{xx} + L^2_{yx}D_{yx} \prec \Gamma_{1}^{-1}
             \quad \text{and} \quad
             L^2_{xy}D_{xy} + L^2_{yy}D_{yy} \prec \Gamma_{2}^{-1}.$
        \item  \label{ass:PD:stoch:boundedvar}
            Bounded variance:
            \(%
                \mathbb{E}_{\xi}\left[\|\hat F(z,\xi)-\hat F(z',\xi)\|^2_{\Gamma}\right] \leq \sigma_F^2  \quad \forall z,z' \in \R^n
            \).%
        \item \label{ass:PD:Lipmean:phi}
        \(\hat \varphi(\cdot, \xi):\R^{n+r} \to \R\) is continuously differentiable and for some symmetric positive definite matrices \(D_{\widehat{xz}},D_{\widehat{yz}}, D_{\widehat{yx}}, D_{\widehat{yy}}\), the following holds for all \(z=(x,y), z^\prime=(x^\prime,y^\prime)\in\R^{n+r}\) and $v,v'\in\R^n$
        \looseness=-1
        \begin{align*}
            \text{for $\theta \in [0, \infty)$:}\quad &
                \mathbb E_{\xi}\left[\|\nabla_{x}\hat\varphi(z^\prime,\xi)-\nabla_{x}\hat\varphi(z,\xi)\|^{2}_{\Gamma_{1}}\right]    
                    {}\leq{}
                L^2_{\widehat{xz}}\|z^\prime - z\|_{D_{\widehat{xz}}}^{2}
                \\
            \text{if $\theta \neq 1$:}\quad &
                \mathbb E_{\xi}\left[\|
                        \nabla_{y}\hat\varphi(z,\xi)
                        - \nabla_{y}\hat\varphi(z',\xi)
                    \|_{\Gamma_{2}}^{2}
                \right]
                    {}\leq{}
                L^2_{\widehat{yz}}\|z^\prime - z\|_{D_{\widehat{yz}}}^{2}
                \\
            \text{if $\theta \neq 0$:}\quad &
                \mathbb E_{\xi}\left[\|
                        \nabla_{y}\hat\varphi(v',y',\xi)
                        - \nabla_{y}\hat\varphi(v,y,\xi)
                    \|_{\Gamma_{2}}^{2}
                \right]
                    {}\leq{}
                L^2_{\widehat{yx}}\|v^\prime - v\|_{D_{\widehat{yx}}}^{2}
                    {}+{}
                L^2_{\widehat{yy}}\|y^\prime - y\|_{D_{\widehat{yy}}}^{2}.
        \end{align*}    
    \end{enumerate}
\end{ass}

\begin{rem}
    In \Cref{alg:NP-PDHG} the choice of \(\theta \in [0,\infty)\) leads to different algorithmic oracles and underlying assumptions in terms of Lipschitz continuity in \Cref{ass:PD:Lip:phi,ass:PD:Lipmean:phi}.
    \begin{enumerate}
        \item 
            If \(\theta=0\) then the first two steps may be computed in parallel and we recover \Cref{alg:WeakMinty:Sto:StructA}. Moreover, to ensure \Cref{ass:PD:Lip:phi} in this case it suffices to assume for $L_x,L_y \in [0,\infty)$,
            \[
                \|\nabla_x \varphi(z^\prime) - \nabla_x \varphi(z)\| \leq L_{x}\|z^\prime -z\|,\quad
                \|\nabla_y \varphi(z^\prime) - \nabla_y \varphi(z)\| \leq L_{y}\|z^\prime -z\|.
            \]
        \item \label{rem:PDHG:Lips2}
            Taking \(\theta=1\) leads to Gauss-Seidel updates 
            \rbl{and a nonlinear primal dual extragradient algorithm }
            with sufficient Lipschitz continuity assumptions for some $L_x,L_y \in [0,\infty)$,
            \[
                \|\nabla_x \varphi(z^\prime) - \nabla_x \varphi(z)\| \leq L_{x}\|z^\prime -z\|,\quad
                \|\nabla_y \varphi(z^\prime) - \nabla_y \varphi(x^\prime,y)\| \leq L_{y}\|y^\prime -y\|.
            \]
    \end{enumerate}
    \vspace{-2em}
\end{rem}

\Cref{alg:NP-PDHG} is an application of \eqref{eq:AFBA:Stoc} applied for solving \eqref{eq:AnF}. 
In order to cast the algorithm as an instance of the template algorithm \eqref{eq:AFBA:Stoc}, we choose the positive definite stepsize matrix as $\Gamma = \blockdiag(\Gamma_{1}, \Gamma_{2})$ with $\Gamma_{1} \succ0,\; \Gamma_{2} \succ 0$,
and the nonlinear part of the preconditioner as
 \begin{equation}
 \label{eq:Q:def}
    \SQC[u][\z]
        {}\coloneqq{}
    \left(%
        0, 
        - \theta \nabla_y \hat \varphi(\bar x, y,\xi)
    \right),
    \quad \text{and} \quad 
    \QC[u][\z]
        {}\coloneqq{}
    \left(%
        0, 
        - \theta \nabla_y \varphi(\bar x, y)
    \right)
 \end{equation}
 where $u=(x,y)$ and $\bar z=(\bar x,\bar y)$.
 Recall $\HC[u][z] \coloneqq \PC[u][z] - F(z)$ and define $S_{u}(z;\z) \coloneqq H_{u}(z) - \QC[u][\z]$.
 The convergence in \Cref{thm:AsymPrecon:PDHG:convergence} depends on the distance between the initial estimate $\Gamma^{-1}\hat{z}^{-1}$ with $\hat{z}^{-1} = (\hat{x}^{-1},\hat{y}^{-1})$ and the deterministic $S_{z^{-1}}(z^{-1};\z^{-1})$.
 See \Cref{app:preliminaries} for additional notation.

\begin{thm}\label{thm:AsymPrecon:PDHG:convergence}
    Suppose that \cref{ass:AsymPrecon:Minty:Struct} to \ref{ass:AsymPrecon:stoch:unbiased} and \ref{ass:PD} hold.
    Moreover, suppose that \(\alpha_k\in(0,1)\), \rbl{\(\theta\in[0,\infty)\)} and the following holds,
    \begin{gather}\label{thm:AsymPrecon:PDHG:convergence:conditions}
    \mu \coloneqq \tfrac{1-\sqrt{\alpha_0}}{1+\sqrt{\alpha_0}}
        + \tfrac{2\rho}{\bar \gamma} 
        -\alpha_0 
        - 2\alpha_0 (\hat{c}_1 + 2\hat{c}_2(1+\hat{c}_3))\eta 
         > 0 
    \quad \text{and} \quad
    1 - 4\hat{c}_2\alpha_0 > 0
    \end{gather}
    where \(\bar \gamma\) denotes the smallest eigenvalue of \(\Gamma\), $\eta \geq {(1+4\hat{c}_2\alpha_0^2)(\tfrac{1}{\sqrt{\alpha_0}(1-L_M)^2} + \tfrac {1-\sqrt{\alpha_0}}{\sqrt{\alpha_0}})}/{(1 - 4\hat{c}_2\alpha_0)}$ and 
    \begin{gather*}
    \hat{c}_1 \coloneqq L^2_{\widehat{xz}}\|\Gamma D_{\widehat{xz}}\|+2(1-\theta)^2L^2_{\widehat{yz}}\|\Gamma D_{\widehat{yz}}\|+2\theta^2L^2_{\widehat{yy}}\|\Gamma_2D_{\widehat{yy}}\|, 
    \quad
    \hat{c}_2 \coloneqq 2\theta^2L^2_{\widehat{yx}}\|\Gamma_1 D_{\widehat{yx}}\|, 
    \quad
    \hat{c}_3 \coloneqq L^2_{\widehat{xz}}\|\Gamma D_{\widehat{xz}}\|, \\
    L_M^2 \coloneqq \max\big\{L_{xx}^2\|D_{xx}\Gamma_1\|+L_{yx}^2\|D_{yx}\Gamma_1\|, \|L_{xy}^2\|D_{xy}\Gamma_2\|+L_{yy}^2\|D_{yy}\Gamma_2\|\big\}.
    \end{gather*}
    Consider the sequence \(\seq{z^k}\) generated by \Cref{alg:NP-PDHG}.
    Then, the following holds for all $z^\star \in \mathcal S^\star$ 
    \begin{align*}
        \rbl{
        \mathbb E[\dist_\Gamma(0,T\z^{k_\star})^2]
        }
            {}\leq{}
        \frac{\mathbb E[\|z^{0}-z^{\star}\|_{\Gamma^{-1}}^{2}] + \eta \mathbb E[\|\Gamma^{-1}\hat z^{-1}-S_{z^{-1}}(z^{-1};\z^{-1})\|_{\Gamma}^{2}] + C \sigma_F^2\sum_{j=0}^K \alpha_{j}^{2}}{\mu \sum_{j=0}^K \alpha_{j}}
    \end{align*}
    where
    $C \coloneqq 2(\eta + \alpha_0 (\tfrac{1}{\sqrt{\alpha_0}(1-L_M)^2} + \tfrac {1-\sqrt{\alpha_0}}{\sqrt{\alpha_0}}))(1+2\hat{c}_2) + 1 + 2(\hat{c}_1 + 2\hat{c}_2(\Theta+\hat{c}_3))\eta$ 
    with 
    $\Theta = (1-\theta)^2 + 2\theta^2$ 
    and
    $k_{\star}$ is chosen from $\{0,1, \ldots, K\}$ according to probability $\mathcal{P}\left[k_{\star}=k\right]=\frac{\alpha_k}{\sum_{j=0}^K \alpha_j}$.
\end{thm}
\begin{remark}
When $\alpha_0 \rightarrow 0$ the conditions in \eqref{thm:AsymPrecon:PDHG:convergence} reduces to $1+\tfrac{2\rho}{\bar \gamma} > 0$ as in the deterministic case.
\rbl{Almost sure convergence is provided in \Cref{app:thm:AFBA:almostsure} of the appendix.}
\end{remark}
For $\theta=0$ \Cref{alg:NP-PDHG} reduces to \Cref{alg:WeakMinty:Sto:StructA}.
With this choice \Cref{thm:AsymPrecon:PDHG:convergence} simplifies, since the constant $\hat{c}_2=0$, and we recover the convergence result of \Cref{thm:const:convergence}.

\section{Experiments}
\label{sec:experiments}

\begin{figure}[t]
\vspace{-1.5em}
\centering
\includegraphics[width=0.5\textwidth]{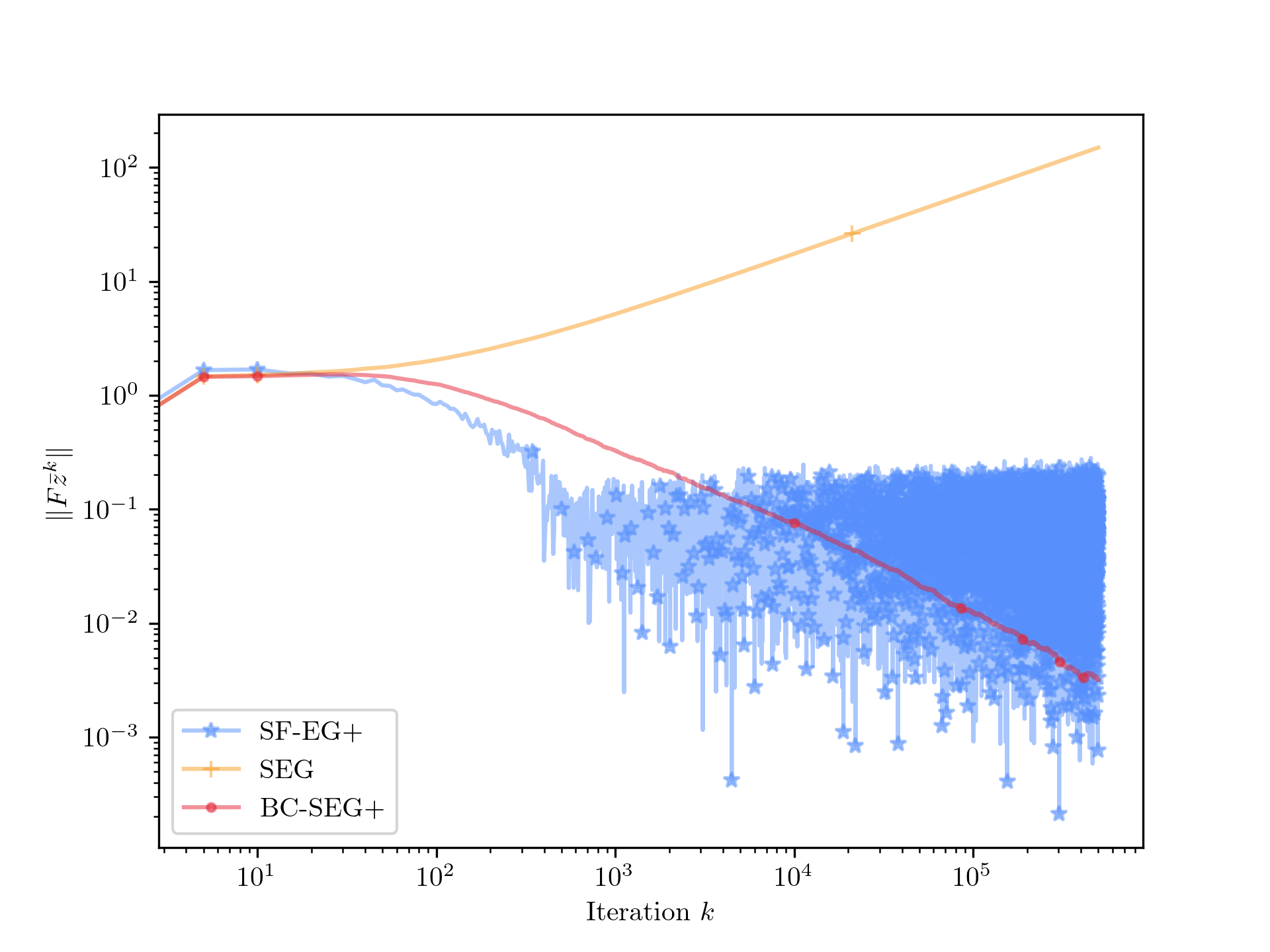}%
\includegraphics[width=0.5\textwidth]{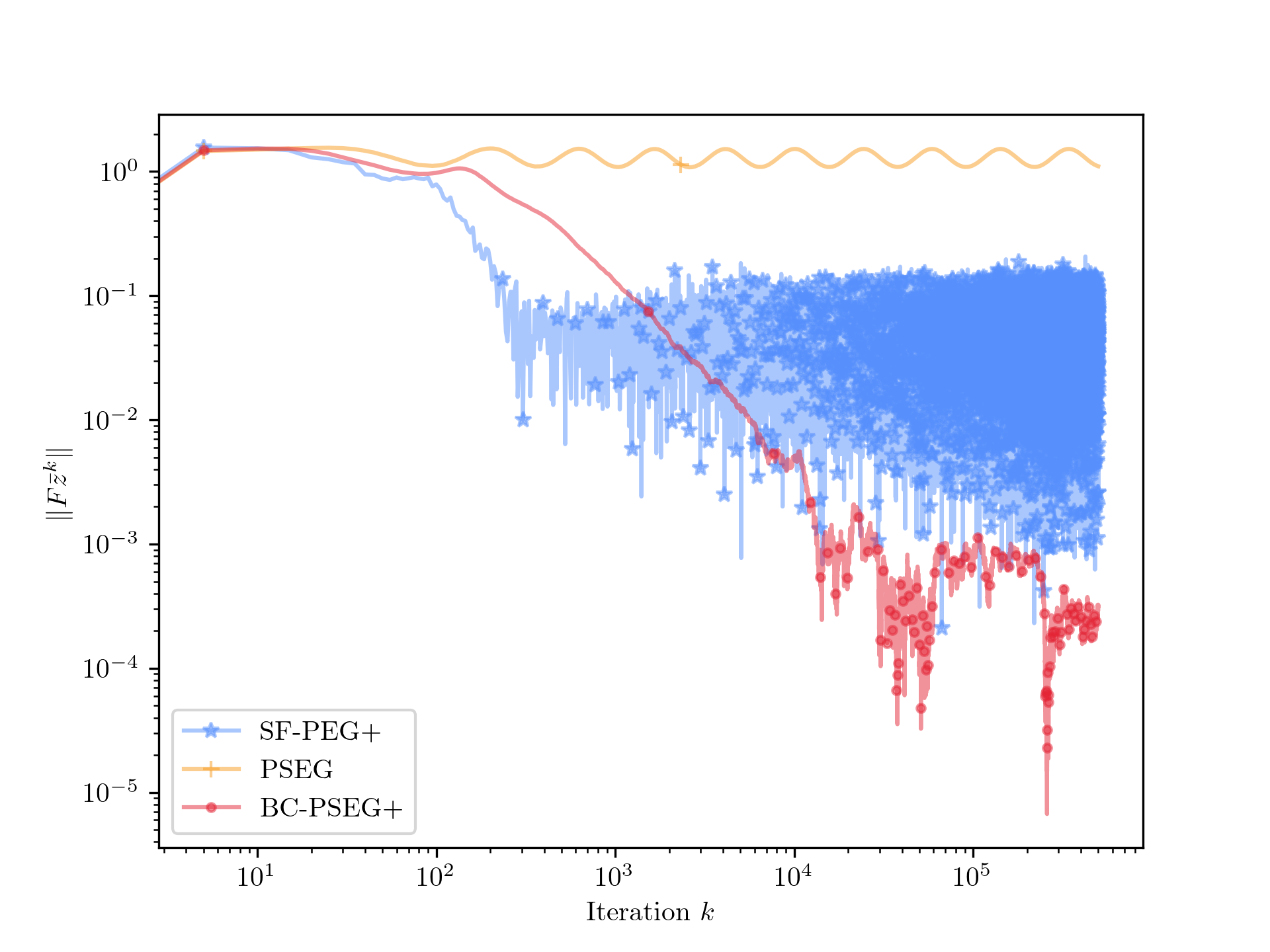}%
\caption{Comparison of methods in the unconstrained setting of \Cref{ex:quadratic} (left) and the constrained setting of \Cref{ex:globalforsaken} (right).
\rbl{
Notice that only BC-SEG+ and BC-PSEG+ converges properly while \eqref{eq:seg} diverges, \eqref{eq:PSEG} cycles and both \eqref{eq:SF-EG+} and \eqref{eq:SF-PEG+} only converge to a neighborhood.
BC-(P)SEG+ is guaranteed to converge with probability 1 as established through \Cref{thm:BiasCorr,app:thm:AFBA:almostsure}.
}
}
\label{fig:simulations}
\vspace{-1.5em}
\end{figure}

\rbl{We compare BC-SEG+ and BC-PSEG+ against \eqref{eq:eg+} using \emph{stochastic feedback} (which we refer to as \eqref{eq:SF-EG+}) and \eqref{eq:seg} in both an unconstrained setting and a constrained setting introduced in \citet{pethick2022escaping}.
See \Cref{app:algo} for the precise formulation of the projected variants which we denote \eqref{eq:SF-PEG+} and \eqref{eq:PSEG} respectively.}
In the unconstrained example we control all problem constant and set $\rho=-\nicefrac{1}{10L_F}$, while the constrained example is a specific minimax problem where $\rho > -\nicefrac{1}{2L_F}$ holds within the constrained set for a Lipschitz constant $L_F$ restricted to the same constrained set.
To simulate a stochastic setting in both examples, we consider additive Gaussian noise, i.e.
$\hat F(z,\xi) = Fz + \xi$ where $\xi \sim \mathcal N(0,\sigma^2 I)$.
In the experiments we choose $\sigma = 0.1$ \rbl{and $\alpha_k \propto \nicefrac{1}{k}$, which ensures almost sure convergence of BC-(P)SEG+}.
For a more aggressive stepsize choice $\alpha_k \propto \nicefrac{1}{\sqrt{k}}$ see \Cref{fig:aggressive-step}.
Further details can be found in \Cref{app:experiments}.%

The results are shown in \Cref{fig:simulations}.
The sequence generated by \eqref{eq:seg} and \eqref{eq:PSEG} diverges for the unconstrained problem and cycles in the constrained problem respectively.
In comparison \eqref{eq:SF-EG+} and \eqref{eq:SF-PEG+} gets within a neighborhood of the solutions but fails to converge due to the non-diminishing stepsize, while BC-SEG+ and BC-PSEG+ converges in the examples.

\section{Conclusion}
\label{sec:conclusion}

This paper shows that nonconvex-nonconcave problems characterize by the weak Minty variational inequality can be solved efficiently even when only \emph{stochastic} gradients are available.
The approach crucially avoids increasing batch sizes by instead introducing a bias-correction term.
We show that convergence is possible for the same range of problem constant $\rho \in (-\nicefrac{\gamma}{2},\infty)$ as in the deterministic case.
Rates are established for a \emph{random iterate}, which matches those of stochastic extragradient in the monotone case,
and the result is complemented with almost sure convergence, thus providing asymptotic convergence for the \emph{last iterate}.
We show that the idea extends to a family of extragradient-type methods which includes a nonlinear extension of the celebrated primal dual hybrid gradient (PDHG) algorithm.
For future work it is interesting to see if the rate can be improved by considering accelerated methods such as Halpern iterations.

\section{Acknowledgments and disclosure of funding}
\label{sec:acknowledgement}
This project has received funding from the European Research Council (ERC) under the European Union's Horizon 2020 research and innovation programme (grant agreement n° 725594 - time-data).
This work was supported by the Swiss National Science Foundation (SNSF) under  grant number 200021\_205011. %
The work of the third and fourth author was supported by the Research Foundation Flanders (FWO) postdoctoral grant 12Y7622N
and research projects G081222N, G033822N, G0A0920N; Research Council
KU Leuven C1 project No. C14/18/068; European Union's Horizon 2020
research and innovation programme under the Marie Skłodowska-Curie
grant agreement No. 953348. %
The work of Olivier Fercoq was supported by the Agence National de la Recherche grant ANR-20-CE40-0027, Optimal Primal-Dual Algorithms (APDO).

\bibliography{TeX/SWMVI_NeurIPS2022.bib,TeX/Bibliography.bib}

\begin{thebibliography}{34}
\providecommand{\natexlab}[1]{#1}
\providecommand{\url}[1]{\texttt{#1}}
\expandafter\ifx\csname urlstyle\endcsname\relax
  \providecommand{\doi}[1]{doi: #1}\else
  \providecommand{\doi}{doi: \begingroup \urlstyle{rm}\Url}\fi

\bibitem[Alacaoglu \& Malitsky(2021)Alacaoglu and
  Malitsky]{alacaoglu2021stochastic}
Ahmet Alacaoglu and Yura Malitsky.
\newblock Stochastic variance reduction for variational inequality methods.
\newblock \emph{arXiv preprint arXiv:2102.08352}, 2021.

\bibitem[Arjevani et~al.(2019)Arjevani, Carmon, Duchi, Foster, Srebro, and
  Woodworth]{arjevani2019lower}
Yossi Arjevani, Yair Carmon, John~C Duchi, Dylan~J Foster, Nathan Srebro, and
  Blake Woodworth.
\newblock Lower bounds for non-convex stochastic optimization.
\newblock \emph{arXiv preprint arXiv:1912.02365}, 2019.

\bibitem[Bauschke \& Combettes(2017)Bauschke and Combettes]{Bauschke2017Convex}
Heinz~H. Bauschke and Patrick~L. Combettes.
\newblock \emph{Convex analysis and monotone operator theory in {H}ilbert
  spaces}.
\newblock CMS Books in Mathematics. Springer, 2017.
\newblock ISBN 978-3-319-48310-8.

\bibitem[Bauschke et~al.(2021)Bauschke, Moursi, and
  Wang]{bauschke2021generalized}
Heinz~H Bauschke, Walaa~M Moursi, and Xianfu Wang.
\newblock Generalized monotone operators and their averaged resolvents.
\newblock \emph{Mathematical Programming}, 189\penalty0 (1):\penalty0 55--74,
  2021.

\bibitem[Bertsekas(2011)]{Bertsekas2011Incremental}
Dimitri~P. Bertsekas.
\newblock Incremental proximal methods for large scale convex optimization.
\newblock \emph{Mathematical programming}, 129\penalty0 (2):\penalty0 163--195,
  2011.

\bibitem[Beznosikov et~al.(2022)Beznosikov, Gorbunov, Berard, and
  Loizou]{beznosikov2022stochastic}
Aleksandr Beznosikov, Eduard Gorbunov, Hugo Berard, and Nicolas Loizou.
\newblock Stochastic gradient descent-ascent: Unified theory and new efficient
  methods.
\newblock \emph{arXiv preprint arXiv:2202.07262}, 2022.

\bibitem[B{\"o}hm(2022)]{bohm2022solving}
Axel B{\"o}hm.
\newblock Solving nonconvex-nonconcave min-max problems exhibiting weak minty
  solutions.
\newblock \emph{arXiv preprint arXiv:2201.12247}, 2022.

\bibitem[Bo{\c{t}} et~al.(2021)Bo{\c{t}}, Mertikopoulos, Staudigl, and
  Vuong]{boct2021minibatch}
Radu~Ioan Bo{\c{t}}, Panayotis Mertikopoulos, Mathias Staudigl, and Phan~Tu
  Vuong.
\newblock Minibatch forward-backward-forward methods for solving stochastic
  variational inequalities.
\newblock \emph{Stochastic Systems}, 11\penalty0 (2):\penalty0 112--139, 2021.

\bibitem[Cai et~al.(2022)Cai, Song, Guzm{\'a}n, and
  Diakonikolas]{cai2022stochastic}
Xufeng Cai, Chaobing Song, Crist{\'o}bal Guzm{\'a}n, and Jelena Diakonikolas.
\newblock A stochastic {Halpern} iteration with variance reduction for
  stochastic monotone inclusion problems.
\newblock \emph{arXiv preprint arXiv:2203.09436}, 2022.

\bibitem[Chambolle \& Pock(2011)Chambolle and Pock]{Chambolle2011firstorder}
A.~Chambolle and T.~Pock.
\newblock A first-order primal-dual algorithm for convex problems with
  applications to imaging.
\newblock \emph{Journal of Mathematical Imaging and Vision}, 40\penalty0
  (1):\penalty0 120--145, 2011.

\bibitem[Chen et~al.(2021)Chen, Sun, and Yin]{chen2021solving}
Tianyi Chen, Yuejiao Sun, and Wotao Yin.
\newblock Solving stochastic compositional optimization is nearly as easy as
  solving stochastic optimization.
\newblock \emph{IEEE Transactions on Signal Processing}, 69:\penalty0
  4937--4948, 2021.

\bibitem[Combettes \& Pennanen(2004)Combettes and
  Pennanen]{combettes2004proximal}
Patrick~L Combettes and Teemu Pennanen.
\newblock Proximal methods for cohypomonotone operators.
\newblock \emph{SIAM journal on control and optimization}, 43\penalty0
  (2):\penalty0 731--742, 2004.

\bibitem[Daskalakis et~al.(2021)Daskalakis, Skoulakis, and
  Zampetakis]{daskalakis2021complexity}
Constantinos Daskalakis, Stratis Skoulakis, and Manolis Zampetakis.
\newblock The complexity of constrained min-max optimization.
\newblock In \emph{Proceedings of the 53rd Annual ACM SIGACT Symposium on
  Theory of Computing}, pp.\  1466--1478, 2021.

\bibitem[Diakonikolas et~al.(2021)Diakonikolas, Daskalakis, and
  Jordan]{diakonikolas2021efficient}
Jelena Diakonikolas, Constantinos Daskalakis, and Michael Jordan.
\newblock Efficient methods for structured nonconvex-nonconcave min-max
  optimization.
\newblock In \emph{International Conference on Artificial Intelligence and
  Statistics}, pp.\  2746--2754. PMLR, 2021.

\bibitem[Fang et~al.(2018)Fang, Li, Lin, and Zhang]{fang2018spider}
Cong Fang, Chris~Junchi Li, Zhouchen Lin, and Tong Zhang.
\newblock Spider: Near-optimal non-convex optimization via stochastic
  path-integrated differential estimator.
\newblock \emph{Advances in Neural Information Processing Systems}, 31, 2018.

\bibitem[Ghadimi \& Lan(2013)Ghadimi and Lan]{ghadimi2013stochastic}
Saeed Ghadimi and Guanghui Lan.
\newblock Stochastic first-and zeroth-order methods for nonconvex stochastic
  programming.
\newblock \emph{SIAM Journal on Optimization}, 23\penalty0 (4):\penalty0
  2341--2368, 2013.

\bibitem[Gorbunov et~al.(2022)Gorbunov, Berard, Gidel, and
  Loizou]{gorbunov2022stochastic}
Eduard Gorbunov, Hugo Berard, Gauthier Gidel, and Nicolas Loizou.
\newblock Stochastic extragradient: General analysis and improved rates.
\newblock In \emph{International Conference on Artificial Intelligence and
  Statistics}, pp.\  7865--7901. PMLR, 2022.

\bibitem[Hirsch \& Vavasis(1987)Hirsch and Vavasis]{hirsch1987exponential}
M~Hirsch and S~Vavasis.
\newblock Exponential lower bounds for finding {B}rouwer fixed points.
\newblock In \emph{Proceedings of the 28th Symposium on Foundations of Computer
  Science}, pp.\  401--410, 1987.

\bibitem[Hsieh et~al.(2019)Hsieh, Iutzeler, Malick, and
  Mertikopoulos]{hsieh2019convergence}
Yu-Guan Hsieh, Franck Iutzeler, J{\'e}r{\^o}me Malick, and Panayotis
  Mertikopoulos.
\newblock On the convergence of single-call stochastic extra-gradient methods.
\newblock \emph{Advances in Neural Information Processing Systems}, 32, 2019.

\bibitem[Hsieh et~al.(2020)Hsieh, Iutzeler, Malick, and
  Mertikopoulos]{hsieh2020explore}
Yu-Guan Hsieh, Franck Iutzeler, J{\'e}r{\^o}me Malick, and Panayotis
  Mertikopoulos.
\newblock Explore aggressively, update conservatively: Stochastic extragradient
  methods with variable stepsize scaling.
\newblock \emph{arXiv preprint arXiv:2003.10162}, 2020.

\bibitem[Juditsky et~al.(2011)Juditsky, Nemirovski, and
  Tauvel]{juditsky2011solving}
Anatoli Juditsky, Arkadi Nemirovski, and Claire Tauvel.
\newblock Solving variational inequalities with stochastic mirror-prox
  algorithm.
\newblock \emph{Stochastic Systems}, 1\penalty0 (1):\penalty0 17--58, 2011.

\bibitem[Kannan \& Shanbhag(2019)Kannan and Shanbhag]{kannan2019optimal}
Aswin Kannan and Uday~V Shanbhag.
\newblock Optimal stochastic extragradient schemes for pseudomonotone
  stochastic variational inequality problems and their variants.
\newblock \emph{Computational Optimization and Applications}, 74\penalty0
  (3):\penalty0 779--820, 2019.

\bibitem[Latafat \& Patrinos(2017)Latafat and Patrinos]{Latafat2017Asymmetric}
Puya Latafat and Panagiotis Patrinos.
\newblock Asymmetric forward--backward--adjoint splitting for solving monotone
  inclusions involving three operators.
\newblock \emph{Computational Optimization and Applications}, 68\penalty0
  (1):\penalty0 57--93, Sep 2017.

\bibitem[Lee \& Kim(2021)Lee and Kim]{lee2021fast}
Sucheol Lee and Donghwan Kim.
\newblock Fast extra gradient methods for smooth structured
  nonconvex-nonconcave minimax problems.
\newblock \emph{arXiv preprint arXiv:2106.02326}, 2021.

\bibitem[Li et~al.(2022)Li, Yu, Loizou, Gidel, Ma, Le~Roux, and
  Jordan]{li2022convergence}
Chris~Junchi Li, Yaodong Yu, Nicolas Loizou, Gauthier Gidel, Yi~Ma, Nicolas
  Le~Roux, and Michael Jordan.
\newblock On the convergence of stochastic extragradient for bilinear games
  using restarted iteration averaging.
\newblock In \emph{International Conference on Artificial Intelligence and
  Statistics}, pp.\  9793--9826. PMLR, 2022.

\bibitem[Loizou et~al.(2020)Loizou, Berard, Jolicoeur-Martineau, Vincent,
  Lacoste-Julien, and Mitliagkas]{loizou2020stochastic}
Nicolas Loizou, Hugo Berard, Alexia Jolicoeur-Martineau, Pascal Vincent, Simon
  Lacoste-Julien, and Ioannis Mitliagkas.
\newblock Stochastic hamiltonian gradient methods for smooth games.
\newblock In \emph{International Conference on Machine Learning}, pp.\
  6370--6381. PMLR, 2020.

\bibitem[Loizou et~al.(2021)Loizou, Berard, Gidel, Mitliagkas, and
  Lacoste-Julien]{loizou2021stochastic}
Nicolas Loizou, Hugo Berard, Gauthier Gidel, Ioannis Mitliagkas, and Simon
  Lacoste-Julien.
\newblock Stochastic gradient descent-ascent and consensus optimization for
  smooth games: Convergence analysis under expected co-coercivity.
\newblock \emph{Advances in Neural Information Processing Systems},
  34:\penalty0 19095--19108, 2021.

\bibitem[Mishchenko et~al.(2020)Mishchenko, Kovalev, Shulgin, Richt{\'a}rik,
  and Malitsky]{mishchenko2020revisiting}
Konstantin Mishchenko, Dmitry Kovalev, Egor Shulgin, Peter Richt{\'a}rik, and
  Yura Malitsky.
\newblock Revisiting stochastic extragradient.
\newblock In \emph{International Conference on Artificial Intelligence and
  Statistics}, pp.\  4573--4582. PMLR, 2020.

\bibitem[Nguyen et~al.(2019)Nguyen, van Dijk, Phan, Nguyen, Weng, and
  Kalagnanam]{nguyen2019finite}
Lam~M Nguyen, Marten van Dijk, Dzung~T Phan, Phuong~Ha Nguyen, Tsui-Wei Weng,
  and Jayant~R Kalagnanam.
\newblock Finite-sum smooth optimization with {SARAH}.
\newblock \emph{arXiv preprint arXiv:1901.07648}, 2019.

\bibitem[Pethick et~al.(2022)Pethick, Latafat, Patrinos, Fercoq, and
  Cevher]{pethick2022escaping}
Thomas Pethick, Puya Latafat, Panagiotis Patrinos, Olivier Fercoq, and Volkan
  Cevher.
\newblock Escaping limit cycles: Global convergence for constrained
  nonconvex-nonconcave minimax problems.
\newblock In \emph{International Conference on Learning Representations}, 2022.

\bibitem[Rockafellar(1970)]{Rockafellar1970Convex}
Ralph~Tyrell Rockafellar.
\newblock \emph{Convex analysis}.
\newblock Princeton University Press, 1970.

\bibitem[Song et~al.(2020)Song, Zhou, Zhou, Jiang, and Ma]{song2020optimistic}
Chaobing Song, Zhengyuan Zhou, Yichao Zhou, Yong Jiang, and Yi~Ma.
\newblock Optimistic dual extrapolation for coherent non-monotone variational
  inequalities.
\newblock \emph{Advances in Neural Information Processing Systems},
  33:\penalty0 14303--14314, 2020.

\bibitem[Tseng(2000)]{Tseng2000modified}
P.~Tseng.
\newblock A modified forward-backward splitting method for maximal monotone
  mappings.
\newblock \emph{SIAM Journal on Control and Optimization}, 38\penalty0
  (2):\penalty0 431--446, 2000.

\bibitem[Yang et~al.(2020)Yang, Kiyavash, and He]{yang2020global}
Junchi Yang, Negar Kiyavash, and Niao He.
\newblock Global convergence and variance-reduced optimization for a class of
  nonconvex-nonconcave minimax problems.
\newblock \emph{arXiv preprint arXiv:2002.09621}, 2020.

\end{thebibliography}
\bibliographystyle{iclr2023_conference}

\newpage
\appendix

\begin{center}
\vspace{7pt}
{\Large \fontseries{bx}\selectfont Appendix}
\end{center}

\renewcommand{\contentsname}{Table of Contents}
\etocdepthtag.toc{mtappendix}
\etocsettagdepth{mtchapter}{none}
\etocsettagdepth{mtappendix}{subsection}
\tableofcontents

\newpage

\section{Prelude}
    
For the unconstrained and smooth setting
\Cref{app:seg+} treats convergences of \eqref{eq:seg+} for the restricted case where $F$ is linear.
\Cref{app:smooth} shows both random iterate results and almost sure convergence of \Cref{alg:WeakMinty:Sto:Struct}.
\Cref{thm:BiasCorr:2,thm:BiasCorr} in the main body are implied by the more general results in this section, which preserves certain free parameters and more general stepsize requirements.
\Cref{app:const,app:AFBA} moves beyond the unconstrained and smooth case by showing convergence for instances of the template scheme \eqref{eq:AFBA}.
Almost sure convergence is established in \Cref{app:thm:AFBA:almostsure}.
The analysis of \Cref{alg:NP-PDHG} in \Cref{app:AFBA} applies to \Cref{alg:WeakMinty:Sto:StructA}, but for completeness we establish convergence for general $F$ separately in \Cref{app:const}. %
The relationship between the theorems are presented in \Cref{app:tbl:overview}.

\begin{table}
\centering
\caption{Overview of the results. The second row is obtained as special cases of the first row.}
\label{app:tbl:overview}
\begin{tabular}{l|c|c|c|c}
\hline
\multicolumn{1}{r}{} & \multicolumn{2}{|c|}{Unconstrained \& smooth ($A\equiv 0$)}             & \multicolumn{2}{c}{Constrained ($A\not\equiv 0$)}                                                      \\
\hline
                     &  \multicolumn{1}{c|}{Random iterate} &  \multicolumn{1}{c|}{Last iterate} &  \multicolumn{1}{c|}{BC-PSEG+}                                  &  \multicolumn{1}{c}{NP-PDHG} \\
\hline
Appendix             & \Cref{app:thm:BiasCorr:2}          & \Cref{app:thm:BiasCorr}          & \Cref{app:thm:const:convergence}                                   & \Cref{app:thm:AsymPrecon:convergence} \\
             & $\Downarrow$ & $\Downarrow$ & $\Downarrow$ & $\Downarrow$ \\
Main paper           & \Cref{thm:BiasCorr:2}              & \Cref{thm:BiasCorr}              & \Cref{thm:const:convergence} &                                 \Cref{thm:AsymPrecon:PDHG:convergence}
\end{tabular}
\end{table}

\rbl{\section{Preliminaries}\label{app:preliminaries}
    Given a psd matrix $V$ we define the inner product as $\langle\cdot,\cdot\rangle_V \coloneqq \langle\cdot,V\cdot\rangle$ and the corresponding norm $\|\cdot\| \coloneqq \sqrt{\langle\cdot,\cdot\rangle_V}$.
The distance from $u \in \R^n$ to a set $\mathcal U \subseteq \R^n$ with respect to a positive definite matrix $V$ is defined as $\dist_V(u,\mathcal U) \coloneqq \min_{u' \in \mathcal U} \|u-u'\|_V$, which we simply denote $\dist(u,\mathcal U)$ when $V=I$.
The norm $\|X\|$ refers to spectral norm when $X$ is a matrix.

We summarize essential definitions from operator theory, but otherwise refer to
\cite{Bauschke2017Convex,Rockafellar1970Convex} for further details.

An operator $A:\R^n\rightrightarrows\R^d$ maps each point $x\in\R^n$ to a subset $Ax \subseteq \R^d$, where the notation $A(x)$ and $Ax$ will be used interchangably. 
We denote the domain of $A$ by 
$\dom A\coloneqq\{x\in\R^n\mid Ax\neq\emptyset\},$
its graph by 
$\graph A\coloneqq\{(x,y)\in\R^n\times \R^d\mid y\in Ax\}.$
The inverse of $A$ is defined through its graph, $\graph A^{-1}\coloneqq\{(y,x)\mid (x,y)\in\graph A\}$
and the set of its zeros by $\zer A\coloneqq\{x\in\R^n \mid 0\in Ax\}$. 

\begin{defin}[(co)monotonicity \cite{bauschke2021generalized}] 
An operator $A:\R^n\rightrightarrows\R^n$ is said to be $\rho$-monotone for some $\rho\in \R$, if for all $(x,y),(x^\prime,y^\prime)\in\graph A$
 \[
 \langle y-y^\prime, x-x^\prime\rangle \geq \rho\|x-x^\prime\|^2,
 \]
 and it is said to be $\rho$-comonotone if for all $(x,y),(x^\prime,y^\prime)\in\graph A$
\[
 \langle y-y^\prime, x-x^\prime\rangle \geq \rho\|y-y^\prime\|^2.
 \]
The operator $A$ is said to be maximally (co)monotone if there exists no other (co)monotone operator $B$ for which $\graph A \subset \graph B$ properly.
\end{defin}
If $A$ is $0$-monotone we simply say it is monotone.
When $\rho<0$, $\rho$-comonotonicity is also referred to as $|\rho|$-cohypomonotonicity.
\begin{defin}[Lipschitz continuity and cocoercivity]
    Let $\mathcal D\subseteq \R^n$ be a nonempty subset of $\R^n$. A single-valued operator $A:\mathcal D\to \R^n$ is said to be $L$-Lipschitz continuous if for any $x,x^\prime\in \mathcal D$
	\[
	\|Ax - Ax^\prime\| \leq L\|x-x^\prime\|,
	\]
	and $\beta$-cocoercive if 
	\[
		\langle x-x^\prime, Ax - Ax^\prime \rangle \geq \beta\|Ax - Ax^\prime\|^2. 
	\]
	Moreover, $A$ is said to be nonexpansive if it is $1$-Lipschitz continuous, and firmly nonexpansive if 
	it is $1$-cocoercive. 
\end{defin}
A $\beta$-cocoercive operator is also $\beta^{-1}$-Lipschitz continuity by direct implication of Cauchy-Schwarz. 
The resolvent operator  $J_A=(\id + A)^{-1}$ is firmly nonexpansive (with $\dom J_A= \R^n$) if and only if $A$ is (maximally) monotone.

We will make heavy use of the Fenchel-Young inequality.
For all $a,b \in \R^n$ and $e > 0$ we have,

\begin{align}
2\langle a,b\rangle &\leq e \|a\|^2 + \tfrac{1}{e} \|b\|^2 \\
\|a+b\|^2 &\leq (1+e) \|a\|^2 + (1+\tfrac{1}{e}) \|b\|^2 \\
-\|a-b\|^2 &\leq -\tfrac{1}{1+e} \|a\|^2 + \tfrac{1}{e} \|b\|^2
\end{align}
}
\section{Proof for SEG+}\label{app:seg+}
    
\begin{appendixproof}{thm:SEG+}
Following \citep{hsieh2020explore} closely, define the reference state $\tilde{\bar{z}}^k :=z^k - \gamma Fz^k$ to be the exploration step using the \emph{deterministic} operator and denote the second stepsize as $\eta_k := \alpha_k \gamma$.
We will let $\zeta$ denote the additive noise term, i.e. $\hat F(z,\xi) := F(z)+\zeta$.
Expanding the distance to solution,
\begin{equation}
\label{eq:proof:seg+:main}
\begin{aligned}
\|z^{k+1}-z^{\star}\|^{2} &=\|z^{k}-\eta_{k} \hat F(\bar z^k,\bar\xi_k)-z^{\star}\|^{2} \\
&=\|z^{k}-z^{\star}\|^{2}-2 \eta_{k}\langle\hat F(\bar z^k,\bar\xi_k), z^{k}-z^{\star}\rangle+\eta_{k}^{2}\|\hat F(\bar z^k,\bar\xi_k)\|^{2} \\
&=
\|z^{k}-z^{\star}\|^{2}
-2 \eta_{k}\langle\hat F(\bar z^k,\bar\xi_k), \tilde{\bar{z}}^k-z^{\star}\rangle
-2 \gamma \eta_{k}\langle\hat F(\bar z^k,\bar\xi_k), F(z^{k})\rangle
+\eta_{k}^{2}\|\hat F(\bar z^k,\bar\xi_k)\|^{2}.
\end{aligned}
\end{equation}
Recall that the operator is assumed to be linear $Fz = Bz+v$ in which case we have,
\begin{equation}
\label{eq:proof:seg+:Fbar}
\begin{aligned}
\hat F(\bar z^k,\bar\xi_k) &= B\bar z^k + v + \bar{\zeta}_k\\
= & B(z^k - \gamma \hat F(z^k,\xi_k)) + v + \bar{\zeta}_k \\
= & B(z^k - \gamma Bz^k - \gamma v  - \gamma \zeta_k) + v + \bar{\zeta}_k \\
= & B(z^k - \gamma (Bz^k + v)) + v - \gamma B\zeta_k + \bar{\zeta}_k \\
= & F(\tilde{\bar z}^k) - \gamma B\zeta_k + \bar{\zeta}_k.
\end{aligned}
\end{equation}
The two latter terms are zero in expectation due to the unbiasedness from \Cref{ass:AsymPrecon:stoch:unbiased}, which lets us write the terms on the RHS of \eqref{eq:proof:seg+:main} as,
\begin{align}
-\mathbb E_k\langle\hat F(\bar z^k,\bar\xi_k), \tilde{\bar{z}}^k-z^{\star}\rangle 
         &= -\langle F(\tilde{\bar z}^k), \tilde{\bar{z}}^k-z^{\star}\rangle\label{eq:proof:seg+:eq1} \\
-\mathbb E_k \langle\hat F(\bar z^k,\bar\xi_k), F(z^{k})\rangle 
         &= -\langle F(\tilde{\bar z}^k), F(z^{k})\rangle \label{eq:proof:seg+:eq2} \\
\mathbb E_k \|\hat F(\bar z^k,\bar\xi_k)\|^{2}
         &= \|F(\tilde{\bar z}^k)\|^2 + \mathbb E_k\|\gamma B\zeta_k\|^2 + \mathbb E_k\|\bar{\zeta}_k\|^2. \label{eq:proof:seg+:eq3}
\end{align}
We can bound \eqref{eq:proof:seg+:eq1} directly through the weak MVI in \Cref{ass:AsymPrecon:Minty:Struct} which might still be positive,
\begin{equation}
-\langle F(\tilde{\bar{z}}^k), \tilde{\bar{z}}^k-z^{\star}\rangle
\leq -\rho \|F(\tilde{\bar{z}}^k)\|^2.
\end{equation}
For the latter two terms of \eqref{eq:proof:seg+:eq3} we have
\begin{equation}
\mathbb E_k\|\gamma B\zeta_k\|^2 + \mathbb E_k\|\bar{\zeta}_k\|^2 
  = \gamma^2 \mathbb E_k\|F(\zeta_k)\rbl{-F(0)}\|^2 + \mathbb E_k\|\bar{\zeta}_k\|^2
  \leq (\gamma^2L_F^2 + 1) \sigma^2_F,
\end{equation}
where the last inequality follows from Lipschitz in \Cref{ass:AsymPrecon:M:Lip} and bounded variance in \Cref{ass:AsymPrecon:stoch:boundedvar}.

Combining everything into \eqref{eq:proof:seg+:main} we are left with
\begin{equation}
\label{eq:proof:seg+:main2}
\begin{aligned}
\mathbb E_k \|z^{k+1}-z^{\star}\|^{2} &\leq
\|z^{k}-z^{\star}\|^{2}
+ \eta_k^2(\gamma^2L_F^2 + 1) \sigma^2_F
-2\gamma\eta_k\langle F(\tilde{\bar z}^k), F(z^{k})\rangle
+(\eta_k^2 - 2\eta_k \rho) \|F(\tilde{\bar z}^k)\|^2
\end{aligned}
\end{equation}
By assuming the stepsize condition, $\rho \geq (\eta_k - \gamma)/2$, we have $\eta_k^2 - 2\eta_k \rho \leq \gamma \eta_k$.
This allows us to complete the square,
\begin{equation}
\begin{aligned}
-2\gamma\eta_k\langle F(\tilde{\bar z}^k), F(z^{k})\rangle
+(\eta_k^2 - 2\eta_k \rho) \|F(\tilde{\bar z}^k)\|^2
&\leq -2\gamma\eta_k\langle F(\tilde{\bar z}^k), F(z^{k})\rangle
+\gamma \eta_k \|F(\tilde{\bar z}^k)\|^2 \\
&= \gamma\eta_k (\|F(z^k)-F(\tilde{\bar z}^k)\|^2 - \|F(z^k)\|^2) \\
&\leq \gamma\eta_k (\gamma^2L_F^2 - 1) \|F(z^k)\|^2,
\end{aligned}
\end{equation}
where the last inequality follows from Lipschitzness of $F$ and the definition of the update rule.
Plugging into \eqref{eq:proof:seg+:main2} we are left with
\begin{equation}
\label{eq:proof:seg+:main3}
\begin{aligned}
\mathbb E_k \|z^{k+1}-z^{\star}\|^{2} &\leq
\|z^{k}-z^{\star}\|^{2}
+ \eta_k^2(\gamma^2L_F^2 + 1) \sigma^2_F
- \gamma\eta_k (1-\gamma^2L_F^2) \|F(z^k)\|^2.
\end{aligned}
\end{equation}
The result is obtained by total expectation and summing.
\end{appendixproof}

\section{Proof for smooth unconstrained case}\label{app:smooth}
    
\begin{lem}\label{lem:recur}
    Consider the recurrent relation \(B_{k+1} = \xi_k B_k + d_k\) such that \(\xi_k >0\) for all \(k\geq 0\). Then 
    \[
        B_{k+1}
            {}={}
        \big(%
            \Pi_{p=0}^{k}\xi_{p}
        \big)
        \left(
            B_{0}+\sum_{\ell=0}^{k}\frac{d_{\ell}}{\Pi_{p=0}^{\ell}\xi_{p}}
        \right). 
    \]
\end{lem}

\begin{ass}\label{ass:stepsizes:Ae0:2} 
\(\gamma\in (\lfloor-2\rho\rfloor_+, \nicefrac1{L_F})\) and 
\rbl{for positive real valued $b$,}
\begin{equation}\label{eq:mu}
    \mu \coloneqq \gamma^2(1 - \gamma^2L_F^2(1+{b}^{-1})) > 0.
\end{equation}
\end{ass}

\begin{thm}\label{app:thm:BiasCorr:2}
     Suppose that \cref{ass:AsymPrecon,ass:AsymPrecon:stoch,ass:AsymPrecon:stoch:stocLips} hold. Suppose in addition that \cref{ass:stepsizes:Ae0:2} holds and that \(\seq{\alpha_k}\subset(0,1)\) is a diminishing sequence such that 
     \begin{align}\label{app:eq:condest:Ae0}
         2\gamma L_{\hat{F}}\sqrt{\alpha_{0}}+\Big(1+\big((b+1)\gamma^{2}L_{F}^{2}\big)\gamma^{2}L_{\hat{F}}^{2}\Big)\alpha_{0}\leq1+\tfrac{2\rho}{\gamma}.
     \end{align}
     Consider the sequence \(\seq{z^k}\) generated by \Cref{alg:WeakMinty:Sto:Struct}.
     Then, the following estimate holds
\begin{align}
\label{app:thm:BiasCorr:2:rate}
\sum_{k=0}^K \frac{\alpha_k}{\sum_{j=0}^K \alpha_j}\mathbb E[\|F(z^k)\|^2] 
    {}\leq{}
\frac{%
    \|z^{0} - z^\star\|^2 + \eta\gamma^2\|F(z^{0})\|^2 
        {}+{}
    C \sigma_F^2\gamma^2 \sum_{j=0}^K \alpha_j^2
     }{%
        \mu \sum_{j=0}^K \alpha_j
    },
\end{align}
where \(C = 1+2\eta\big((\gamma^{2}L_{\hat{F}}^{2}+1)+2\alpha_{0}\big)\)
and $\eta=\tfrac{1}{2}(b+1)\gamma^{2}L_{F}^{2}+\tfrac{1}{\gamma L_{\hat{F}}\sqrt{\alpha_{0}}}$. 
 \end{thm} 

\begin{appendixproof}{app:thm:BiasCorr:2}
The proof relies on establishing a (stochastic) descent property on the following potential function 
\begin{align*}
    \U_{k+1} 
        {}\coloneqq{}
    \|z^{k+1}-z^{\star}\|^{2}+A_{k+1}\|u^{k}\|^{2}+B_{k+1}\|z^{k+1}-z^{k}\|^{2}. %
\end{align*}
where \(u^k \coloneqq \bar z^{k} - z^k + \gamma F(z^k)\) \rbl{measures the difference of the bias-corrected step from the deterministic exploration step}, and \(\seq{A_{k}}\), \(\seq{B_k}\) are positive scalar parameters to be identified. 
We proceed to consider each term individually.

Let us begin by quantifying how well $\bar z^k$ estimates $z^k - \gamma F(z^k)$. 
\begin{align*}
u^k = \bar z^{k} - z^k& + \gamma F(z^k) = \gamma F(z^k) - \gamma \hat F(z^k, \xi_k) + (1-\betaparam_k) (\bar z^{k-1} - z^{k-1} + \gamma \hat F(z^{k-1}, \xi_k)).
\end{align*}
Therefore,
\begin{align*}
    \|u^k \|^2 
        {}={}&
    \|\gamma F(z^k) - \gamma \hat F(z^k, \xi_k)  + (1-\betaparam_k)(\gamma \hat F(z^{k-1}, \xi_k) - \gamma F(z^{k-1}))\|^2
    + (1-\betaparam_k)^2 \|u^{k-1}\|^2
    \\
    & + 2 (1-\betaparam_k) \langle \bar z^{k-1} - z^{k-1} + \gamma F(z^{k-1}), \gamma F(z^k) - \gamma \hat F(z^k, \xi_k)  + (1-\betaparam_k)(\gamma \hat F(z^{k-1}, \xi_k) - \gamma F(z^{k-1}))\rangle.
\end{align*}
Conditioned on \(\F_k\), in the inner product the left term is known and the right term has an expectation that equals zero. Therefore, we obtain
{\mathtight
\begin{align*}
\mathbb E[\|u^{k}\|^2 \mid \F_k] 
    {}={}&
\mathbb E[ \|(1-\betaparam_{k})\left(\gamma F(z^{k})-\gamma\hat{F}(z^{k},\xi_{k})+\gamma\hat{F}(z^{k-1},\xi_{k})-\gamma F(z^{k-1})\right)+\betaparam_{k}\left(\gamma F(z^{k})-\gamma\hat{F}(z^{k},\xi_{k})\right)\|^{2} \mid \F_k]
\\
    &{}+{}
(1-\betaparam_k)^2 \|u^{k-1}\|^2 
\\
    {}\leq{}&
(1-\betaparam_k)^2 \|u^{k-1} \|^2 
    {}+{}
2(1-\betaparam_k)^2 \gamma^2\mathbb E[\| \hat F(z^k, \xi_k) -\hat F(z^{k-1}, \xi_k) \|^2 \mid \F_k] 
\\
    &{}+{} 
2\betaparam_{k}^2 \gamma^2 \mathbb E[\|F(z^k) - \hat F(z^k, \xi_k) \|^2  \mid \F_k]
\\
    {}\leq{}&
(1-\betaparam_k)^2 \|u^{k-1}\|^2 + 2(1-\betaparam_k)^2 \gamma^2 L_{\hat F}^2 \|z^k - z^{k-1}\|^2 + 2\betaparam_k^2 \gamma^2 \sigma_F^2
 \numberthis \label{eq:normuk}
\end{align*}}%
where in the first inequality we used Young inequality and the fact that the second moment is larger than the variance, and \cref{ass:AsymPrecon:stoch:boundedvar,ass:AsymPrecon:stoch:stocLips} were used in the second inequality.

 By \cref{state:z+:sto:Smooth}, the equality 
\begin{align}
\label{eq:from_kp1_to_k}
\|z^{k+1} - z^\star\|^2 = \|z^k - z^\star\|^2 - 2 \alpha_k \gamma \langle \hat F(\bar z^k, \bar \xi_k),  z^k - z^\star\rangle + \alpha_k^2\gamma^2 \| \hat F(\bar z^k, \bar \xi_k)\|^2,
\end{align}
holds. 
The inner product in \eqref{eq:from_kp1_to_k} can be upper bounded using Young inequalities with positive parameters $\varepsilon_k$, \(k \geq0\), and $b$ as follows. 
\begin{align*}
    \mathbb E[\langle  - \gamma \hat F(\bar z^k, \bar \xi_k),  z^k - z^\star\rangle \mid \bar{\F}_k] 
        {}={}&
    -\gamma \langle  F(\bar z^k), z^k - \bar z^k\rangle -\gamma \langle  F(\bar z^k), \bar z^k - z^\star\rangle  
    \\
        {}={}&
    - \gamma^2 \langle  F(\bar z^k),  F(z^k) \rangle 
    + \gamma\langle F(\bar z^k), \bar z^k - z^k + \gamma F(z^k)\rangle
    -\gamma \langle  F(\bar z^k), \bar z^k - z^\star\rangle  
    \\
        {}\leq{}&
    \gamma^2 \Big(\frac{1}{2} \|F(\bar z^k) - F(z^k)\|^2 - \frac{1}{2}\|F(\bar z^k)\|^2 - \frac{1}{2}\|F(z^k)\|^2\Big) + \frac{\gamma^2 \varepsilon_k}{2} \|F(\bar z^k)\|^2
    \\
        &{}+{} \frac{1}{2\varepsilon_k}\|\bar z^k - z^k + \gamma F(z^k)\|^2 - \gamma \rho \|F(\bar z^k)\|^2 
    \\
        {}\leq{}&
    \gamma^2 L_F^2 \frac{1+{b}}{2} \|u^k\|^2 + \frac{1+{b}^{-1}}{2}\gamma^4 L_F^2 \|F(z^k)\|^2- \frac{\gamma^2}{2}\|F(\bar z^k)\|^2
    \\
        &{}-{}
    \frac{\gamma^2}{2}\|F(z^k)\|^2
        {}+{}
    \frac{\gamma^2 \varepsilon_k}{2} \|F(\bar z^k)\|^2 + \frac{1}{2\varepsilon_k}\|u^k \|^2 - \gamma \rho \|F(\bar z^k)\|^2 
    \\
        {}={}& \big(\gamma^2 L_F^2 \frac{1+{b}}{2} + \frac{1}{2\varepsilon_k}\big)\|u^k\|^2 + \frac{\gamma^2(\gamma^2L_F^2(1+{b}^{-1})- 1)}{2} \|F(z^k)\|^2 
    \\
        &{}+{}
     \big(\frac{\gamma^2({\varepsilon_k}-1)}{2} - \gamma \rho\big)\|F(\bar z^k)\|^2. 
     \numberthis \label{eq:innerprod:1}
\end{align*}
Conditioning \eqref{eq:innerprod:1} with 
\(
    \mathbb E\big[ \cdot \mid {\F}_k\big] 
        {}={}
    \mathbb E\big[ \mathbb E\big[ \cdot \mid \bar{\F}_k\big] \mid {\F}_k\big] 
\), since \(\F_k \subset \bar{\F}_{k}\), yields 
\begin{align*}
    2\mathbb E[\langle  - \gamma \hat F(\bar z^k, \bar \xi_k),  z^k - z^\star\rangle \mid {\F}_k] 
        {}\leq{}& 
    \big(\gamma^2 L_F^2(1+b) + \frac{1}{\varepsilon_k}\big)\mathbb E[\|u^k\|^2 \mid \F_k ] {}-{}
    \mu \|F(z^k)\|^2
    \\
        &{}+{}
     \big(\gamma^2({\varepsilon_k}-1) - 2\gamma \rho\big)\E[]{\|F(\bar z^k)\|^2\mid \F_k},
     \numberthis \label{eq:innerprod:2}
\end{align*}
where \(\mu\) was defined in \eqref{eq:mu}. 

\rbl{
The condition expectation of the third term in \eqref{eq:from_kp1_to_k} is bounded through \cref{ass:AsymPrecon:stoch:boundedvar} by
\begin{align*}
\mathbb E\left[  \|\hat F(\bar z^k, \bar \xi_k)\|^2\mid  \F_k\right] 
= \mathbb E\left[ \mathbb E[ \|\hat F(\bar z^k, \bar \xi_k)\|^2\mid  \bar{\F}_k] \mid  \F_k\right] 
\leq \|F(\bar z^k)\|^2 + \sigma_F^2,
\end{align*}
}
which in turn implies 
\begin{align*}
    \E[]{\|z^{k+1}-z^{k}\|^{2}\mid\F_{k}}
        {}={}&
    \alpha_{k}^{2}\gamma^{2}\mathbb{E}\left[\|\hat{F}(\bar{z}^{k},\bar{\xi}_{k})\|^{2}\mid\mathcal{F}_{k}\right]
        {}\leq{}
    \alpha_{k}^{2}\gamma^{2}\mathbb{E}\left[\|F\bar{z}^{k}\|^{2}\mid\mathcal{F}_{k}\right]+\alpha_{k}^{2}\gamma^{2}\sigma_F^{2}
 \numberthis\label{eq:normFzbar}
\end{align*}

Combining \eqref{eq:innerprod:2}, \eqref{eq:normFzbar}, and \eqref{eq:from_kp1_to_k} yields
\begin{align*}
\mathbb E[\|z^{k+1} &- z^\star\|^2 + A_{k+1}  \|u^{k}\|^2 + B_{k+1} \|z^{k+1} - z^k\|^2  \mid \F_k]
\\
    {}\leq{}&
\|z^k - z^\star\|^2 
    {}+{}
    \Big(%
    A_{k+1}+ \alpha_k  \big(\gamma^2 L_F^2(1+b) + \frac{1}{\varepsilon_k}\big)
    \Big)
    \mathbb E[\|u^k\|^2 \mid \F_k ] 
    {}-{}
    \alpha_k  \mu \|F(z^k)\|^2
    \\
    &{}+{}
    \left(\alpha_{k}\big(\gamma^{2}(\varepsilon_{k}-1)-2\gamma\rho\big)+\alpha_{k}^{2}\gamma^{2}\right) \E[]{\|F(\bar z^k)\|^2\mid \F_k}
    {}+{}\alpha_k^2\gamma^2\sigma_F^2
\\
    &{}+{}
B_{k+1}\alpha_{k}^{2}\gamma^{2}\mathbb{E}\left[\|F\bar{z}^{k}\|^{2}\mid\mathcal{F}_{k}\right]+B_{k+1}\alpha_{k}^{2}\gamma^{2}\sigma_F^{2}.
\numberthis
\label{eq:Descent1}
\end{align*}
Further using \eqref{eq:normuk} and denoting  
\begin{align*}
    X_1^k 
        {}\coloneqq{}& 
    \alpha_{k}\left(\gamma^{2}L_F^{2}(1+b)+\tfrac{1}{\varepsilon_{k}}\right)+A_{k+1},
    \\
    X_2^k
        {}\coloneqq{}&
    \alpha_{k}\left(\gamma^{2}(\varepsilon_k-1)-2\rho\gamma+\alpha_{k} \gamma^{2}\right)
\end{align*}
leads to 
\begin{align*}
    \E[]{\U_{k+1} \mid \F_k} - \U_k 
        {}\leq{}&
    -\alpha_{k}\mu\|F(z^{k})\|^{2}+\left(X_{1}^k(1-\betaparam_{k})^{2}-A_{k}\right)\|u^{k-1}\|^{2} 
    \\
        &{}+{}
    \left(2X_{1}^k(1-\betaparam_{k})^{2}\gamma^{2}L_{\hat{F}}^{2}-B_{k}\right)\|z^{k}-z^{k-1}\|^{2} 
        {}+{}
    \left(X_{2}^k+B_{k+1}\alpha_{k}^{2}\gamma^{2}\right)\E[]{\|F(\bar{z}^{k})\|^{2}\mid\F_{k}}
    \\
        &{}+{}
    \left(B_{k+1}\alpha_{k}^{2}+\alpha_{k}^{2}+2X_{1}^k\betaparam_{k}^{2}\right)\gamma^{2}\sigma_F^{2}. 
    \numberthis\label{eq:Descent:U:Ae0}
\end{align*}
Having established \eqref{eq:Descent:U:Ae0}, set \(A_k=A\), \(B_k= 2A\gamma^2L_{\hat F}^2\), and \(\varepsilon_k =\varepsilon\) to obtain by the law of total expectation that
\begin{align*}
    \E[]{\U_{k+1}} - \E[]{\U_k} 
        {}\leq{}&
    -\alpha_{k}\mu\E[]{\|F(z^{k})\|^{2}}+\left(X_{1}^k(1-\alpha_{k})^{2}-A\right)\E[]{\|u^{k-1}\|^{2}}
    \\
        &{}+{}
    2\gamma^{2}L_{\hat{F}}^{2}\left(X_{1}^k(1-\alpha_{k})^{2}-A\right)\E[]{\|z^{k}-z^{k-1}\|^{2}} 
        {}+{}
    \left(X_{2}^k+2A\gamma^4L_{\hat F}^2\alpha_{k}^{2}\right)\E[]{\|F(\bar{z}^{k})\|^{2}}
    \\
        &{}+{}
    \left(2A\gamma^2L_{\hat F}^2+1+2X_{1}^k\right)\alpha_{k}^{2}\gamma^{2}\sigma_F^{2}. 
    \numberthis\label{eq:Descent:U:Ae0:2}
\end{align*}
\rbl{
To get a recursion we require
\begin{equation}\label{eq:smooth:requirements}
X_{1}^k(1-\alpha_{k})^{2}-A \leq 0 \quad \text{and} \quad X_{2}^k+2A\gamma^4L_{\hat F}^2\alpha_{k}^{2} \leq 0.
\end{equation}
By developing the first requirement of \eqref{eq:smooth:requirements} we have,
\begin{equation}
0 \geq X_{1}^k(1-\alpha_{k})^{2}-A
     = \alpha_{k}(1-\alpha_{k})^{2}\left(\gamma^{2}L_F^{2}(1+b)+\tfrac{1}{\varepsilon}\right)+\alpha_k(\alpha_{k}-2)A.
\end{equation}
Equivalently, $A$ needs to satisfy
\begin{equation}
A \geq \frac{(1-\alpha_{k})^{2}}{2-\alpha_k}\left(\gamma^{2}L_F^{2}(1+b)+\tfrac{1}{\varepsilon}\right).
\end{equation}
for any $\alpha_k \in (0,1)$.
Since $\frac{(1-\alpha_{k})^{2}}{2-\alpha_k} \leq \tfrac{1}{2}$ given $\alpha_k \in (0,1)$ it suffice to pick
\begin{equation}\label{eq:smooth:A}
A = \tfrac{1}{2}\left((b+1)\gamma^{2}L_{F}^{2}+\tfrac{1}{\varepsilon}\right).
\end{equation}
For the second requirement of \eqref{eq:smooth:requirements} note that we can equivalently require that the following quantity is negative
\begin{align*}
    \tfrac{1}{\alpha_{k}\gamma^{2}}\left(X_{2}^{k}+2A\gamma^{4}L_{\hat{F}}^{2}\alpha_{k}^{2}\right) 
        &{}={}
    \varepsilon-1-\tfrac{2\rho}{\gamma}+\alpha_{k}+2A\gamma^{2}L_{\hat{F}}^{2}\alpha_{k} \\
        &{}\leq{}
    \varepsilon-1-\tfrac{2\rho}{\gamma}+\left(1+\left((b+1)\gamma^{2}L_{F}^{2}+\tfrac{1}{\varepsilon}\right)\gamma^{2}L_{\hat{F}}^{2}\right)\alpha_{0}
\end{align*}
where we have used that $\alpha_k \leq \alpha_0$ and the choice of $A$ from \eqref{eq:smooth:A}.
}
Setting the Young parameter \(\varepsilon = \gamma L_{\hat{F}}\sqrt{\alpha_{0}}\) we obtain that 
\(X_{2}^{k}+2A\gamma^{4}L_{\hat{F}}^{2}\alpha_{k}^{2}\leq 0\) owing to \eqref{app:eq:condest:Ae0}. 

On the other hand, the last term in \eqref{eq:Descent:U:Ae0:2} may be upper bounded by
\begin{align*}
    2A\gamma^{2}L_{\hat{F}}^{2}+1+2X_{1}^{k} 
        {}={}&
    1+\left((b+1)\gamma^{2}L_{F}^{2}+\tfrac{1}{\gamma L_{\hat{F}}\sqrt{\alpha_{0}}}\right)\left((\gamma^{2}L_{\hat{F}}^{2}+1)+2\alpha_{k}\right)
    \\
        {}\leq{}&
    1+\left((b+1)\gamma^{2}L_{F}^{2}+\tfrac{1}{\gamma L_{\hat{F}}\sqrt{\alpha_{0}}}\right)\left((\gamma^{2}L_{\hat{F}}^{2}+1)+2\alpha_{0}\right) = C. 
\end{align*}
Thus, it follows from \eqref{eq:Descent:U:Ae0:2} that 
\begin{align*}
    \E[]{\U_{k+1}} - \E[]{\U_k} 
        {}\leq{}&
    -\alpha_{k}\mu\E[]{\|F(z^{k})\|^{2}} 
        {}+{}
    C\alpha_{k}^{2}\gamma^{2}\sigma_F^{2}. 
\end{align*}
Telescoping the above inequality completes the proof.
\end{appendixproof}

\begin{appendixproof}{thm:BiasCorr:2}
The theorem is obtained as a particular instantiation of  \Cref{app:thm:BiasCorr:2}.

The condition in \eqref{eq:mu} can be rewritten as $b > \tfrac{\gamma^2L_F^2}{1-\gamma^2L_F^2}$.
A reasonable choice is $b = \tfrac{2\gamma^2L_F^2}{1-\gamma^2L_F^2}$.
Substituting back into $\mu$ we obtain
\begin{equation}
\begin{split}
\mu &= \gamma^2(1 - \gamma^2L_F^2(1+\tfrac{1-\gamma^2L_F^2}{2\gamma^2L_F^2})) = \tfrac{\gamma^2(1-\gamma^2L_F^2)}{2} > 0.
\end{split}
\end{equation}
Similarly, the choice of $b$ is substituted into $\eta$ and \eqref{app:eq:condest:Ae0} of \Cref{app:thm:BiasCorr:2}.

The rate in \eqref{app:eq:condest:Ae0} is further simplified by applying Lipschitz continuity of $F$ from \Cref{ass:AsymPrecon:M:Lip} to $\|Fz^0\|^2 = \|Fz^0-Fz^\star\|^2$.
The proof is complete by observing that the guarantee on the weighted sum can be converted into an expectation over a sampled iterate in the style of \citet{ghadimi2013stochastic}. 
\end{appendixproof}

\begin{ass}[almost sure convergence]\label{ass:stepsizes:Ae0}
    Let \(d\in[0,1]\), \(b>0\).
    Suppose that the following holds 
    \begin{enumerate}
        \item \label{ass:stepsizes:Ae0:1} the diminishing sequence \(\seq{\alpha_k}\subset(0,1)\) satisfies the classical conditions
        \[
        \textstyle
            \sum_{k=0}^\infty \alpha_k 
                {}={}
            \infty,\qquad 
            \bar{\alpha}
                {}\coloneqq{}
            \sum_{k=0}^\infty \alpha_k^2 
                {}<{}
            \infty; 
        \]
         \item\label{ass:stepsizes:Ae0:3} letting 
         \(
            c_k   
                {}\coloneqq{}
            (1+b) \gamma^2 L_F^2 + \frac{1}{\gamma L_{\hat{F}}}\alpha_{k}^{-d} 
        \)  for all \(k\geq 0\)
     \begin{equation}\label{eq:A0}
        \eta_{k} 
            {}\coloneqq{}
        \textstyle
        \sum_{\ell=k}^{\infty}
            \left(%
                c_{l}\alpha_{l}\Pi_{p=0}^{\ell}(1-\alpha_{p})^{2}
            \right) 
            {}<{}
        \infty,
        \qquad 
        \nu
            {}\coloneqq{}
        \sum_{k=0}^\infty \eta_{k+1}\alpha_{k}^{2}\left(\Pi_{p=0}^{k}\tfrac{1}{(1-\alpha_{p})^{2}}\right) < \infty,
    \end{equation}
    and 
\begin{align*}
    \gamma L_{\hat{F}}\alpha_{k}^d+\alpha_{k}
    +2\gamma^{2}L_{\hat{F}}^{2}\alpha_{k}
    \eta_{k+1}
        \Pi_{p=0}^{k}\tfrac{1}{(1-\alpha_{p})^{2}}
        {}\leq{} 
    1 + \tfrac{2\rho}{\gamma}.
    \numberthis
    \label{eq:app:thm:BiasCorr:ineq}
\end{align*}
     \end{enumerate} 
\end{ass}
Although at first look the above assumptions may appear involved, as shown in \Cref{app:thm:BiasCorr} classical stepsize choice of \(\tfrac{\alpha_0}{k+1}\) is sufficient to satisfy  \eqref{eq:A0}, and to ensure almost sure convergence provided that instead \eqref{eq:con2:d1} holds. Note that with this choice as \(k\) goes to infinity, \(\alpha_k\searrow 0\) and the deterministic range 
    \(
        \gamma + 2\rho > 0 
    \)
    is obtained.

\begin{thm}[almost sure convergence]\label{app:thm:BiasCorr}
Suppose that \cref{ass:AsymPrecon,ass:AsymPrecon:stoch,ass:AsymPrecon:stoch:stocLips} hold. 
Additionally, suppose the stepsize conditions in \cref{ass:stepsizes:Ae0:2,ass:stepsizes:Ae0}.
Then, the sequence \(\seq{z^k}\) generated by \Cref{alg:WeakMinty:Sto:Struct} converges almost surely to some \(z^\star\in \zer T\). Moreover, the following estimate holds
\begin{align}
\label{app:thm:BiasCorr:rate}
\sum_{k=0}^K \frac{\alpha_k}{\sum_{j=0}^K \alpha_j}\mathbb E[\|F(z^k)\|^2] 
    {}\leq{}
\frac{%
    \|z^{0} - z^\star\|^2 + \eta_0\gamma^2 \|F(z^{0})\|^2 
        {}+{}
    \bar C
    }{%
        \mu \sum_{j=0}^K \alpha_j
    },
\end{align}
where \(\bar C = 2\gamma^{2}\sigma_F^{2}\big((\gamma^{2}L_{\hat{F}}^{2}+1)\nu+\bar{\alpha}\big(\tfrac{1}{2}+(b+1)\gamma^{2}L_{F}^{2}+\tfrac{1}{\gamma L_{\hat{F}}}\big)\big)\) is finite. 

In particular, if \(\alpha_k = \tfrac{1}{k+r}\) for any positive natural number  \(r\), and \(d = 1\), then \cref{ass:stepsizes:Ae0:3} can be replaced by 
\begin{equation}\label{eq:con2:d1}
    (\gamma L_{\hat{F}}+1)\alpha_{k}+\rbl{2}\left((1+b)\gamma^{4}\rbl{L_{F}^{2}L_{\hat F}^{2}}\alpha_{k+1}+\gamma L_{\hat{F}}\right)\left(\alpha_{k+1}+1\right)\alpha_{k+1} 
        {}\leq{}
    1 + \tfrac{2\rho}{\gamma}.
\end{equation}
\end{thm}

\begin{appendixproof}{app:thm:BiasCorr}
Having established \eqref{eq:Descent:U:Ae0}, 
let \(B_k = 2A_k \gamma^2 L_{\hat F}^2\) such that
\begin{equation}
\left(2X_{1}^k(1-\betaparam_{k})^{2}\gamma^{2}L_{\hat{F}}^{2}-B_{k}\right)\|z^{k}-z^{k-1}\|^{2} 
= 2\gamma^{2} L_{\hat F}^2\left(X_{1}^k(1-\betaparam_{k})^{2}-A_{k}\right)\|z^{k}-z^{k-1}\|^{2}.
\end{equation}
In what follows we show that it is sufficient to ensure 
\begin{align*}
    X_1^k (1-\alpha_k)^2 \leq A_k,
    \quad 
    X_{2}^k+2A_{k+1}\gamma^{4}L_{\hat{F}}^{2}\alpha_{k}^{2}
        {}\leq{}
    0,
    \numberthis\label{eq:cond:stoch:A0}
\end{align*}
resulting in the inequality 
\begin{align*}
    \E[]{\U_{k+1} \mid \F_k} - \U_k 
        {}\leq{}&
    -\alpha_{k}\mu\|F(z^{k})\|^{2}
        {}+{}
    \left(2A_{k+1} \gamma^2 L_{\hat F}^2+1+2X_{1}^k\right)\alpha_{k}^{2}\gamma^{2}\sigma_F^{2}. 
    \numberthis\label{eq:Quasidescenttype}
\end{align*}
A reasonable choice for the Young parameter \(\varepsilon_k\) is to choose 
\begin{equation}\label{eq:varepsilon}
    \varepsilon_k 
        {}={}
    \gamma L_{\hat{F}}\alpha_{k}^d \quad
    \text{for some } \; d\in[0,1].
\end{equation}
The rational for this choice will become more clear in what follows.

The first inequality in \eqref{eq:cond:stoch:A0} is linear and we can solve it to equality by \cref{lem:recur}.
Let 
\begin{align*}
    A_0
    {}\coloneqq {}
    \sum_{\ell=0}^{\infty}
        \left(%
            c_{l}\alpha_{l}\Pi_{p=0}^{\ell}(1-\alpha_{p})^{2}
        \right) 
    {}={}
        \eta_0 
        \overrel[<]{\eqref{eq:A0}}\infty,
    \quad \textrm{and}
    \quad
    \nu 
        {}={}
    \sum_{k=0}^\infty A_{k+1}\alpha_k^2   
    \overrel[<]{\eqref{eq:A0}}\infty.
    \numberthis\label{eq:nuboundA0}
\end{align*}
Furthermore, let \(c_k\) and \(\eta_{k}\) be as in \cref{ass:stepsizes:Ae0:3}. Then, \cref{lem:recur} yields
\begin{align*}
    A_{k+1} 
        {}={}
    \left(%
        \Pi_{p=0}^{k}\frac{1}{(1-\alpha_{p})^{2}}
    \right)
        \left(
            A_{0}-\sum_{\ell=0}^{k}
            \left(%
                c_{l}\alpha_{l}\Pi_{p=0}^{\ell}(1-\alpha_{p})^{2}
            \right)
        \right)
        {}={}
    \eta_{k+1} \Pi_{p=0}^{k}\frac{1}{(1-\alpha_{p})^{2}} 
    \numberthis\label{eq:Akrec}
\end{align*}
which would ensure \(A_k\geq 0\) for all \(k\). 
Therefore, assumptions \eqref{eq:A0} and \eqref{eq:app:thm:BiasCorr:ineq} (which is a restatement of the conditions in \eqref{eq:cond:stoch:A0}) are sufficient for ensuring  \eqref{eq:Quasidescenttype}. Substituting \(X_1^k\) and  \(A_{k+1}\) in \eqref{eq:Quasidescenttype}  yields 
\begin{align*} 
    \E[]{\U_{k+1} \mid \F_k} - \U_k 
        {}\leq{}&
    -\alpha_{k}\mu\|F(z^{k})\|^{2}
        {}+{}
    \xi_k, 
    \numberthis\label{eq:Quasidescenttype:2}
\end{align*}
where 
\(
\xi_k = 2\left(A_{k+1}(\gamma^{2}L_{\hat{F}}^{2}+1)+\tfrac{1}{2}+(b+1)\gamma^{2}L_{F}^{2}\alpha_{k}+\tfrac{1}{\gamma L_{\hat{F}}}\alpha_{k}^{1-d}\right)\alpha_{k}^{2}\gamma^{2}\sigma_F^{2}. 
\)
By \cref{ass:stepsizes:Ae0} we have that 
\begin{align*}
    \textstyle
    \sum_{k=0}^\infty \xi_k 
        {}={}&
    2\gamma^{2}\sigma_F^{2}\left((\gamma^{2}L_{\hat{F}}^{2}+1)\sum_{k=0}^{\infty}A_{k+1}\alpha_{k}^{2}+\sum_{k=0}^{\infty}\tfrac{\alpha_{k}^{2}}{2}+(b+1)\gamma^{2}L_{F}^{2}\sum_{k=0}^{\infty}\alpha_{k}^{3}+\tfrac{1}{\gamma L_{\hat{F}}}\sum_{k=0}^{\infty}\alpha_{k}^{3-d}\right) 
    \\
        {}\leq{}&
    2\gamma^{2}\sigma_F^{2}\left((\gamma^{2}L_{\hat{F}}^{2}+1)\sum_{k=0}^{\infty}A_{k+1}\alpha_{k}^{2}+\left(\tfrac{1}{2}+(b+1)\gamma^{2}L_{F}^{2}+\tfrac{1}{\gamma L_{\hat{F}}}\right)\sum_{k=0}^{\infty}\alpha_{k}^{2}\right) 
        {}<{}
    \infty
\end{align*}
where we used the fact that $\alpha_k^3 \leq \alpha_k^2$ and \(d \leq 1\) in the first inequality, while the second inequality uses \eqref{eq:nuboundA0}, and \cref{ass:stepsizes:Ae0:1}. 
The claimed convergence result follows by the Robbins-Siegmund supermartingale theorem \cite[Prop. 2]{Bertsekas2011Incremental} and standard arguments as in \cite[Prop. 9]{Bertsekas2011Incremental}.

The claimed rate follows by taking total expectation and summing the above inequality over \(k\) and noting that initial iterates were set as \(\bar z^{-1} = z^{-1}= z^0\). 

To provide an instance of the sequence \(\seq{\alpha_k}\) that satisfy the assumptions, let \(r\) denote a positive natural number and set 
\begin{equation} \label{eq:alphak}
    \alpha_k = \tfrac{1}{k+r}. 
\end{equation}
Then, 
\begin{align*}
    \Pi_{p=0}^{\ell}(1-\alpha_{p})^{2} 
        {}={}&
    \Pi_{p=0}^{\ell}(\tfrac{p+r-1}{p+r})^{2} 
        {}={}
    \tfrac{(r-1)^2}{(\ell+r)^2}
        {}={}
    (r-1)^2 \alpha_\ell^2,  
\end{align*}
and for any \(K \geq 0\) 
\begin{align*}
    \sum_{\ell=0}^{K}\left(c_{\ell}\alpha_{\ell}\Pi_{p=0}^{\ell}(1-\alpha_{p})^{2}\right)
        {}={}
    \sum_{\ell=0}^{K}\tfrac{(r-1)^2}{(\ell+r)^{3}}c_{\ell}.
\end{align*}
Plugging the value of \(c_\ell\) and \(\varepsilon_k\) from \Cref{ass:stepsizes:Ae0:3} and \eqref{eq:varepsilon} we obtain that \(A_0\) is finite valued since \(\sum_{\ell=0}^{\infty}\tfrac{1}{(\ell+r)^3 \varepsilon_{\ell}}=\sum_{\ell=0}^{\infty}\tfrac{1}{(\ell+r)^{3-d}}<\infty\) owing to the fact that \(d\leq 1\). 

Moreover, 
\begin{align*}
     A_{k+1} 
        {}={}
        \frac{(k+r)^{2}}{(r-1)^{2}}
        \left(
            A_{0}-\sum_{\ell=0}^{k}
            \left(%
                \tfrac{(r-1)^2}{(\ell+r)^{3}}c_{\ell}
            \right)
        \right)
            {}={}
        (k+r)^{2}
            \sum_{\ell=k+1}^{\infty}
                \tfrac{1}{(\ell+r)^{3}}c_{\ell}
            {}={}
        \tfrac{1}{\alpha_k^2}
            \sum_{\ell=k+1}^{\infty}
                \alpha_\ell^3c_{\ell}      
        \numberthis\label{eq:Ak1}
\end{align*}

On the other hand, for \(e>1\) we have the following bound  %
\begin{align*}
    \sum_{\ell=k+1}^{\infty}\alpha_{\ell}^{e} 
        {}\leq{}
    \tfrac{1}{(k+1+r)^{e}} + \int_{k+1}^{\infty} \tfrac{1}{(x+r)^e}dx 
        {}={}
    \tfrac{1}{(k+1+r)^{e}}+\tfrac{1}{(e-1)(k+1+r)^{e-1}}.
    \numberthis\label{eq:intHarmonic}
\end{align*}
Therefore, it follows from \eqref{eq:Ak1} that  
\begin{align*}
      A_{k+1}\alpha_k 
        {}={}&
    \tfrac{1}{\alpha_k}\sum_{\ell=k+1}^{\infty} \big(%
        \alpha_\ell^3(1+{b}) \gamma^2 L_F^2 + \tfrac{1}{\gamma L_F}\alpha_\ell^{3-d}
    \big)
    \\
    \dueto{\eqref{eq:intHarmonic}}
        {}\leq{}&
    \left((1+b)\gamma^{2}L_{F}^{2}\tfrac{1}{2(k+1+r)}\right)\left(\tfrac{2}{k+1+r}+1\right)\tfrac{1}{k+1+r} 
    + \left(\tfrac{1}{\gamma L_{\hat{F}}}\tfrac{1}{(2-d)(k+1+r)^{1-d}}\right)\left(\tfrac{1}{k+1+r}+1\right)\tfrac{1}{k+1+r}
    \\
        {}={}&
    \left(\tfrac{1+b}{2}\gamma^{2}L_{F}^{2}\alpha_{k+1}\right)\left(2\alpha_{k+1}+1\right)\alpha_{k+1}
    +\left(\tfrac{1}{\gamma L_{\hat{F}}(2-d)}\alpha_{k+1}^{1-d}\right)\left(\alpha_{k+1}+1\right)\alpha_{k+1}
    \\
        {}\leq{}& 
    \left((1+b)\gamma^{2}L_{F}^{2}\alpha_{k+1}+\tfrac{1}{\gamma L_{\hat{F}}(2-d)}\alpha_{k+1}^{1-d}\right)\left(\alpha_{k+1}+1\right)\alpha_{k+1}
    \numberthis\label{eq:Aalpha}
  \end{align*}  
In turn, this inequality ensures that \(\nu\) as defined in \cref{ass:stepsizes:Ae0:3} is finite. To see this note that 
\begin{align*}
    \textstyle
    \nu 
        {}={}
    \sum_{k=0}^{\infty} A_{k+1}\alpha_k^2 
        {}\overrel[\leq]{\eqref{eq:Aalpha}}{} 
    \sum_{k=0}^{\infty} \left((1+b)\gamma^{2}L_{F}^{2}\alpha_{k+1}+\tfrac{1}{\gamma L_{\hat{F}}(2-d)}\alpha_{k+1}^{1-d}\right)\left(\alpha_{k+1}+1\right)\alpha_{k+1} \alpha_k 
        {}\leq{}
    \delta \sum_{k=0}^{\infty} \alpha_k^2 <\infty,  
\end{align*}
where in the last two inequalities \cref{ass:stepsizes:Ae0:1} was used.

It remains to confirm the second inequality in \eqref{eq:cond:stoch:A0}. With the choice of \(\alpha_k\) and \(\varepsilon_k\) as in \eqref{eq:alphak} and \eqref{eq:varepsilon} we have 
\begin{align*}
    \tfrac{1}{\alpha_k\gamma^2}\big(X_{2}+2A_{k+1}&\gamma^{4}L_{\hat{F}}^{2}\alpha_{k}^{2}\big) 
    \\
        {}={}&
    \gamma L_{\hat{F}}\alpha_{k}^d-1-\tfrac{2\rho}{\gamma}+\alpha_{k}
    +2A_{k+1}\gamma^{2}L_{\hat{F}}^{2}\alpha_{k}
    \\
    \dueto{\eqref{eq:Aalpha}}
        {}\leq{}&
    \gamma L_{\hat{F}}\alpha_{k}^{d}+\alpha_{k}+2\gamma^{2}L_{\hat{F}}^{2}\left((1+b)\gamma^{2}L_{F}^{2}\alpha_{k+1}+\tfrac{1}{\gamma L_{\hat{F}}(2-d)}\alpha_{k+1}^{1-d}\right)\left(\alpha_{k+1}+1\right)\alpha_{k+1}
        {}-{}
    1-\tfrac{2\rho}{\gamma}.
  \end{align*}  
It follows that with \(d=1\) the assumption \eqref{eq:con2:d1} is sufficient to ensure that the second condition in \eqref{eq:cond:stoch:A0} holds. 
\end{appendixproof}

\begin{appendixproof}{thm:BiasCorr}
The result is a restatement of the special case in \Cref{app:thm:BiasCorr} where $\alpha_k = \frac{1}{k+r}$.
We proceed similarly to the proof of \Cref{thm:BiasCorr:2}.

The condition in \eqref{eq:mu} can be rewritten as $b > \tfrac{\gamma^2L_F^2}{1-\gamma^2L_F^2}$.
A reasonable choice is $b = \tfrac{2\gamma^2L_F^2}{1-\gamma^2L_F^2}$.
The choice of $b$ is substituted into \eqref{eq:mu}, \eqref{eq:con2:d1} and $\bar C$ of \Cref{app:thm:BiasCorr}.
This completes the proof.

\end{appendixproof}

\section{Proof for constrained case}\label{app:const}
    
\rbl{
We will rely on two well-known and useful properties of the \emph{deterministic} operator $H = \id - \gamma F$ from \cite[Lm. A.3]{pethick2022escaping} that we restate here for convenience.
\begin{lem}\label{lem:H}
	Let $F:\R^n\to \R^n$ be a $L_F$-Lipschitz operator and $H=\id - \gamma F$ with $\gamma\in(0,\nicefrac1{L_F}]$. Then, 
	\begin{enumerate}
		\item\label{lem:H:coco} 
        The operator $H$ is $\nicefrac12$-cocoercive.
		\item \label{lem:H:SM} The operator $H$ is $(1-\gamma L_F)$-monotone, and 
		in particular 
		\begin{equation}
        \|Hz'-Hz\| \geq (1-\gamma L_F)\|z'-z\|
		\quad \forall z,z'\in\R^n. 
        \end{equation}
	\end{enumerate}
\begin{proof}
    The first claim follows from direct computation
    \begin{equation}
    \label{eq:const:cocoercive}
    \begin{split}
    \langle Hz - Hz', z - z'\rangle &= \langle Hz - Hz', Hz - Hz' + \gamma Fz - \gamma Fz'\rangle \\
    &= \tfrac{1}{2} \|Hz - Hz'\|^2 - \tfrac{\gamma^2}{2} \|Fz' - Fz\|^2 + \tfrac{1}{2} \|z' - z\|^2 \\
    &\geq \tfrac{1}{2} \|Hz - Hz'\|^2,
    \end{split}
    \end{equation}
    where the last inequality is due to Lipschitz continuity and $\gamma \leq \nicefrac{1}{L_F}$.
    The strongly monotonicity of $H$ is a consequence of Cauchy-Schwarz and Lipschitz continuity of $F$,
	\begin{align*}
		\langle Hz' - Hz, z'-z \rangle  
			{}={} &
		\|z'-z\|^2 - \gamma\langle Fz' - Fz, z'-z\rangle  
			{}\geq{} 
		(1-\gamma L)\|z'-z\|^2.
	\end{align*}
	The last claim follows from the Cauchy-Schwarz inequality.
\end{proof}
\end{lem}
}

\begin{thm}\label{app:thm:const:convergence}
    Suppose that \cref{ass:AsymPrecon,ass:AsymPrecon:stoch,ass:AsymPrecon,ass:AsymPrecon:stoch:stocLips} hold. 
    Moreover, suppose that \(\alpha_k\in(0,1)\), \(\gamma\in (\lfloor-2\rho\rfloor_+, \nicefrac1{L_F})\) and for positive parameters $\varepsilon$ and $b$ the following holds,
    \begin{equation}\label{app:thm:const:convergence:conditions}
    \mu \coloneqq \tfrac{1}{1+b}(1-\tfrac {1}{\varepsilon (1-\gamma L_F)^2}) - \alpha_0(1+2\gamma^2L_{\hat F}^2 A) + \tfrac{2\rho}{\gamma} > 0
    \quad \text{and} \quad 
    1-\tfrac {1}{\varepsilon (1-\gamma L_F)^2} \geq 0
    \end{equation}
    where $A \geq \varepsilon + \tfrac 1b(1 - \tfrac {1}{\varepsilon (1-\gamma L_F)^2})$. 
    Consider the sequence \(\seq{z^k}\) generated by \Cref{alg:WeakMinty:Sto:StructA}.
    Then, the following estimate holds for all $z^\star \in \mathcal S^\star$ 
    \begin{equation}\label{app:thm:const:rate}
        \sum_{k=0}^K \tfrac{\alpha_k}{\sum_{j=0}^K \alpha_{j}} \mathbb E[\|h^k - H\z^k\|^{2}]
            {}\leq{}
        \frac{\mathbb E[\|z^{0}-z^{\star}\|^{2}] + A \mathbb E[\|h^{-1}-Hz^{-1}\|^{2}] + C \gamma^2\sigma_F^2\sum_{j=0}^K \alpha_{j}^{2}}{\mu \sum_{j=0}^K \alpha_{j}}
    \end{equation}
    where $C = 1 + 2A(1+\gamma^2L_{\hat F}^2) + 2\alpha_0 A$.
\end{thm}

\begin{appendixproof}{app:thm:const:convergence}
We rely on the following potential function,
\begin{align*}
    \U_{k+1} 
        {}\coloneqq{}
    \|z^{k+1}-z^{\star}\|^{2}
    +A_{k+1}\|h^{k}-Hz^{k}\|^{2}
    +B_{k+1}\|z^{k+1}-z^{k}\|^{2},
\end{align*}
where \(\seq{A_{k}}\) and \(\seq{B_{k}}\) are positive scalar parameters to be identified.

We will denote  $\hat {\bar{H}}_k := \z^k - \gamma\hat F(\z^k, \bar \xi_k)$, so that $z^{k+1} = z^k - \alpha_k (h^k - \hat{\bar H}_k)$. 
Then, expanding one step,
\begin{align}
\|z^{k+1} - z^\star\|^2= \|z^k - z^\star\|^2- 2\alpha_k \langle h^k - \hat{\bar H}_k, z^k - z^\star\rangle + \alpha_k^2 \|h^k - \hat{\bar H}_k\|^2.
\label{eq:const:from_kp1_to_k_2}
\end{align}
Recall that $Hz \coloneqq z - \gamma Fz$ in the deterministic case.
In the \Cref{alg:WeakMinty:Sto:StructA}, $h^k$ estimates $Hz^k$. 
	Let us quantify how good this estimation is.
	\begin{align*}
	h^{k} - Hz^k ={}& \gamma Fz^k - \gamma\hat F(z^k,\xi_k) + (1-\alpha_{k-1}) (h^{k-1} - z^{k-1} + \gamma\hat F(z^{k-1},\xi_k))\\
	\|h^{k} - Hz^k\|^2 ={}& (1-\alpha_{k-1})^2 \|h^{k-1} - z^{k-1} + \gamma Fz^{k-1}\|^2 \\
	&+ \|\gamma Fz^k - \gamma\hat F(z^k,\xi_k)  + (1-\alpha_{k-1})(\gamma\hat F(z^{k-1},\xi_k) - \gamma Fz^{k-1})\|^2 \\
	& + 2 (1-\alpha_{k-1}) \langle h^{k-1} - z^{k-1} + \gamma Fz^{k-1}, \\
  & \hspace{6em} \gamma Fz^k - \gamma\hat F(z^k,\xi_k)  + (1-\alpha_{k-1})(\gamma\hat F(z^{k-1},\xi_k) - \gamma Fz^{k-1})\rangle
	\end{align*}
	In the scalar product, the left term is known when $z^k$ is known and the right term has an expectation equal to 0 by \Cref{ass:AsymPrecon:stoch:unbiased} when $z^k$ is known.
	Thus, taking conditional expectation and using the fact that the second moment is larger than the variance, we can go on as
	\begin{align}
	\mathbb E[\|h^{k} - Hz^k \|^2 \;|\; \mathcal F_k] &\leq (1-\alpha_{k})^2 \|h^{k-1} - Hz^{k-1}\|^2 \notag \\
	&\qquad + \mathbb E[ 2(1-\alpha_{k})^2\gamma^2 \| \hat F(z^k,\xi_k) - \hat F(z^{k-1},\xi_k) \|^2 \; |\; \mathcal F_k] \notag \\
  &\qquad + \mathbb E[2\alpha_{k}^2\gamma^2 \| Fz^k - \hat F(z^k,\xi_k) \|^2  \; |\; \mathcal F_k] \notag\\
	&\leq (1-\alpha_{k})^2 \|h^{k-1} - Hz^{k-1}\|^2 + 2(1-\alpha_{k})^2 L_{\hat F}^2\gamma^2 \|z^k - z^{k-1}\|^2 + 2\alpha_{k}^2 \gamma^2\sigma_F^2
	\label{eq:const:approx_Hzk}
	\end{align}
  where we have used \Cref{ass:AsymPrecon:stoch:boundedvar} and \Cref{ass:AsymPrecon:stoch:stocLips}.
	
	We continue with the conditional expectation of the inner term in \eqref{eq:const:from_kp1_to_k_2}. 
	\begin{equation}\label{eq:const:innerprods}
  \begin{split}
	-\mathbb E[\langle h^k - \hat{\bar H}_k,  z^k - z^\star\rangle \;|\; \mathcal F_k] &= -\langle h^k - H\bar z^k, z^k - z^\star\rangle \\
	& = -\langle h^k - H\bar z^k, z^k - \bar z^k\rangle -\langle h^k - H\bar z^k, \bar z^k - z^\star\rangle \\
	& = -\langle h^k - Hz^k, z^k - \bar z^k\rangle -  \langle Hz^k - H\bar z^k, z^k - \bar z^k\rangle -\langle h^k - H\bar z^k, \bar z^k - z^\star\rangle \\
	& \leq -\langle h^k - Hz^k, z^k - \bar z^k\rangle -  \tfrac{1}{2}\|Hz^k - H\bar z^k\|^2 -\langle h^k - H\bar z^k, \bar z^k - z^\star\rangle 
  \end{split}
	\end{equation}
where the last inequality uses $\nicefrac{1}{2}$-cocoercivity of $H$ from \Cref{lm:H:cocoercive} under \Cref{ass:AsymPrecon:M:Lip} and the choice $\gamma \leq 1/L_F$.

By definition of $\bar z^k$ in \Cref{state:barzL:sto:Struct}, we have $h^k \in \z^k + \gamma A(\bar z^k)$, so that $\tfrac{1}{\gamma}(h^k - H\bar z^k) \in F(\bar z^k) + A(\bar z^k)$. Hence, using the weak MVI from \Cref{ass:AsymPrecon:Minty:Struct},
\begin{equation}\label{eq:const:minty}
\langle h^k - H\bar z^k, \bar z^k - z^\star\rangle 
\geq \tfrac{\rho}{\gamma} \|h^k - H\bar z^k\|^2 \; .
\end{equation}

Using \eqref{eq:const:minty} in \eqref{eq:const:innerprods} leads to the following inequality, true for any $\varepsilon_k>0$:
	\begin{align*}
	-\mathbb E[\langle h^k - \hat{\bar H}_k,  z^k - z^\star\rangle \;|\; \mathcal F_k]  \leq \tfrac{\varepsilon_k}{2} \|h^k - Hz^k\|^2 + \tfrac 1{2\varepsilon_k} \|\bar z^k - z^k\|^2 - \tfrac 12 \|Hz^k-H\bar z^k\|^2 - \tfrac{\rho}{\gamma} \|h^k - H\bar z^k\|^2\;.
	\end{align*}
	\rbl{
    To majorize the term $\|\bar z^k - z^k\|^2$, we use \Cref{lm:H:strmonotone} to get
\begin{align*}
\|H\bar z^k - Hz^k\|^2 \geq (1- \gamma L_F)^2\|\bar z^k - z^k\|^2 \;.
\end{align*}
Hence, as long as $\gamma L_F < 1$, then
}
	\begin{align}
-\mathbb E[\langle h^k - \hat{\bar H}_k,  z^k - z^\star\rangle \;|\; \mathcal F_k]  \leq \tfrac{\varepsilon_k}{2}  \|h^k - Hz^k\|^2 + \Big(\tfrac {1}{2\varepsilon_k (1-\gamma L_F)^2}- \tfrac 12\Big) \|Hz^k-H\z^k\|^2 - \tfrac{\rho}{\gamma} \|h^k - H\bar z^k\|^2\;.
\label{eq:const:second_term}
\end{align}
The third term in \eqref{eq:const:from_kp1_to_k_2} is bounded by
\begin{align}
\alpha_k^2 \mathbb E[  \|h^k - \hat{\bar H}_k\|^2 \; |\; \mathcal F_k] = \alpha_k^2\|h^k - H\bar z^k\|^2 + \alpha_k^2\gamma^2 \mathbb E[  \|F\bar z^k - \hat F(\bar z^k,\bar \xi_k)\|^2 \; |\; \mathcal F_k] \leq \alpha_k^2\|h^k - H\bar z^k\|^2 + \alpha_k^2\gamma^2 \sigma_F^2
\label{eq:const:third_term}
\end{align}

Combined with the update rule, \eqref{eq:const:third_term} can also be used to bound the difference of iterates
\begin{equation}\label{eq:const:iterate_diff}
\mathbb E[\|z^{k+1} - z^k\|^2 \; |\; \mathcal F_k]
= \mathbb E[\alpha_k^2\|h^k - \hat{\bar H}_k\|^2 \; |\; \mathcal F_k]
\leq \alpha_k^2\|h^k - H\z^k\|^2 + \alpha_k^2\gamma^2\sigma_F^2
\end{equation}
Using \eqref{eq:const:from_kp1_to_k_2}, \eqref{eq:const:second_term}, \eqref{eq:const:third_term} and \eqref{eq:const:iterate_diff} we have,
\begin{equation}
\begin{split}
\mathbb E[\U_{k+1} \; |\; \mathcal F_k] 
& \leq 
    \|z^k - z^\star\|^2 + (A_{k+1} + \alpha_k \varepsilon_k)  \|h^k - Hz^k\|^2 - \alpha_k \Big(1-\tfrac {1}{\varepsilon_k (1-\gamma L_F)^2}\Big) \|Hz^k-H\bar z^k\|^2 \\
    & \qquad + \alpha_k(\alpha_k - \tfrac{2\rho}{\gamma} + \alpha_k B_{k+1}) \|h^k - H\bar z^k\|^2 + \alpha_k^2(1+B_{k+1})\gamma^2\sigma_F^2 \\
& \leq 
    \|z^k - z^\star\|^2 
    + \big(A_{k+1} + \alpha_k (\varepsilon_k + \tfrac 1b(1-\tfrac {1}{\varepsilon_k (1-\gamma L_F)^2}))\big)  \|h^k - Hz^k\|^2 \\
    &\qquad + \alpha_k \Big(\alpha_k - \tfrac{2\rho}{\gamma} + \alpha_k B_{k+1} - \tfrac{1}{1+b}(1-\tfrac {1}{\varepsilon_k (1-\gamma L_F)^2})\Big) \|h^k - H\bar z^k\|^2 \\
    & \qquad + \alpha_k^2(1+B_{k+1})\gamma^2\sigma_F^2,
\end{split}
\end{equation}
where the last inequality follows from Young's inequality with positive $b$ and requiring $1-\tfrac {1}{\varepsilon_k (1-\gamma L_F)^2} \geq 0$ as also stated in \eqref{app:thm:const:convergence:conditions}.
By defining
\begin{equation}
\begin{split}
X_k^1 & \coloneqq A_{k+1} + \alpha_k (\varepsilon_k + \tfrac 1b(1 - \tfrac {1}{\varepsilon_k (1-\gamma L_F)^2})) \\
X_k^2 & \coloneqq \alpha_k \Big(\alpha_k - \tfrac{2\rho}{\gamma} + \alpha_k B_{k+1} - \tfrac{1}{1+b}(1-\tfrac {1}{\varepsilon_k (1-\gamma L_F)^2})\Big)
\end{split}
\end{equation}
and applying \eqref{eq:const:approx_Hzk}, we finally obtain 
\begin{equation}
\begin{split}
\label{eq:const:descent}
\mathbb E[\U_{k+1} \; |\; \mathcal F_k]  - \U_{k}
& \leq 
    X_k^2 \|h^k - H\bar z^k\|^2 \\
    & \qquad + (X_k^1(1-\alpha_{k})^2 - A_k) \|h^{k-1} - Hz^{k-1}\|^2  \\
    & \qquad + (2X_k^1(1-\alpha_{k})^2 \gamma^2L_{\hat F}^2 - B_k) \|z^k - z^{k-1}\|^2  \\
    & \qquad + 2X_k^1\alpha_{k}^2 \gamma^2\sigma_F^2 + \alpha_k^2(1+B_{k+1})\gamma^2\sigma_F^2,
\end{split}
\end{equation}
We can pick $B_k = 2\gamma^2L_{\hat F}^2 A_k$ in which case, to get a recursion, we only require the following.
\rbl{
\begin{equation}
\label{eq:const:recursion}
X_k^1(1- \alpha_{k})^2 - A_k \leq 0
\quad \text{and} \quad
X_k^2 < 0
\end{equation}
Set $A_k=A$, $\varepsilon_k = \varepsilon$.
For the first requirement of \eqref{eq:const:recursion}, 
\begin{equation}
\begin{split}
X_k^1(1- \alpha_{k})^2 - A_k
&= 
    \alpha_k (1-\alpha_{k})^2 (\varepsilon + \tfrac 1b(1 - \tfrac {1}{\varepsilon (1-\gamma L_F)^2}))
    + (1-\alpha_{k})^2A - A \\
&\leq 
    \alpha_k (\varepsilon + \tfrac 1b(1 - \tfrac {1}{\varepsilon (1-\gamma L_F)^2}))
    + (1-\alpha_{k})^2A - A \\
&\leq 
    \alpha_k (\varepsilon + \tfrac 1b(1 - \tfrac {1}{\varepsilon (1-\gamma L_F)^2}))
    + (1-\alpha_{k})A - A \\
&= 
    \alpha_k (\varepsilon + \tfrac 1b(1 - \tfrac {1}{\varepsilon (1-\gamma L_F)^2}))
    - \alpha_k A
\end{split}
\end{equation}
where the first inequality follows from $\left(1-\alpha_{k}\right)^2 \leq 1$ and the second inequality follows from $\left(1-\alpha_{k}\right)^2 \leq\left(1-\alpha_{k}\right)$.
Thus, to satisfy the first inequality of \eqref{eq:const:recursion} it suffice to pick
\begin{equation}\label{eq:const:A}
A \geq \varepsilon + \tfrac 1b(1 - \tfrac {1}{\varepsilon (1-\gamma L_F)^2}).
\end{equation}
}

The noise term in \eqref{eq:const:descent} can be made independent of $k$ by using $\alpha_k \leq \alpha_0$ and \eqref{eq:const:A} as follows
\begin{equation}
\begin{split}
2X_k^1 + 1 + B_{k+1}
    &= 1 + 2A(1+\gamma^2L_{\hat F}^2)
        + 2\alpha_k (\varepsilon + \tfrac 1b(1 - \tfrac {1}{\varepsilon (1-\gamma L_F)^2}))\\
    &\leq 1 + 2A(1+\gamma^2L_{\hat F}^2)
        + 2\alpha_0 A
    = C.
\end{split}
\end{equation}

Thus it follows from \eqref{eq:const:descent} and $\alpha_k \leq \alpha_0$ that
\begin{equation}
\begin{split}
\mathbb E[\U_{k+1} \; |\; \mathcal F_k]  &- \U_{k} \\
   \leq{}& \alpha_k \Big(\alpha_0 - \tfrac{2\rho}{\gamma} + 2\alpha_0 \gamma^2L_{\hat F}^2 A - \tfrac{1}{1+b}(1-\tfrac {1}{\varepsilon_k (1-\gamma L_F)^2})\Big)\|h^k-H\bar z^k\|^2
    + \alpha_k^2C\gamma^2\sigma_F^2.
\end{split}
\end{equation}
The result is obtained by total expectation and summing the above inequality while noting that the initial iterate were set as \(z^{-1}= z^0\). 
\end{appendixproof}

\begin{appendixproof}{thm:const:convergence}
The theorem is a specialization of \Cref{app:thm:const:convergence} with a particular a choice of $b$ and $\varepsilon$.
The second requirement in \eqref{app:thm:const:convergence:conditions} can be rewritten as,
\begin{equation}
\varepsilon \geq \tfrac{1}{(1-\gamma L_F)^2},
\end{equation}
which is satisfied by $\varepsilon = \tfrac{1}{\sqrt{\alpha_0}(1-\gamma L_F)^2}$.
We substitute in the choice of $\varepsilon$, $b=\sqrt{\alpha_0}$ and denotes $\eta \coloneqq A$.

The weighted sum in \eqref{app:thm:const:rate} can be converted into an expectation over a sampled iterate in the style of \citet{ghadimi2013stochastic},
\begin{align*}
    \mathbb E[\|h^{k_\star} - H\z^{k_\star}\|^{2}]
        {}={}
    \sum_{k=0}^K \tfrac{\alpha_k}{\sum_{j=0}^K \alpha_{j}} \mathbb E[\|h^k - H\z^k\|^{2}]
\end{align*}
with $k_{\star}$ chosen from $\{0,1, \ldots, K\}$ according to probability $\mathcal{P}\left[k_{\star}=k\right]=\frac{\alpha_k}{\sum_{j=0}^K \alpha_j}$.

\rbl{
Noticing that $h^{k_\star} - H\z^{k_\star} \in \gamma(F\z^{k_\star} + A\z^{k_\star}) = \gamma T\z^{k_\star}$ so
\begin{align*}
    \mathbb E[\|h^{k_\star} - H\z^{k_\star}\|^{2}] 
        \geq \min_{u \in T\z^{k_\star}} \mathbb E[\|\gamma u\|^2] 
        \geq \mathbb E[\min_{u \in T\z^{k_\star}} \|\gamma u\|^2]
        =: \mathbb E[\dist(0,\gamma T\z^{k_\star})^2]
\end{align*}
where the second inequality follows from concavity of the minimum.
This completes the proof.
}
\end{appendixproof}

\section{Proof for NP-PDEG through a nonlinear asymmetric preconditioner}\label{app:AFBA}\label{app:NP-PDHG}
    
{%
\subsection{Preliminaries}\label{app:AFBA:prelim}

Consider the decomposition \(z=(z_1,\ldots, z_m)\), \(u=(u_1,\ldots, u_m)\) with \(z_i,u_i\in\R^{n_i}\) and define the shorthand notation \(u_{\leq i} \coloneqq (u_1,u_2,\ldots, u_i)\) and \(u_{\geq i} \coloneqq (u_{i},\ldots, u_m)\) for the truncated vectors. Moreover suppose that \(A\) conforms to the decomposition \(Az =(A_1,z_1,\ldots, A_mz_m)\) with \(A_i:\R^{n_i}\rightrightarrows\R^{n_i}\) maximally monotone. Consistently with the decomposition define \(\Gamma = \blockdiag(\Gamma_1,\ldots,\Gamma_m)\) where \(\Gamma_i\in\R^{n_i\times n_i}\) are positive definite matrices and let 
\begin{equation}\label{eq:AsymPrecon:Q}
\PC[u][z] 
    {}\coloneqq{}
\Gamma^{-1} z + \QC[u][z], \quad 
\text{\rm where}
\;
\QC[u][z]
    {}={}
\big(
    0, q_1(z_1, u_{\geq 2}), q_2(z_1, z_2, u_{\geq 3}),\ldots, q_{m-1}(z_{\leq m-1}, u_m) 
\big)
\end{equation}   
When \(P_{u}\) furnishes such an asymmetric structure the preconditioned resolvent has full domain, thus ensuring that the algorithm is well-defined. 

In the following lemma we show that the iterates in \eqref{eq:AFBA} are well-defined for a particular choice of the preconditioner $\PC[u]$ in \eqref{eq:AsymPrecon:Q}. The proof is similar to that of  \cite[Lem. 3.1]{Latafat2017Asymmetric} and is included for completeness. 

\begin{lem}\label{lem:AFBAfulldomain}
    Let \(z=(z_1,\ldots, z_m)\), \(u=(u_1,\ldots, u_m)\) be given vectors, suppose that \(A\) conforms to the decomposition \(Az =(A_1,z_1,\ldots, A_mz_m)\) with \(A_i:\R^{n_i}\rightrightarrows\R^{n_i} \) maximally monotone, and let \(P_{u}\) be defined as in \eqref{eq:AsymPrecon:Q}. Then, the preconditioned resolvent \((P_u + A)^{-1}\) is Lipschitz continuous and has full domain. Moreover, the update \(\bar z = (P_u + A)^{-1} z\) reduces to the following update
    \begin{align*}
        \bar z_i 
            {}={}
        \begin{cases}
             (\Gamma_1^{-1} + A_1)^{-1}z_1 & \quad \text{if } \; i=1\\
             (\Gamma_i^{-1} + A_i)^{-1}(z_i - q_{i-1}(\bar z_{\leq i-1}, u_{\geq {i}}) & \quad \text{if } \; i = 2,\ldots, m
         \end{cases} 
         \numberthis\label{eq:barz:update:AFBA}
    \end{align*}
\begin{proof}
Owing to the asymmetric structure \eqref{eq:AsymPrecon:Q}, the resolvent may equivalently be expressed as 
\begin{align*}
    \bar z 
        {}={} 
    (\bar z_1,\ldots \bar z_m) = (P_u + A)^{-1}z 
        {}\iff{} 
    \Gamma_{i}^{-1}\bar z_i + A_i(\bar z_i) \in z_i - q_{i-1}(\bar z_{\leq i-1}, u_{\geq i}), \quad i=1,\ldots, m,
\end{align*}
where \(q_0 \equiv 0\). The Gauss-Seidel-type update in \eqref{eq:barz:update:AFBA} is of  immediate verification after noting that \((\Gamma_i^{-1}+A_i)^{-1}\)  is single-valued (in fact Lipschitz continuous) 
 since the sum of \(\Gamma_i \succ 0\) and \(A_i\) is (maximally) strongly monotone. This also implies that \(\Gamma_i^{-1} + A_i = \bar A_i + \beta \I\) for some \(\beta>0\) and some maximally monotone operator \(\bar A_i\). Thus 
 \(
    \dom\big(\Gamma_i^{-1}+ A_i)^{-1}\big)
        {}={}
    \range(\Gamma_i^{-1}+ A_i)
        {}={}
    \range(\tfrac{1}{\beta}\bar A + \I) = \R^n 
\), where we used Minty's theorem in the last equality. 
\end{proof}
\end{lem}

\subsection{Deterministic lemmas}

To eventually prove \Cref{app:thm:AsymPrecon:convergence} we will compare the stochastic algorithm \eqref{eq:AFBA:Stoc} with its deterministic counterpart \eqref{eq:AFBA}, so we introduce
\begin{subequations}
\begin{align}
\HC[u][z]
    {}\coloneqq{}&
\PC[u][z] - F(z) \label{eq:AFBA:det:H}
\\
\bar G (z) 
    {}\coloneqq{}&
(\PC[z] + A)^{-1}(\HC[z][z]) \label{eq:AFBA:det:barT}
\\
G(z) 
    {}\coloneqq{}&
z - \alpha_k \Gamma \left(\HC[z][z] - \HC[z][\bar G(z)]\right) \label{eq:AFBA:det:T}.
\end{align}
\end{subequations}
We first derive results for the deterministic operator $G$ and then shows that $z^{k+1}$ from the stochastic scheme behaves similarly to  ${G}(z^k)$ when $\alpha_k$ is small enough, even if $\Gamma$, which also appears inside the preconditioner $\SPC[u][\cdot][\xi]$, remains large.

Instead of making assumptions on $F$ directly, we instead consider the following important operator,
\begin{equation}\label{eq:def:M}
\MMC[u][z]:=F(z)-\QC[u][z].
\end{equation}
such that we can write \eqref{eq:AFBA:det:barT} as $\HC[u][z]=\Gamma^{-1}z - \MMC[u][z]$. As a shorthand we write $\MC[z][z]=\MMC[z][z]$.
\begin{ass}\label{ass:M:Lips}
The operator $\MMC[u]$ as defined in \eqref{eq:def:M} is $L_M$-Lipschitz with $L_M \leq 1$ with respect to a positive definite matrix $\Gamma \in \R^{n\times n}$, i.e.
\begin{equation}
\|\MMC[u][z] - \MMC[u][z']\|_{\Gamma} \leq L_M\|z-z^\prime\|_{\Gamma^{-1}} \quad \forall z,z' \in \R^n.
\end{equation}
\end{ass}
\begin{remark}
This is satisfied by the choice of $\QC[u]$ in \eqref{eq:Q:def} and \Cref{ass:PD:Lip:phi,ass:PD:stepsize}.
\end{remark}

With $\MMC[u]$ defined, it is straightforward to establish that \(\HC[u]\) is \(\nicefrac12\)-cocoercive and strongly monotone. %
\begin{lem}\label{lm:H:properties}
Suppose \Cref{ass:M:Lips} holds. Then, 
\begin{enumerate}
\item\label{lm:H:cocoercive} 
The mapping \(\HC[u]\) is \(\nicefrac12\)-cocoercive for all $u \in \R^n$, i.e.
\begin{equation}
\langle \HC[u][z']-\HC[u][z],z'-z\rangle \geq \tfrac12\|\HC[u][z']-\HC[u][z]\|^2_{\Gamma} \quad \forall z,z'\in\R^n.
\end{equation}
\item\label{lm:H:strmonotone} 
Furthermore, $\HC[u]$ is $(1-L_M)$-monotone for all $u \in \R^n$, and in particular
\begin{equation}
\|H_{u}(z') - H_{u}(z)\|_\Gamma \geq (1- L_M)\|z' - z\|_{\Gamma^{-1}} \quad \forall z,z'\in\R^n.
\end{equation}
\end{enumerate}
\end{lem}
\begin{proof}
By expanding using \eqref{eq:def:M},
\begin{align*}
    \HC[u][z] - \HC[u][z'] 
        {}={}
    \Gamma^{-1}(z -z') - (\MMC[u][z] - \MMC[u][z']).
    \numberthis\label{eq:Hdif:1}
\end{align*}
Using this we can show cocoercivity, 
\begin{align*}
    \langle \HC[u][z']-\HC[u][z],z'-z\rangle
        {}={}&
    \langle \HC[u][z']-\HC[u][z],\HC[u][z']-\HC[u][z] - (\MMC[u][z] - \MMC[u][z'])\rangle_{\Gamma}
    \\
    \dueto{\eqref{eq:Hdif:1}}    
        {}={}&
    \tfrac12\|\HC[u][z']-\HC[u][z]\|^2_{\Gamma}
        {}+{}
    \tfrac12 \|z'-z\|^2_{\Gamma^{-1}} 
        {}-{}
    \tfrac12 \|\MMC[u][z] - \MMC[u][z']\|^2_{\Gamma}
    \\
    \dueto{\Cref{ass:M:Lips}}
        {}\geq{}&
    \tfrac12\|\HC[u][z']-\HC[u][z]\|^2_{\Gamma}
    \numberthis\label{eq:AFBA:cocoercive}
\end{align*}
That $\HC[u]$ is strongly monotone follows from Cauchy-Schwarz and \Cref{ass:M:Lips},
\begin{equation}
\langle \HC[u][z'] - \HC[u][z], z'-z \rangle  
			{}={} 
		\|z'-z\|^2_{\Gamma^{-1}} - \langle \MMC[u][z'] - \MMC[u][z], z'-z\rangle
			{}\geq{} 
		(1-L_M)\|z'-z\|^2_{\Gamma^{-1}}.
\end{equation}
The last claim follows from Cauchy-Schwarz and dividing by $\|z'-z\|_{\Gamma^{-1}}$.
\end{proof}
We will rely on the resolvent remaining nonexpansive when preconditioned with a variable stepsize matrix.
\begin{lem}\label{lm:GammaA:nonexpansive}
Let $\Gamma \in \R^{n\times n}$ be positive definite and the operator $A: \R^n \rightrightarrows \R^n$ be maximally monotone. 
Then, $R = (\Gamma^{-1} + A)^{-1}$ is nonexpansive, i.e.
$\|Rx - Ry\|_{\Gamma^{-1}} \leq \|x-y\|_{\Gamma}$ for all $x,y \in \R^n$.
\end{lem}
\begin{proof}
Let $v \in Rx$ and $u \in Ry$.
By maximal monotonicity of \(A\),
\begin{align*}
    0
        {}\leq{}
    \langle v-\Gamma^{-1}x-u+\Gamma^{-1}y,x-y\rangle
        {}={}
    -\|x-y\|_{\Gamma^{-1}}^{2}+\langle v-u,x-y\rangle.
\end{align*}
Therefore, using the Cauchy–Schwarz inequality 
\begin{align}
    \|x-y\|_{\Gamma^{-1}}^{2}\leq\langle v-u,x-y\rangle
        {}\leq{}
    \|x-y\|_{\Gamma^{-1}}\|v-u\|_{\Gamma}
    \label{eq:nonexpansicez}
\end{align}
The proof is complete by rearranging.
\end{proof}
}

\subsection{Stochastic results}

The stochastic assumptions on $\hat F$ in \Cref{app:thm:AsymPrecon:convergence} propagates to $\SMC[u]$ and $\SQC[u]$ as captured by the following lemma.

\begin{lem}
\label{lem:MQ:noise}
Suppose \Cref{ass:AsymPrecon:stoch:unbiased,ass:PD:stoch:boundedvar} for $\hat F(z, \xi) = (\nabla_x \hat\varphi(z,\xi), -\nabla_y \hat\varphi(z,\xi))$ as defined in \eqref{eq:AnF}.
Let $\SMC[u]$ and $\MC[u]$ be as defined in \eqref{eq:AsymPrecon:M} and $\SQC[u]$ and $\QC[u]$ as in \eqref{eq:Q:def} with $\theta \in [0,\infty)$.
Then, the following holds for all $z,z' \in \mathbb R^n$
\begin{enumerate}
\item\label{lem:MQ:noise:unbiased} $\mathbb E_{\xi}[\SMC[z][z][\xi]] = \MC[z][z]$ and 
      $\mathbb E_{\xi}[\SQC[z'][z][\xi]] = \QC[z'][z]$
\item\label{lem:MQ:noise:boundvar} $\mathbb E_{\xi}[\|\MC[z][z] - \SMC[z][z][\xi]\|^2_\Gamma] \leq ((1-\theta)^2 + \theta^2)\sigma_F^2$ and 
      $\mathbb E_{\xi}[\|\QC[z'][z] - \SQC[z'][z][\xi]\|^2_\Gamma] \leq \theta^2\sigma_F^2$.
\end{enumerate}
\end{lem}
\begin{proof}
Unbiasedness follows immediately through \Cref{ass:AsymPrecon:stoch:unbiased}.
For the second claim we have for all $(x,y)=z \in \R^n$
\begin{equation}
\begin{split}
\mathbb E_{\xi}[\|\MC[z][z] - \SMC[z][z][\xi]\|^2_\Gamma]
&{}= 
    \mathbb E_{\xi}\left[ \left\|
        \begin{pmatrix}
        \nabla_{x}\hat\varphi(z,\xi)-\nabla_{x}\varphi(z) \\
        (1-\theta)(
            \nabla_{y}\hat\varphi(z,\xi)
            - \nabla_{y}\varphi(z'))
        \end{pmatrix}
    \right\|^2_\Gamma\right] \\
&{}= 
    \mathbb E_{\xi}\left[ \left\|
        \begin{pmatrix}
        (1-\theta)(\nabla_{x}\hat\varphi(z,\xi)-\nabla_{x}\varphi(z))
        +\theta(\nabla_{x}\hat\varphi(z,\xi)-\nabla_{x}\varphi(z)) \\
        (1-\theta)(
            \nabla_{y}\hat\varphi(z,\xi)
            - \nabla_{y}\varphi(z))
        \end{pmatrix}
    \right\|^2_\Gamma\right] \\
\dueto{(\Cref{ass:AsymPrecon:stoch:unbiased})}&{}\leq
    (1-\theta)^2\mathbb E_{\xi}\left[ \left\|
        \begin{pmatrix}
        \nabla_{x}\hat\varphi(z,\xi)-\nabla_{x}\varphi(z) \\
            \nabla_{y}\hat\varphi(z,\xi)
            - \nabla_{y}\varphi(z)
        \end{pmatrix}
    \right\|^2_\Gamma\right]
    + \theta^2\mathbb E_{\xi}\left[ \left\|
        \begin{pmatrix}
        \nabla_{x}\hat\varphi(z,\xi)-\nabla_{x}\varphi(z) \\
        0
        \end{pmatrix}
    \right\|^2_\Gamma\right] \\
\dueto{(\Cref{ass:PD:stoch:boundedvar})}&{}\leq
    ((1-\theta)^2 + \theta^2) \sigma^2_F.
\end{split}
\end{equation}
The last claim follows directly through \Cref{ass:PD:stoch:boundedvar}.
This completes the proof.
\end{proof}

\begin{thm}\label{app:thm:AsymPrecon:convergence}
    Suppose that \cref{ass:AsymPrecon:Minty:Struct} to \ref{ass:AsymPrecon:stoch:unbiased} and \ref{ass:PD} hold.
    Moreover, suppose that \(\alpha_k\in(0,1)\), \(\theta\in[0,\infty)\) and for positive parameter $b$ and $\varepsilon$ the following holds,
    \begin{gather}\label{app:thm:AsymPrecon:convergence:conditions}
    \mu \coloneqq \tfrac{1}{1+b}(1 - \tfrac {1}{\varepsilon (1-L_M)^2})
        + \tfrac{2\rho}{\bar \gamma} 
        -\alpha_0 
        - 2\alpha_0 (\hat{c}_1 + 2\hat{c}_2(1+\hat{c}_3))A 
         > 0, \\
    \qquad 
    1 - 4\hat{c}_2\alpha_0 > 0
    \quad \text{and} \quad
    1 - \tfrac {1}{\varepsilon (1-L_M)^2} \geq 0 \notag
    \end{gather}
    where \(\bar \gamma\) denotes the smallest eigenvalue of \(\Gamma\), $A \geq {(1+4\hat{c}_2\alpha_0^2)(\varepsilon + \tfrac 1b(1 - \tfrac {1}{\varepsilon (1-L_M)^2}))}/{(1 - 4\hat{c}_2\alpha_0)}$ and 
    \begin{gather*}
    \hat{c}_1 \coloneqq L^2_{\widehat{xz}}\|\Gamma D_{\widehat{xz}}\|+2(1-\theta)^2L^2_{\widehat{yz}}\|\Gamma D_{\widehat{yz}}\|+2\theta^2L^2_{\widehat{yy}}\|\Gamma_2D_{\widehat{yy}}\|, 
    \quad
    \hat{c}_2 \coloneqq 2\theta^2L^2_{\widehat{yx}}\|\Gamma_1 D_{\widehat{yx}}\|, 
    \quad
    \hat{c}_3 \coloneqq L^2_{\widehat{xz}}\|\Gamma D_{\widehat{xz}}\|, \\
    L_M^2 \coloneqq \max\big\{L_{xx}^2\|D_{xx}\Gamma_1\|+L_{yx}^2\|D_{yx}\Gamma_1\|, \|L_{xy}^2\|D_{xy}\Gamma_2\|+L_{yy}^2\|D_{yy}\Gamma_2\|\big\}.
    \end{gather*}
    Consider the sequence \(\seq{z^k}\) generated by \Cref{alg:NP-PDHG}.
    Then, the following holds for all $z^\star \in \mathcal S^\star$ 
    \begin{equation}\label{app:thm:AsymPrecon:rate}
        \sum_{k=0}^K \tfrac{\alpha_k}{\sum_{j=0}^K \alpha_{j}} \mathbb E[\|\Gamma^{-1}\hat z^{k} - S_{z^k}(\z^k;\z^k)\|^{2}_\Gamma]
            {}\leq{}
        \frac{\mathbb E[\|z^{0}-z^{\star}\|_{\Gamma^{-1}}^{2}] + A \mathbb E[\|\Gamma^{-1}\hat z^{-1}-S_{z^{-1}}(z^{-1};\z^{-1})\|_{\Gamma}^{2}] + C \sigma_F^2\sum_{j=0}^K \alpha_{j}^{2}}{\mu \sum_{j=0}^K \alpha_{j}}
    \end{equation}
    where $C \coloneqq 2(A + \alpha_0 (\varepsilon + \tfrac 1b(1 - \tfrac {1}{\varepsilon (1-L_M)^2})))(\Theta+2\hat{c}_2) + 1 + 2(\hat{c}_1 + 2\hat{c}_2(1+\hat{c}_3))A$ and $\Theta = (1-\theta)^2 + 2\theta^2$.
\end{thm}

\begin{appendixproof}{app:thm:AsymPrecon:convergence}
The proof relies on tracking the two following important operators instead of $F$ and $\hat F$
\begin{equation}
\label{eq:AsymPrecon:M}
\MC[u][z]:=F(z)-\QC[z][z] \quad \text{and} \quad \SMC[u][z]:=\hat F(z,\xi)-\SQC[z][z].
\end{equation}

We will denote  $\hat {\bar{H}}_k := \SPC[k][\z^k][\bar \xi_k] - \hat F(\bar z^k, \bar \xi_k)$, so that $z^{k+1} = z^k - \alpha_k \Gamma (h^k - \hat{\bar H}_k)$.
We will further need the following change of variables to later be able to apply weak MVI (see \Cref{app:bias-correction}):
\begin{equation}
\begin{split}
s^k {}={}& h^k - \SQC[z^k][\z^k][\xi_k'] \\
\hat{\bar{S}}_k {}={}& \hat{\bar{H}}_k - \SQC[z^k][\z^k][\xi_k'] \\
S_u(\z) {}={}& H_u(\z) - \QC[u][\z] \\
S_{u}(z;\z) {}={}& H_{u}(z) - \QC[u][\z]
\end{split}
\end{equation}
where \(\QC[u][z]\) and \(H_u\) are as defined in \Cref{sec:precon}.

In contrast with the unconstrained smooth case we will rely on a slightly different potential function, namely, 
\begin{align*}
    \U_{k+1} 
        {}\coloneqq{}
    \|z^{k+1}-z^{\star}\|^{2}_{\Gamma^{-1}}
    +A_{k+1}\|s^{k}-S_{z^k}(z^{k};\z^{k})\|^{2}_{\Gamma}
    +B_{k+1}\|z^{k+1}-z^{k}\|^{2}_{\Gamma}
\end{align*}
where \(\seq{A_{k}}\) and \(\seq{B_{k}}\) are positive scalar parameters to be identified.

We start by writing out one step of the update
\begin{align}
\|z^{k+1} - z^\star\|^2_{\Gamma^{-1}} 
& = \|z^k - z^\star\|^2_{\Gamma^{-1}} - 2\alpha_k \langle h^k - \hat{\bar H}_k, z^k - z^\star\rangle + \alpha_k^2 \|h^k - \hat{\bar H}_k\|^2_\Gamma \\
& = \|z^k - z^\star\|^2_{\Gamma^{-1}} - 2\alpha_k \langle s^k - \hat{\bar S}_k, z^k - z^\star\rangle + \alpha_k^2 \|s^k - \hat{\bar S}_k\|^2_\Gamma
\label{eq:AFBA:from_kp1_to_k_2}
\end{align}
In the algorithm, $s^k$ estimates $S_{z^k}(z^k;\z^k)$. 
Let us quantify how good this estimation is. 
We will make use of the careful choice of the bias-correction term to shift the noise index by 1 in the second equality.
	\begin{align*}
	s^{k} - S_{z^k}(z^k;\z^k) 
    ={}& \MC[z^k][z^k] + \QC[z^k][\z^k] - \SMC[z^k][z^k][\xi_k] - \SQC[z^k][\z^k][\xi_k'] \\
      &\quad + (1-\alpha_{k}) (h^{k-1} - \Gamma^{-1}z^{k-1} + \SMC[z^{k-1}][z^{k-1}][\xi_k] - \SQC[z^{k-1}][\z^{k-1}][\xi_{k-1}'] + \SQC[z^{k-1}][\z^{k-1}][\xi_{k}'])\\
    ={}& \MC[z^k][z^k] + \QC[z^k][\z^k] - \SMC[z^k][z^k][\xi_k] - \SQC[z^k][\z^k][\xi_k'] \\
        &\quad + (1-\alpha_{k}) (s^{k-1} + \SQC[z^{k-1}][\z^{k-1}][\bar\xi_{k}'] - \Gamma^{-1}z^{k-1} + \SMC[z^{k-1}][z^{k-1}][\xi_k]) \\
    ={}& \MC[z^k][z^k] + \QC[z^k][\z^k] - \SMC[z^k][z^k][\xi_k] - \SQC[z^k][\z^k][\xi_k'] 
        + (1-\alpha_{k}) (s^{k-1} - S_{z^{k-1}}(z^{k-1};\z^{k-1}))\\
        &\quad + (1-\alpha_{k}) (
          \SMC[z^{k-1}][z^{k-1}][\xi_k] - \MC[z^{k-1}][z^{k-1}] 
          + \SQC[z^{k-1}][\z^{k-1}][\bar\xi_{k}'] - \QC[z^{k-1}][\z^{k-1}])
  \end{align*}
Using the shorthand notation 
\begin{align*}
    \tilde s^k 
        {}\coloneqq{}&
    s^k- S_{z^k}(z^k;\z^k),
    \\
    \tilde Q_{z^k}(\z^k, \xi_k^\prime) 
        {}\coloneqq{}&
     \QC[z^k][\z^k] - \SQC[z^k][\z^k][\xi_k'],
     \\
     \tilde M (z^k,\xi_k) 
        {}\coloneqq{}&
    \MC[z^k][z^k] - \SMC[z^k][z^k][\xi_k],
 \end{align*} 
 it follows that
  \begin{align*}
    \|\tilde s^k\|^2_\Gamma 
        {}={}& 
    (1-\alpha_{k})^2 \|\tilde s^{k-1}\|^2_\Gamma 
        {}+{}
    \|\tilde M(z^k, \xi_k)  + \tilde Q_{z^k}(\z^k, \xi_k^\prime)  
        {}-{}
    (1-\alpha_{k}) (
        \tilde M (z^{k-1}, \xi_k)
        +
        \tilde Q_{z^{k-1}}(\z^{k-1}, \xi_k^\prime)  
        \|^2_\Gamma  
  \\
        &{}+{}
    2 (1-\alpha_{k}) 
    \langle \tilde s^{k-1}, \tilde M(z^k, \xi_k)  + \tilde Q_{z^k}(\z^k, \xi_k^\prime) 
        {}-{}
    (1-\alpha_{k}) \big(%
        \tilde M(z^{k-1}, \xi_k) 
            {}+{}
        \tilde Q_{z^{k-1}}(\z^{k-1}, \xi_k^\prime)
        \big)
     \rangle
     \numberthis\label{eq:sk:1}
    \end{align*}
	In the scalar product,  the left term is known when $z^k$ is known. Moreover, since \(
    \mathbb E\big[ \cdot \mid {\F}_k\big] 
        {}={}
    \mathbb E\big[ \mathbb E\big[ \cdot \mid {\F}_k^\prime\big] \mid {\F}_k\big] 
\), owing to \(\F_k \subset {\F}^\prime_{k}\), we have 
\begin{align*}
    \mathbb E\Big[%
        \tilde M(z^k, \xi_k)  + \tilde Q_{z^k}(\z^k, \xi_k^\prime) 
        {}-{}
    (1-\alpha_{k})& \big(%
        \tilde M(z^{k-1}, \xi_k) 
            {}+{}
        \tilde Q_{z^{k-1}}(\z^{k-1}, \xi_k^\prime)
        \big)
        \mid \mathcal F_k
        \Big] \\
            &{}={}
        \mathbb E\left[%
        \tilde M(z^k, \xi_k)
        - (1-\alpha_{k})\tilde M(z^{k-1}, \xi_k) 
        \mid \mathcal F_k
        \right]
            {}={}
        0,
\end{align*}
where we use \Cref{ass:AsymPrecon:stoch:unbiased} through \Cref{lem:MQ:noise:unbiased}.

 Since the second moment is larger than the variance we have
    \begin{align*}
        \mathbb{E}&\left[\|\tilde{M}(z^{k},\xi_{k})+\tilde{Q}_{z^{k}}(\z^{k},\xi_{k}^{\prime})-\tilde{M}(z^{k-1},\xi_{k})-\tilde{Q}_{z^{k-1}}(\z^{k-1},\xi_{k}^{\prime})\|_{\Gamma}^{2}\mid\mathcal{F}_{k}\right]
            {}\leq{}
        \\   
        &\quad  \mathbb{E}\left[\| \SMC[z^k][z^k][\xi_k] - \SMC[z^{k-1}][z^{k-1}][\xi_k] + \SQC[z^k][\z^k][\xi_k'] - \SQC[z^{k-1}][\z^{k-1}][\xi_{k}'] \|^2_\Gamma \; |\; \mathcal F_k\right]
        \numberthis\label{eq:secondmoment}
    \end{align*}
    Using the Young inequality it follows from \eqref{eq:sk:1}, \eqref{eq:secondmoment} that
    \begin{align*}
    \mathbb E[\|\tilde s^{k}\|^2_\Gamma \mid \mathcal F_k] 
        {}\leq{}&
    (1-\alpha_{k})^2 \|\tilde s^{k-1}\|^2_\Gamma 
        {}+{}
    2\alpha_{k}^2 \mathbb E[\|\tilde{M}(z^{k},\xi_{k})+\tilde{Q}_{z^{k}}(\z^{k},\xi_{k}^{\prime})\|^2_\Gamma  \; |\; \mathcal F_k] 
    \\
        &{}+{} 
    2(1-\alpha_{k})^2 \mathbb E[ \| \tilde{M}(z^{k},\xi_{k})+\tilde{Q}_{z^{k}}(\z^{k},\xi_{k}^{\prime})-\tilde{M}(z^{k-1},\xi_{k})-\tilde{Q}_{z^{k-1}}(\z^{k-1},\xi_{k}^{\prime}) \|^2_\Gamma \mid \mathcal F_k] 
    \\
        {}\leq{}&
    (1-\alpha_{k})^2 \|\tilde s^{k-1}\|^2_\Gamma 
        {}+{}
    2\alpha_{k}^2 \mathbb E[\|\MC[z^k][z^k] - \SMC[z^k][z^k][\xi_k] + \QC[z^k][\z^k] - \SQC[z^k][\z^k][\xi_k']\|^2_\Gamma  \mid \mathcal F_k] 
    \\
        &{}+{}
    \mathbb E[ 2(1-\alpha_{k})^2 \| \SMC[z^k][z^k][\xi_k] - \SMC[z^{k-1}][z^{k-1}][\xi_k] + \SQC[z^k][\z^k][\xi_k'] - \SQC[z^{k-1}][\z^{k-1}][\xi_{k}'] \|^2_\Gamma \; |\; \mathcal F_k] 
    \numberthis\label{eq:AFBA:approx_Hzk}
    \end{align*}
To bound the second last term of \eqref{eq:AFBA:approx_Hzk} we use unbiasedness due to \Cref{ass:AsymPrecon:stoch:unbiased} through \Cref{lem:MQ:noise:unbiased} and that  \(
    \mathbb E\big[ \cdot \mid {\F}_k\big] 
        {}={}
    \mathbb E\big[ \mathbb E\big[ \cdot \mid {\F}_k^\prime\big] \mid {\F}_k\big] 
\), owing to \(\F_k \subset {\F}^\prime_{k}\)
\begin{align*}
\mathbb E[\|\MC[z^k][z^k] - \SMC[z^k][z^k][\xi_k] &+ \QC[z^k][\z^k] - \SQC[z^k][\z^k][\xi_k']\|^2_\Gamma  \mid \mathcal F_k] \\
&= \mathbb E[\|\MC[z^k][z^k] - \SMC[z^k][z^k][\xi_k]\|^2_\Gamma  \mid \mathcal F_k] 
        {}+{}
    \mathbb E[\mathbb E[\|\QC[z^k][\z^k] - \SQC[z^k][\z^k][\xi_k']\|^2_\Gamma  \mid \mathcal F_k'] \mid \mathcal F_k] \\
& \leq \Theta \sigma_F^2
\numberthis\label{eq:AFBA:approx_Hzk_noise}
\end{align*}
with $\Theta = (1-\theta)^2 + 2\theta^2$.
where the last inequality follows from \Cref{ass:AsymPrecon:stoch:unbiased,ass:PD:stoch:boundedvar} through \Cref{lem:MQ:noise:boundvar}.

To bound the last term of \eqref{eq:AFBA:approx_Hzk} we use the particular choice of $\QC[u]$,
\begin{equation}
\begin{split}
\SMC[z^k][z^k][\xi_k]& - \SMC[z^{k-1}][z^{k-1}][\xi_k] + \SQC[z^k][\z^k][\xi_k'] - \SQC[z^{k-1}][\z^{k-1}][\xi_{k}'] \\
    &= 
    \begin{pmatrix}
    \nabla_{x}\hat\varphi(z^k,\xi_k)-\nabla_{x}\hat\varphi(z^{k-1},\xi_k) \\
    (1-\theta)(
        \nabla_{y}\hat\varphi(z^{k-1}\xi_k)
        - \nabla_{y}\hat\varphi(z^{k},\xi_k))
    -\theta(
        \nabla_{y}\hat\varphi(\bar x^{k},y^{k},\xi_k')
        - \nabla_{y}\hat\varphi(\bar x^k,y^k,\xi_k')
    \end{pmatrix}.
\end{split}
\end{equation}
So \Cref{ass:PD:Lipmean:phi} applies after application of Young's inequality and the tower rule, leading to the following bound
\begin{align*}
\mathbb E[\|\SMC[z^k][z^k][\xi_k]& - \SMC[z^{k-1}][z^{k-1}][\xi_k] + \SQC[z^k][\z^k][\xi_k'] - \SQC[z^{k-1}][\z^{k-1}][\xi_k']\|^2_\Gamma  \mid \mathcal F_k] \\
    &{}=
        \mathbb E[\|\nabla_{x}\hat\varphi(z^k,\xi_k)-\nabla_{x}\hat\varphi(z^{k-1},\xi_k)\|^2_{\Gamma_1}  \mid \mathcal F_k] \\
        &\quad +\mathbb E[\|(1-\theta)(\nabla_{y}\hat\varphi(z^{k-1}\xi_k) - \nabla_{y}\hat\varphi(z^{k},\xi_k))
        -\theta \nabla_{y}\hat\varphi(\bar x^{k},y^{k},\xi_k') - \nabla_{y}\hat\varphi(\bar x^k,y^k,\xi_k')\|^2_{\Gamma_2}  \mid \mathcal F_k] \\
    &{}\leq
        \mathbb E[\|\nabla_{x}\hat\varphi(z^k,\xi_k)-\nabla_{x}\hat\varphi(z^{k-1},\xi_k)\|^2_{\Gamma_1}   \mid \mathcal F_k]\\
        &\quad +2(1-\theta)^2\mathbb E[\|(\nabla_{y}\hat\varphi(z^{k-1}\xi_k) - \nabla_{y}\hat\varphi(z^{k},\xi_k))\|^2_{\Gamma_2}   \mid \mathcal F_k]\\
        &\quad +2\theta^2\mathbb E\left[\mathbb E[\|\nabla_{y}\hat\varphi(\bar x^{k},y^{k},\xi_k') - \nabla_{y}\hat\varphi(\bar x^k,y^k,\xi_k')\|^2_{\Gamma_2} \mid \mathcal F_k'] \mid \mathcal F_k\right]\\
    \dueto{\Cref{ass:PD:Lipmean:phi}}{}&{}\leq 
        L^2_{\widehat{xz}}\|z^{k} - z^{k-1}\|_{D_{\widehat{xz}}}^{2}
        + 2(1-\theta)^2L^2_{\widehat{yz}}\|z^{k} - z^{k-1}\|_{D_{\widehat{yz}}}^{2} \\
        &\quad + 2\theta^2L^2_{\widehat{yy}}\|y^{k} - y^{k-1}\|_{D_{\widehat{yy}}}^{2}
        + 2\theta^2L^2_{\widehat{yx}}\|\bar x^{k} - \bar x^{k-1}\|_{D_{\widehat{yx}}}^{2} \\
    &{}\leq \hat{c}_1\|z^{k} - z^{k-1}\|_{\Gamma^{-1}}^{2}
        + \hat{c}_2\|\bar x^{k} - \bar x^{k-1}\|_{\Gamma^{-1}}^{2}
\numberthis\label{eq:AFBA:lipsmean}
\end{align*}
where $\hat{c}_1 \coloneqq L^2_{\widehat{xz}}\|\Gamma D_{\widehat{xz}}\|+2(1-\theta)^2L^2_{\widehat{yz}}\|\Gamma D_{\widehat{yz}}\|+2\theta^2L^2_{\widehat{yy}}\|\Gamma_2D_{\widehat{yy}}\|$ and
$\hat{c}_2 \coloneqq 2\theta^2L^2_{\widehat{yx}}\|\Gamma_1 D_{\widehat{yx}}\|$.

Using \eqref{eq:AFBA:lipsmean} and \eqref{eq:AFBA:approx_Hzk_noise} in \eqref{eq:AFBA:approx_Hzk} yields,
\begin{equation}
\begin{split}
\mathbb E[\|\tilde s^{k}\|^2_\Gamma \mid \mathcal F_k] 
&{}\leq{}
    (1-\alpha_{k})^2 \|\tilde s^{k-1}\|^2_\Gamma + 2\alpha_{k}^2 \Theta\sigma_F^2 
        {}+{}
    2(1-\alpha_{k})^2 \left( \hat{c}_1 \|z^{k} - z^{k-1}\|_{\Gamma^{-1}}^{2}
        + \hat{c}_2 \|\bar x^{k} - \bar x^{k-1}\|_{\Gamma_1^{-1}}^{2}\right).
    \label{eq:AFBA:approx_Hzk:2}
\end{split}
\end{equation}

    To majorize $\|\bar x^k - \bar x^{k-1}\|_{\Gamma_1^{-1}}$ in \eqref{eq:AFBA:approx_Hzk:2} let $s^k_x$ be the primal components of $s^k$ in what follows. 
    Recall that $A$ decomposes as specified in \Cref{sec:precon}, such that we can write \(s^k_x \in \Gamma^{-1}_1\bar x^k + A_1(\bar x^k)\).
    By monotonicity of \(A_1\) we have through \Cref{lm:GammaA:nonexpansive} that
    \begin{equation}
    \label{eq:nonexpansicez}
    \|\bar{x}^{k}-\bar{x}^{k-1}\|_{\Gamma^{-1}_1} \leq \|s^{k}_x-s^{k-1}_x\|_{\Gamma_1}.
    \end{equation}
    We can go on as
    \begin{align*}
    \|s^k_x - s^{k-1}_x\|_{\Gamma_1}
    &{}= \|\Gamma_1^{-1}x^k - \nabla_x \hat \phi(z^k, \xi_k) + (1-\alpha_{k}) \big( \Gamma_1^{-1}(x^{k-1} - x^{k-1}) + \hat \nabla_x \hat \phi(z^{k-1}, \xi_k)\big) - s^{k-1}_x\|_{\Gamma_1}
    \\
    &{}\leq (1-\alpha_{k})\|x^k - x^{k-1}\|_{\Gamma_1^{-1}} 
        + (1-\alpha_{k}) \| \nabla_x \hat \phi(z^k, \xi_k) - \nabla_x \hat \phi(z^{k-1}, \xi_k)\|_{\Gamma_1} \\
        &\qquad + \alpha_{k}\|\Gamma_1^{-1}x^k - \nabla_x \hat \phi(z^k, \xi_k) - s^{k-1}_x\|_{\Gamma_1} \\
    \dueto{(\Cref{ass:PD:Lipmean:phi})}&{}\leq (1-\alpha_{k})\|x^k - x^{k-1}\|_{\Gamma_1^{-1}} + (1-\alpha_{k})L_{\widehat{xz}} \| z^k - z^{k-1}\|_{D_{\widehat{xz}}} \\
    &\qquad + \alpha_{k}\|\Gamma_1^{-1}x^k - \nabla_x \hat \phi(z^k, \xi_k) - s^{k-1}_x\|_{\Gamma_1} \\
    &{}= (1-\alpha_{k})\|x^k - x^{k-1}\|_{\Gamma_1^{-1}} + (1-\alpha_{k})L_{\widehat{xz}} \| z^k - z^{k-1}\|_{D_{\widehat{xz}}} \\
    &{} \qquad+ \alpha_{k}\|s^k_x - s^{k-1}_x\|_{\Gamma_1^{-1}} + \alpha_{k} (1-\alpha_{k})\|\Gamma_1^{-1}x^{k-1} - \nabla_x \hat \phi(z^{k-1}, \xi_k) - s^{k-1}_x\|_{\Gamma_1},
    \end{align*}
    where the last equality uses $\|a-b\|^2 = \|a\|^2 + \|b\|^2 - 2\langle a,b\rangle$ and unbiasedness from \Cref{ass:AsymPrecon:stoch:unbiased} to conclude that the inner product is zero.

    Hence, by subtracting $\alpha_{k}\|s^k_x - s^{k-1}_x\|_{\Gamma_1^{-1}}$ and diving by $1-\alpha_k$, we get
    \begin{align*}
    \mathbb E[\|s^k_x - s^{k-1}_x\|^2_{\Gamma_1^{-1}}\;|\; \F_k] 
        &{}\leq 2(1+\hat{c}_3)\|x^k - x^{k-1}\|^2_{\Gamma_1^{-1}} + 2 \alpha_{k}^2 \mathbb E[\|\Gamma^{-1}x^{k-1} - \nabla_x \hat \phi(z^{k-1}, \xi_k) - s^{k-1}_x\|_{\Gamma_1}^2 \;|\; \F_k] \\
        \dueto{\Cref{ass:AsymPrecon:stoch:unbiased,ass:PD:stoch:boundedvar}}&{}\leq 2(1+\hat{c}_3)\|x^k - x^{k-1}\|^2_{\Gamma_1^{-1}} + 2 \alpha_{k}^2 \mathbb E[\|\Gamma^{-1}x^{k-1} - \nabla_x \phi(z^{k-1}) - s^{k-1}_x\|^2_{\Gamma_1} \;|\; \F_k] +  2 \alpha_{k}^2 \sigma^2_F \\
        &{}\leq 2(1+\hat{c}_3)\|z^k - z^{k-1}\|^2_{\Gamma^{-1}} + 2 \alpha_{k}^2 \mathbb E[\|S_{z^{k-1}}(z^{k-1};\z^{k-1})-s^{k-1}\|^2_{\Gamma} \;|\; \F_k] +  2 \alpha_{k}^2 \sigma^2_F
    \end{align*}
    where $\hat{c}_3 \coloneqq L^2_{\widehat{xz}}\|\Gamma D_{\widehat{xz}}\|$ and the last inequality reintroduces the $y$-components.

  We finally obtain
 \begin{equation}
 \label{eq:AFBA:zbar}
  \begin{split}
  \mathbb E[ \|\bar x^k - \bar x^{k-1}\|^2_{\Gamma^{-1}}  \;|\; \F_k]
  &{}\leq 2(1+\hat{c}_3)\|z^k-z^{k-1}\|^2_{\Gamma^{-1}}
    + 2\alpha_{k}^2 \mathbb E[  \|s^{k-1} - S_{z^{k-1}}(z^{k-1};\z^{k-1})\|^2_\Gamma \;|\; \F_k]
    + 2\alpha_{k}^2 \sigma_F.
  \end{split}
  \end{equation}
  Introducing \eqref{eq:AFBA:zbar} into \eqref{eq:AFBA:approx_Hzk:2} yields
  \begin{equation}
  \label{eq:AFBA:approx_Hzk:3}
  \begin{split}
  \mathbb E[\|s^{k} - S_{z^k}(z^k;\z^k) \|^2_\Gamma \;|\; \mathcal F_k]
	&\leq (1-\alpha_{k})^2(1+4\hat{c}_2\alpha_{k}^2) \|s^{k-1} - S_{z^{k-1}}(z^{k-1};\z^{k-1})\|^2_\Gamma \\
  &\qquad + 2(1-\alpha_{k})^2 (\hat{c}_1 + 2\hat{c}_2(1+\hat{c}_3)) \|z^k - z^{k-1}\|^2_{\Gamma^{-1}} \\
  &\qquad + 2\alpha_{k}^2(\Theta+(1-\alpha_{k})^22\hat{c}_2) \sigma_F^2.
  \end{split}
  \end{equation}

	We continue with the inner term in \eqref{eq:AFBA:from_kp1_to_k_2} under conditional expectation. 
	\begin{align*}
	-\mathbb E[&\langle s^k - \hat{\bar S}_k,  z^k - z^\star\rangle_\Gamma \;|\; \mathcal F_k] \\
    &= -\langle s^k - S_{z^k}(\z^k), z^k - z^\star\rangle \\
	& = -\langle s^k - S_{z^k}(\z^k), z^k - \bar z^k\rangle - \langle s^k - S_{z^k}(\z^k), \bar z^k - z^\star\rangle \\
	& = -\langle s^k - S_{z^k}(z^k;\z^k), z^k - \bar z^k\rangle -  \langle S_{z^k}(z^k;\z^k) - S_{z^k}(\z^k), z^k - \bar z^k\rangle - \langle s^k - S_{z^k}(\z^k), \bar z^k - z^\star\rangle \\
	& = -\langle s^k - S_{z^k}(z^k;\z^k), z^k - \bar z^k\rangle -  \langle H_{z^k}(z^k) - H_{z^k}(\bar z^k), z^k - \bar z^k\rangle - \langle s^k - S_{z^k}(\z^k), \bar z^k - z^\star\rangle 
	\end{align*}
where the last equality uses that $S_{z^k}(z^k;\z^k) - S_{z^k}(\z^k)=H_{z^k}(z^k) - H_{z^k}(\bar z^k)$.

By definition of $\bar z^k$ in \eqref{eq:AFBA:bar}, we have $s^k = h^k - \SQC[z^k][\z^k][\xi_k'] \in \Gamma^{-1}\z^k + A(\bar z^k)$, so that $s^k - S_{z^k}(\z^k) \in F(\bar z^k) + A(\bar z^k)$. Hence, using the weak MVI from \Cref{ass:AsymPrecon:Minty:Struct},
\begin{equation}\label{eq:AsymPrecon:minty}
\langle s^k - S_{z^k}(\z^k), \bar z^k - z^\star\rangle
\geq \rho \|s^k - S_{z^k}(\z^k)\|^2 \; .
\end{equation}
Using also cocoercivity of $\HC[u]$ from \Cref{lm:H:cocoercive}, this leads to the following inequality, true for any $\varepsilon_k>0$:
\begin{align*}
-\mathbb E[\langle s^k - \hat{\bar S}_k,  z^k - z^\star\rangle \;|\; \mathcal F_k] 
&{}\leq \tfrac{\varepsilon_k}{2} \|s^k - S_{z^k}(z^k;\z^k)\|^2_\Gamma 
+ \tfrac 1{2\varepsilon_k} \|\bar z^k - z^k\|^2_{\Gamma^{-1}} \\
&\qquad - \tfrac 12 \|H_{z^k}(z^k)-H_{z^k}(\bar z^k)\|^2_\Gamma 
- \rho \|s^k - S_{z^k}(\z^k)\|^2\;.
\end{align*}
To majorize the term $\|\bar z^k - z^k\|^2_{\Gamma^{-1}}$, we may use \Cref{lm:H:strmonotone}
for which we need to determind $L_M$.
\rbl{
For the particular choice of $\QC[u]$, we have through \Cref{ass:PD:Lip:phi} that
\begin{equation}
\|\MC[u][z']-\MC[u][z]\|^2_{\Gamma} \leq L_M^2 \|z'-z\|^2_{\Gamma^{-1}}
\end{equation}
with $L_M^2 \coloneqq \max\big\{L_{xx}^2\|D_{xx}\Gamma_1\|+L_{yx}^2\|D_{yx}\Gamma_1\|, \|L_{xy}^2\|D_{xy}\Gamma_2\|+L_{yy}^2\|D_{yy}\Gamma_2\|\big\}$.
By the stepsize choice \Cref{ass:PD:stepsize}, $L_M < 1$, which will be important promptly.
}

\rbl{
From \Cref{lm:H:strmonotone} it then follows that
\begin{align*}
\|H_{z^k}(z^k) - H_{z^k}(\z^k)\|^2_\Gamma \geq (1- L_M)^2\|z^k - \z^k\|^2_{\Gamma^{-1}} \;.
\end{align*}
Hence, given $L_M < 1$, 
}
\begin{align*}
-\mathbb E[\langle s^k &- \hat{\bar S}_k,  z^k - z^\star\rangle_\Gamma \;|\; \mathcal F_k]  \\
&\leq \tfrac{\varepsilon_k}{2}  \|s^k - S_{z^k}(z^k;\z^k)\|^2_\Gamma
   + \Big(\tfrac {1}{2\varepsilon_k (1-L_M)^2} - \tfrac 12\Big) \|H_{z^k}(z^k)-H_{z^k}(\bar z^k)\|^2_\Gamma 
   - \rho \|s^k - S_{z^k}(\z^k)\|^2 \\
&= \tfrac{\varepsilon_k}{2}  \|s^k - S_{z^k}(z^k;\z^k)\|^2_\Gamma
   + \Big(\tfrac {1}{2\varepsilon_k (1-L_M)^2} - \tfrac 12\Big) \|S_{z^k}(z^k;\z^k) - S_{z^k}(\z^k)\|^2_\Gamma 
   - \rho \|s^k - S_{z^k}(\z^k)\|^2.
\numberthis\label{eq:AFBA:second_term}
\end{align*}
The conditional expectation of the third term in \eqref{eq:AFBA:from_kp1_to_k_2} is bounded by
\begin{align*}
\alpha_k^2 \mathbb E[  \|s^k - \hat{\bar S}_k\|^2_\Gamma \; |\; \mathcal F_k] 
  &{}= 
    \alpha_k^2\|s^k - S_{z^k}(\bar z^k)\|^2_\Gamma 
    + \alpha_k^2 \mathbb E[  \|
      F(\bar z^k) - \hat F(\bar z^k,\bar \xi_k)
    \|^2_\Gamma \; |\; \mathcal F_k] \\
  &{}\leq 
    \alpha_k^2\|s^k - S_{z^k}(\bar z^k)\|^2_\Gamma + \alpha_k^2 \sigma_F^2
\numberthis\label{eq:AFBA:third_term}
\end{align*}
where we have used \Cref{ass:PD:stoch:boundedvar}.

Combined with the update rule, \eqref{eq:AFBA:third_term} can also be used to bound the conditional expectation of the difference of iterates
\begin{equation}\label{eq:AFBA:iterate_diff}
\mathbb E[\|z^{k+1} - z^k\|^2_{\Gamma^{-1}} \; |\; \mathcal F_k]
= \mathbb E[\alpha_k^2\|s^k - \hat{\bar S}_k\|^2_\Gamma \; |\; \mathcal F_k]
\leq \alpha_k^2\|s^k - S_{z^k}(\z^k)\|^2_\Gamma + \alpha_k^2\sigma_F^2
\end{equation}
Using \eqref{eq:AFBA:from_kp1_to_k_2}, \eqref{eq:AFBA:second_term}, \eqref{eq:AFBA:third_term}, \eqref{eq:AFBA:iterate_diff} and that $-\rho\|s^k - S_{z^k}(\bar z^k)\|^2 \leq -\frac{\rho}{\bar \gamma}\|s^k - S_{z^k}(\bar z^k)\|^2_\Gamma$ with $\bar\gamma$ denoting the smallest eigenvalue of $\Gamma$ we have,
\begin{align*}
\mathbb E[\U_{k+1} \; |\; \mathcal F_k] 
& \leq 
    \|z^k - z^\star\|^2_{\Gamma^{-1}} + (A_{k+1} + \alpha_k \varepsilon_k)  \|s^k - S_{z^k}(z^k;\z^k)\|^2_{\Gamma} - \alpha_k \Big(1 - \tfrac {1}{\varepsilon_k (1-L_M)^2}\Big) \|S_{z^k}(z^k;\z^k) - S_{z^k}(\z^k)\|^2_{\Gamma} \\
    & \qquad + \alpha_k(\alpha_k - \tfrac{2\rho}{\bar \gamma} + \alpha_k B_{k+1}) \|s^k - S_{z^k}(\bar z^k)\|^2_\Gamma + \alpha_k^2(1+B_{k+1})\sigma_F^2 \\
& \leq 
    \|z^k - z^\star\|^2_{\Gamma^{-1}} 
    + \big(A_{k+1} + \alpha_k (\varepsilon_k + \tfrac 1b(1 - \tfrac {1}{\varepsilon_k (1-L_M)^2}))\big)  \|s^k - S_{z^k}(z^k;\z^k)\|^2_{\Gamma} \\
    &\qquad + \alpha_k \Big(\alpha_k - \tfrac{2\rho}{\bar \gamma} + \alpha_k B_{k+1} - \tfrac{1}{1+b}(1-\tfrac {1}{\varepsilon_k (1-L_M)^2})\Big) \|H_{z^k}(z^k)-H_{z^k}(\bar z^k)\|^2_{\Gamma} \\
    & \qquad + \alpha_k^2(1+B_{k+1})\sigma_F^2,
\numberthis
\end{align*}
where the last inequality follows from Young's inequality with positive $b$ as long as $1 - \tfrac {1}{\varepsilon_k (1-L_M)^2} \geq 0$.

By defining
\begin{equation}
\begin{split}
X_k^1 & \coloneqq A_{k+1} + \alpha_k (\varepsilon_k + \tfrac 1b(1 - \tfrac {1}{\varepsilon_k (1-L_M)^2})) \\
X_k^2 & \coloneqq \tfrac{2\rho}{\bar \gamma} - \alpha_k - \alpha_k B_{k+1} + \tfrac{1}{1+b}(1 - \tfrac {1}{\varepsilon_k (1-L_M)^2})
\end{split}
\end{equation}
and applying \eqref{eq:AFBA:approx_Hzk:3}, we finally obtain 
\begin{equation}
\begin{split}
\label{eq:AFBA:descent}
\mathbb E[\U_{k+1} \; |\; \mathcal F_k]  - \U_{k}
& \leq 
    -\alpha_k X_k^2 \|s^k - S_{z^k}(\z^k)\|^2_{\Gamma} \\
    & \qquad + (X_k^1(1-\alpha_{k})^2(1+4\hat{c}_2\alpha_{k}^2) - A_k) \|s^{k-1} - S_{z^{k-1}}(z^{k-1};\z^{k-1})\|^2_\Gamma  \\
    & \qquad + (2X_k^1(1-\alpha_{k})^2(\hat{c}_1 + 2\hat{c}_2(1+\hat{c}_3)) - B_k) \|z^k - z^{k-1}\|^2_{\Gamma^{-1}}  \\
    & \qquad + 2X_k^1\alpha_{k}^2(\Theta+(1-\alpha_{k})^22\hat{c}_2) \sigma_F^2 + \alpha_k^2(1+B_{k+1})\sigma_F^2,
\end{split}
\end{equation}
If $A_k \geq X_k^1(1-\alpha_{k})^2(1+4\hat{c}_2\alpha_{k}^2)$,
then it suffice to pick $B_k$ as
\begin{equation}
2X_k^1(1-\alpha_{k})^2(\hat{c}_1 + 2\hat{c}_2(1+\hat{c}_3))
  = \tfrac{2(\hat{c}_1 + 2\hat{c}_2(1+\hat{c}_3))A_k}{1+4\hat{c}_2\alpha_{k}^2}
  \leq 2(\hat{c}_1 + 2\hat{c}_2(1+\hat{c}_3))A_k =: B_k.
\end{equation}
To get a recursion, we then only require the following conditions
\begin{equation}
\label{eq:AFBA:recursion}
X_k^1(1- \alpha_{k})^2(1+4\hat{c}_2\alpha_{k}^2) \leq A_k 
\quad \text{and} \quad
X_k^2 > 0.
\end{equation}
Set $A_k=A$, $\varepsilon_k = \varepsilon$.
For the first inequality of \eqref{eq:AFBA:recursion}, since $\left(1-\alpha_{k}\right)^2 \leq\left(1-\alpha_{k}\right)$, 
the terms involving $A$ are bounded as
\begin{align*}
(1-\alpha_{k})^2&(1+4\hat{c}_2\alpha_k^2)A - A\\
&\leq (1-\alpha_{k})(1+4\hat{c}_2\alpha_k^2)A - A\\
&= -\alpha_k A + (1-\alpha_{k})(4\hat{c}_2\alpha_k^2)A \\
&\leq -\alpha_k\left(1 - 4\hat{c}_2\alpha_0\right)A
\numberthis
\end{align*}
where the last inequality follows from $\left(1-\alpha_{k}\right) \leq 1$ and $\alpha_k \leq \alpha_0$.
Thus to satisfy the first inequality of \eqref{eq:AFBA:recursion} it suffice to pick
\begin{equation}
A \geq \frac{(1+4\hat{c}_2\alpha_0^2)(\varepsilon + \tfrac 1b(1 - \tfrac {1}{\varepsilon (1-L_M)^2}))}{1 - 4\hat{c}_2\alpha_0}
\end{equation}
where $1 - 4\hat{c}_2\alpha_0 > 0$ is required. %

The second equality of \eqref{eq:AFBA:recursion} is satisfied owing to \eqref{app:thm:AsymPrecon:convergence:conditions}.

The noise term in \eqref{eq:AFBA:descent} can be made independent of $k$ by using $\alpha_k \leq \alpha_0$,
\begin{align*}
2X_k^1&(1+(1-\alpha_k)^22\hat{c}_2) + 1+B_{k+1} \\
  &= 2(A + \alpha_k (\varepsilon + \tfrac 1b(1 - \tfrac {1}{\varepsilon (1-L_M)^2})))(\Theta+(1-\alpha_k)^22\hat{c}_2)
    + 1 + 2(\hat{c}_1 + 2\hat{c}_2(1+\hat{c}_3))A \\
  &\leq 2(A + \alpha_0 (\varepsilon + \tfrac 1b(1 - \tfrac {1}{\varepsilon (1-L_M)^2})))(\Theta+2\hat{c}_2)
    + 1 + 2(\hat{c}_1 + 2\hat{c}_2(1+\hat{c}_3))A
    =: C.
\numberthis
\end{align*}

Thus, it follows from \eqref{eq:AFBA:descent} that
\begin{equation}
\begin{split}
\mathbb E[\U_{k+1} &\; |\; \mathcal F_k]  - \U_{k} \\
 &{}\leq 
    \alpha_k \Big(\alpha_0 - \tfrac{2\rho}{\bar \gamma} + 2\alpha_0 (\hat{c}_1 + 2\hat{c}_2(1+\hat{c}_3))A - \tfrac{1}{1+b}(1 - \tfrac {1}{\varepsilon (1-L_M)^2})\Big) \|s^k - S_{z^k}(\z^k)\|^2_{\Gamma} \\
    & \qquad + \alpha_k^2C\sigma_F^2.
\end{split}
\end{equation}
The result is obtained by total expectation and summing the above inequality while noting that the initial iterate were set as \(z^{-1}= z^0\). 
\end{appendixproof}

\begin{appendixproof}{thm:AsymPrecon:PDHG:convergence}
The theorem is a specialization of \Cref{app:thm:AsymPrecon:convergence} for a particular a choice of $b$ and $\varepsilon$.
The third requirement of \eqref{app:thm:AsymPrecon:convergence:conditions} can be rewritten as,
\begin{equation}
\varepsilon \geq \tfrac{1}{(1-L_M)^2},
\end{equation}
which is satisfied by $\varepsilon = \tfrac{1}{\sqrt{\alpha_0}(1-L_M)^2}$.
We substitute in the choice of $\varepsilon$, $b=\sqrt{\alpha_0}$ and denotes $\eta \coloneqq A$.

The weighted sum in \eqref{app:thm:AsymPrecon:rate} is equivalent to an expectation over a sampled iterate in the style of \citet{ghadimi2013stochastic},
\begin{align*}
    \mathbb E[\|\Gamma^{-1}\hat z^{k_\star} - S_{z^{k_\star}}(\z^{k_\star};\z^{k_\star})\|^{2}_\Gamma]
        {}={}
     \sum_{k=0}^K \tfrac{\alpha_k}{\sum_{j=0}^K \alpha_{j}} \mathbb E[\|\Gamma^{-1}\hat z^{k} - S_{z^k}(\z^k;\z^k)\|^{2}_\Gamma].
\end{align*}
with $k_{\star}$ chosen from $\{0,1, \ldots, K\}$ according to probability $\mathcal{P}\left[k_{\star}=k\right]=\frac{\alpha_k}{\sum_{j=0}^K \alpha_j}$.

\rbl{
Noticing that $\Gamma^{-1}\hat z^{k_\star} - S_{z^{k_\star}}(\z^{k_\star};\z^{k_\star}) \in F\z^{k_\star} + A\z^{k_\star} =  T\z^{k_\star}$ so
\begin{align*}
    \mathbb E[\|\Gamma^{-1}\hat z^{k_\star} - S_{z^{k_\star}}(\z^{k_\star};\z^{k_\star})\|^{2}_\Gamma] 
        \geq \min_{u \in T\z^{k_\star}} \mathbb E[\|u\|_\Gamma^2] 
        \geq \mathbb E[\min_{u \in T\z^{k_\star}} \|u\|_\Gamma^2]
        =: \mathbb E[\dist_\Gamma(0,T\z^{k_\star})^2]
\end{align*}
where the second inequality follows from concavity of the minimum.
}%
This completes the proof.

\end{appendixproof}

    \rbl{
\begin{thm}[almost sure convergence]\label{app:thm:AFBA:almostsure}
    Suppose that \cref{ass:AsymPrecon:Minty:Struct} to \ref{ass:AsymPrecon:stoch:unbiased} and \ref{ass:PD} hold.
    Moreover, suppose that \(\alpha_k\in(0,1)\), \(\theta\in[0,\infty)\) and the following holds for positive parameter $b$,
    \begin{gather}
    \eta_{k} 
        {}\coloneqq{}
    \textstyle
    \sum_{\ell=k}^{\infty}
        \left(%
            c_{l}\alpha_{l}\Pi_{p=0}^{\ell}(1-\alpha_{p})^{2}\tau_p
        \right)
        < \infty, 
    \quad
    \nu \coloneqq \sum_{k=0}^\infty \alpha_{k}^2\eta_{k+1} \Pi_{p=0}^{k}\frac{1}{(1-\alpha_{p})^{2}\tau_p} < \infty \quad \text{and}\label{app:thm:AsymPrecon:almostsure:conditions1}\\
    \tfrac{2\rho}{\bar \gamma} - \alpha_k - \alpha_k 2(\hat{c}_1 + 2\hat{c}_2(1+\hat{c}_3))\eta_{k+1} \Pi_{p=0}^{k}\frac{1}{(1-\alpha_{p})^{2}\tau_p} + \tfrac{1}{1+b}(1 - \alpha_k^d) > 0 \label{app:thm:AsymPrecon:almostsure:conditions2}
    \end{gather}
    where \(\bar \gamma\) denotes the smallest eigenvalue of \(\Gamma\), $d\in [0,1]$ and
    \begin{gather*}
    \tau_k = 1+4\hat{c}_2\alpha_k^2, 
    \quad
    c_k = \tfrac{1}{\alpha_k^d(1-L_M)^2} + \tfrac 1b(1 - \alpha_k^d), \\
    \hat{c}_1 \coloneqq L^2_{\widehat{xz}}\|\Gamma D_{\widehat{xz}}\|+2(1-\theta)^2L^2_{\widehat{yz}}\|\Gamma D_{\widehat{yz}}\|+2\theta^2L^2_{\widehat{yy}}\|\Gamma_2D_{\widehat{yy}}\|, 
    \quad
    \hat{c}_2 \coloneqq 2\theta^2L^2_{\widehat{yx}}\|\Gamma_1 D_{\widehat{yx}}\|, 
    \quad
    \hat{c}_3 \coloneqq L^2_{\widehat{xz}}\|\Gamma D_{\widehat{xz}}\|, \\
    L_M^2 \coloneqq \max\big\{L_{xx}^2\|D_{xx}\Gamma_1\|+L_{yx}^2\|D_{yx}\Gamma_1\|, \|L_{xy}^2\|D_{xy}\Gamma_2\|+L_{yy}^2\|D_{yy}\Gamma_2\|\big\}.
    \end{gather*}
    Then, the sequence \(\seq{z^k}\) generated by \Cref{alg:NP-PDHG} converges almost surely to some \(z^\star\in \zer T\).
    Furthermore, if $\alpha_k=\tfrac{1}{k+r}$ for some positive natural number $r$, then \eqref{app:thm:AsymPrecon:almostsure:conditions1} is satisfied and it suffice to assume
    \begin{equation}\label{app:thm:AsymPrecon:almostsure:conditions:linear}
    \tfrac{2\rho}{\bar \gamma} - \alpha_k - 2(\hat{c}_1 + 2\hat{c}_2(1+\hat{c}_3))\left(\tfrac 1b (1-\alpha_k)\alpha_{k+1}+\tfrac{1}{(1-L_M)^2}\right)\left(\alpha_{k+1}+1\right)\alpha_{k+1} + \tfrac{1}{1+b}(1 - \alpha_k) < 0.
    \end{equation}
\end{thm}

\begin{appendixproof}{app:thm:AFBA:almostsure}
We continue from the conditions in \eqref{eq:AFBA:recursion} which we restate here for convenience:
\begin{equation}
\label{eq:AFBA:recursion2}
X_k^1(1- \alpha_{k})^2(1+4\hat{c}_2\alpha_{k}^2) \leq A_k 
\quad \text{and} \quad
X_k^2 > 0,
\end{equation}
where 
\begin{equation}
\begin{split}
X_k^1 & \coloneqq A_{k+1} + \alpha_k (\varepsilon_k + \tfrac 1b(1 - \tfrac {1}{\varepsilon_k (1-L_M)^2})) \\
X_k^2 & \coloneqq \tfrac{2\rho}{\bar \gamma} - \alpha_k - \alpha_k 2(\hat{c}_1 + 2\hat{c}_2(1+\hat{c}_3))A_{k+1} + \tfrac{1}{1+b}(1 - \tfrac {1}{\varepsilon_k (1-L_M)^2}).
\end{split}
\end{equation}
Under \eqref{eq:AFBA:recursion2} the descent inequality \eqref{eq:AFBA:descent} reduces to
\begin{equation}
\label{eq:AFBA:descent2}
\mathbb E[\U_{k+1} \; |\; \mathcal F_k]  - \U_{k}
\leq -\alpha_k X_k^2 \|s^k - S_{z^k}(\z^k)\|^2_{\Gamma} + \alpha_{k}^2\zeta_k,
\end{equation}
where $\zeta_k = \big(2X_k^1(\Theta+(1-\alpha_{k})^22\hat{c}_2) + (1+2(\hat{c}_1 + 2\hat{c}_2(1+\hat{c}_3))A_{k+1})\big)\sigma_F^2$.

To show almost sure convergence through the Robbins-Siegmund supermartingale theorem \cite[Prop. 2]{Bertsekas2011Incremental}, we need $0 \leq A_k < \infty$ (inside $\U_k$), $X_k^2$ and $\zeta_k$ to be nonnegative and furthermore $\sum_{k=0}^\infty \alpha_{k}^2\zeta_k < \infty$.
To make a concrete choice of $A_k$ we solve the first (linear) inequality of \eqref{eq:AFBA:recursion2} to equality with \Cref{lem:recur}.
Letting $\tau_k \coloneqq 1+4\hat{c}_2\alpha_k^2$ and $c_k \coloneqq \varepsilon_k + \tfrac 1b(1 - \tfrac {1}{\varepsilon_k (1-L_M)^2})$ we have
\begin{align*}
    A_{k+1} 
        {}={}
    \left(%
        \Pi_{p=0}^{k}\frac{1}{(1-\alpha_{p})^{2}\tau_p}
    \right)
        \left(
            A_{0}-\sum_{\ell=0}^{k}
            \left(%
                c_{l}\alpha_{l}\Pi_{p=0}^{\ell}(1-\alpha_{p})^{2}\tau_p
            \right)
        \right)
        {}={}
    \eta_{k+1} \Pi_{p=0}^{k}\frac{1}{(1-\alpha_{p})^{2}\tau_p} 
    \numberthis\label{eq:AFBA:almostsure:Akrec}
\end{align*}
where the last equality follows from picking
\begin{align*}
    A_0
    {}\coloneqq {}
    \sum_{\ell=0}^{\infty}
        \left(%
            c_{l}\alpha_{l}\Pi_{p=0}^{\ell}(1-\alpha_{p})^{2}\tau_p
        \right) 
    \quad \text{and} \quad
    \eta_{k} 
        {}\coloneqq{}
    \textstyle
    \sum_{\ell=k}^{\infty}
        \left(%
            c_{l}\alpha_{l}\Pi_{p=0}^{\ell}(1-\alpha_{p})^{2}\tau_p
        \right)
        \overrel[<]{\eqref{app:thm:AsymPrecon:almostsure:conditions1}}\infty.
\end{align*}
This choice ensures \(A_k\geq 0\) for all \(k\) and consequently $\zeta_k \geq 0$ and $X_k^2 \geq 0$ as long as $(1-\tfrac{1}{\varepsilon_k (1-L_M)^2}) > 0$.
It remains to show boundedness of the cumulative noise terms.
\begin{align*}
\sum_{k=0}^\infty \alpha_{k}^2\zeta_k
    &= 2\sigma_F^2\sum_{k=0}^\infty (\Theta+(1-\alpha_{k})^22\hat{c}_2) \alpha_{k}^2 A_{k+1}
     + \sigma_F^2(1+2(\hat{c}_1 + 2\hat{c}_2(1+\hat{c}_3)))\sum_{k=0}^\infty \alpha_{k}^2 A_{k+1} \\
    &\quad + \sum_{k=0}^\infty \alpha_{k}^3 \sigma_F^2(\varepsilon_k + \tfrac 1b(1 - \tfrac {1}{\varepsilon_k (1-L_M)^2}))(\Theta+(1-\alpha_{k})^22\hat{c}_2) \\
    &\leq 2\sigma_F^2(\Theta+2\hat{c}_2)\sum_{k=0}^\infty \alpha_{k}^2 A_{k+1}
     + \sigma_F^2(1+2(\hat{c}_1 + 2\hat{c}_2(1+\hat{c}_3)))\sum_{k=0}^\infty \alpha_{k}^2 A_{k+1} \\
    &\quad 
        + \sigma_F^2\tfrac{1}{(1-L_M)^2}(\Theta+2\hat{c}_2)\sum_{k=0}^\infty \alpha_{k}^{3-d} 
        + \sigma_F^2\tfrac 1b(\Theta+2\hat{c}_2)\sum_{k=0}^\infty \alpha_{k}^3 (1 - \alpha_k^{d})
\numberthis \label{eq:AFBA:as:boundednoise}
\end{align*}
where the last inequality follows from $(1-\alpha_k)^2 \leq 1$ and picking $\varepsilon_k = \tfrac{1}{\alpha_k^d(1-L_M)^2}$ with $d \in [0,1]$ to ensure $1-\tfrac{1}{\varepsilon_k (1-L_M)^2} > 0$.
Assuming $\nu \coloneqq \sum_{k=0}^\infty \alpha_{k}^2A_{k+1} < \infty$ as in \eqref{app:thm:AsymPrecon:almostsure:conditions1} is sufficient to show that \eqref{eq:AFBA:as:boundednoise} is bounded.
This finishes the proof for the first claim of \Cref{app:thm:AFBA:almostsure}. 

To provide an instance of the sequence \(\seq{\alpha_k}\) that satisfy the assumptions, let \(r\) denote a positive natural number and set 
\begin{equation} \label{eq:AFBA:almostsure:alphak}
    \alpha_k = \tfrac{1}{k+r}. 
\end{equation}
Then, 
\begin{align*}
    \Pi_{p=0}^{\ell}(1-\alpha_{p})^{2}\tau_p 
        {}={}&
    \Pi_{p=0}^{\ell}(\tfrac{p+r-1}{p+r})^{2}(1+4\hat{c}_2\alpha_p^2)
        {}={}
    \tfrac{(r-1)^2}{(\ell+r)^2}(1+4\hat{c}_2 \Pi_{p=0}^{\ell}\tfrac{1}{(p+r)^2})
        {}={}
    \tfrac{(r-1)^2}{(\ell+r)^2}
    \omega_\ell
\end{align*}
with $\omega_\ell = (1 + r^{2(1-\ell)}4\hat{c}_2) \in [1, 1+4\hat{c}_2]$ owing to the fact that $\Pi_{p=0}^{\ell}\tfrac{1}{p+r} = r^{1-\ell}$.
It follows that, for any \(K \geq 0\) 
\begin{align*}
    \sum_{\ell=0}^{K}\left(c_{\ell}\alpha_{\ell}\Pi_{p=0}^{\ell}(1-\alpha_{p})^{2}\tau_p\right)
        {}={}
    \sum_{\ell=0}^{K}c_{\ell}\tfrac{(r-1)^2}{(\ell+r)^{3}}\omega_\ell.
\end{align*}
Plugging the value of \(c_\ell\) and \(\varepsilon_k\) from \eqref{eq:AFBA:almostsure:Akrec} and \eqref{eq:AFBA:as:boundednoise} we obtain that \(A_0\) is finite valued since $b>0$, $L_M < 1$ and \(\sum_{\ell=0}^{\infty}\tfrac{\varepsilon_{\ell}}{(\ell+r)^3}=\sum_{\ell=0}^{\infty}\tfrac{1}{(\ell+r)^{3-d}(1-L_M)^2}<\infty\) owing to the fact that \(d\leq 1\). 

Moreover, 
\begin{align*}
     A_{k+1} 
        {}={}
        \frac{(k+r)^{2}}{(r-1)^{2}\omega_k}
        \left(
            A_{0}-\sum_{\ell=0}^{k}
            \left(%
                \tfrac{(r-1)^2\omega_\ell}{(\ell+r)^{3}}c_{\ell}
            \right)
        \right)
            {}={}
        (k+r)^{2}
            \sum_{\ell=k+1}^{\infty}
                \tfrac{\omega_\ell}{(\ell+r)^{3}\omega_k}c_{\ell}
            {}={}
        \tfrac{1}{\alpha_k^2}
            \sum_{\ell=k+1}^{\infty}
                \alpha_\ell^3c_{\ell}\tfrac{\omega_\ell}{\omega_k}
            {}\leq{}
        \tfrac{1}{\alpha_k^2}
            \sum_{\ell=k+1}^{\infty}
                \alpha_\ell^3c_{\ell}
        \numberthis\label{eq:AFBA:almostsure:Ak1}
\end{align*}
where the last inequality follows from $\tfrac{\omega_i}{\omega_j} \leq \tfrac{1}{1+4\hat{c}_2} \leq 1$ for any $i,j\in \N$.

On the other hand, for \(e>1\) we have the following bound  %
\begin{align*}
    \sum_{\ell=k+1}^{\infty}\alpha_{\ell}^{e} 
        {}\leq{}
    \tfrac{1}{(k+1+r)^{e}} + \int_{k+1}^{\infty} \tfrac{1}{(x+r)^e}dx 
        {}={}
    \tfrac{1}{(k+1+r)^{e}}+\tfrac{1}{(e-1)(k+1+r)^{e-1}}.
    \numberthis\label{eq:AFBA:almostsure:intHarmonic}
\end{align*}
Therefore, recalling $c_k = \alpha_k^{-d}\tfrac{1}{(1-L_M)^2}+\tfrac 1b (1-\alpha_k^d)$, it follows from \eqref{eq:AFBA:almostsure:Ak1} that  
\begin{align*}
      \alpha_k A_{k+1} 
        {}\leq{}&
    \tfrac{1}{\alpha_k}\sum_{\ell=k+1}^{\infty} \big(%
        \alpha_k^2\tfrac 1b (1-\alpha_k^d) + \alpha_k^{3-d}\tfrac{1}{(1-L_M)^2}
    \big)
    \\
    \dueto{\eqref{eq:AFBA:almostsure:intHarmonic}}
        {}\leq{}&
    \left(
        \tfrac 1b (1-\alpha_k^d)\tfrac{1}{2(k+1+r)}
        +\tfrac{1}{(1-L_M)^2}\tfrac{1}{(2-d)(k+1+r)^{1-d}}
      \right)\left(\tfrac{1}{k+1+r}+1\right)\tfrac{1}{k+1+r}
    \\
        {}={}&
    \left(\tfrac 1b (1-\alpha_k^d)\alpha_{k+1}+\tfrac{1}{(1-L_M)^2(2-d)}\alpha_{k+1}^{1-d}\right)\left(\alpha_{k+1}+1\right)\alpha_{k+1}.
    \numberthis\label{eq:AFBA:almostsure:Aalpha}
  \end{align*}  
In turn, this inequality ensures that \(\nu\) as defined in \cref{app:thm:AsymPrecon:almostsure:conditions1} is finite. To see this note that 
\begin{align*}
    \textstyle
    \nu 
        {}={}
    \sum_{k=0}^{\infty} A_{k+1}\alpha_k^2 
        {}\overrel[\leq]{\eqref{eq:AFBA:almostsure:Aalpha}}{} 
    \sum_{k=0}^{\infty} \left(\tfrac 1b (1-\alpha_k^d)\alpha_{k+1}+\tfrac{1}{(1-L_M)^2(2-d)}\alpha_{k+1}^{1-d}\right)\left(\alpha_{k+1}+1\right)\alpha_{k+1}\alpha_k
        {}\leq{}
    \delta \sum_{k=0}^{\infty} \alpha_k^2 <\infty.
\end{align*}

It remains to simplify the remaining condition \eqref{app:thm:AsymPrecon:almostsure:conditions2}. 
Using \eqref{eq:AFBA:almostsure:Aalpha}
\begin{align*}
\tfrac{2\rho}{\bar \gamma} &- \alpha_k - \alpha_k 2(\hat{c}_1 + 2\hat{c}_2(1+\hat{c}_3))A_{k+1} + \tfrac{1}{1+b}(1 - \alpha_k^d)  \\
&\leq \tfrac{2\rho}{\bar \gamma} - \alpha_k - 2(\hat{c}_1 + 2\hat{c}_2(1+\hat{c}_3))\left(\tfrac 1b (1-\alpha_k^d)\alpha_{k+1}+\tfrac{1}{(1-L_M)^2(2-d)}\alpha_{k+1}^{1-d}\right)\left(\alpha_{k+1}+1\right)\alpha_{k+1} + \tfrac{1}{1+b}(1 - \alpha_k^d) \\
&= \tfrac{2\rho}{\bar \gamma} - \alpha_k - 2(\hat{c}_1 + 2\hat{c}_2(1+\hat{c}_3))\left(\tfrac 1b (1-\alpha_k)\alpha_{k+1}+\tfrac{1}{(1-L_M)^2}\right)\left(\alpha_{k+1}+1\right)\alpha_{k+1} + \tfrac{1}{1+b}(1 - \alpha_k) \\
&\overrel[<]{\eqref{app:thm:AsymPrecon:almostsure:conditions:linear}} 0
\end{align*}
where the equality follows from choosing \(d=1\).
This completes the proof.
\end{appendixproof}
}
    \subsection{Explanation of bias-correction term}\label{app:bias-correction}
    Consider the naive analysis which would track $h^k$.
By the definition of $\bar z^k$ in \eqref{eq:AFBA:bar} and $\HC[k][\z^k]$ we would have $h^k - \HC[k][\z^k] +  \PC[k][\z^k] - \SPC[k][\z^k][\bar \xi_k] \in F(\bar z^k) + A(\bar z^k)$. 
Hence, assuming zero mean and using the weak MVI from \Cref{ass:AsymPrecon:Minty:Struct}, 
\begin{equation}
\begin{split}
\mathbb E [\langle h^k - H_k(\bar z^k), \bar z^k - z^\star\rangle \;|\; \F_k']
&= \mathbb E [\langle h^k - H_k(\bar z^k) + \PC[k][\z^k] - \SPC[k][\z^k][\xi_k'], \bar z^k - z^\star\rangle \;|\; \F_k'] \\
&\geq \mathbb E [\rho \|h^k - H_k(\bar z^k) + \PC[k][\z^k] - \SPC[k][\z^k][\xi_k']\|^2 \;|\; \F_k'] \; .
\end{split}
\end{equation}
To proceed we could apply Young's inequality, but this would produce a noise term, which would propagate to the descent inequality in \eqref{eq:AFBA:descent} with a $\alpha_k$ factor in front.
To show convergence we would instead need a smaller factor of $\alpha_k^2$.

To avoid this error term entirely we instead do a change of variables with $s^k \coloneqq h^k{}-\SPC[z^k][\bar{z}^k][\xi_k']$ such that,
\begin{equation}
\begin{split}
h^k{}\in{}\SPC[z^k][\bar{z}^k][\xi_k']+A\bar{z}^k 
&{}\Leftrightarrow h^k{}-\SPC[z^k][\bar{z}^k][\xi_k']\in{}A\bar{z}^k \\
&{}\Leftrightarrow s^k\in{}A\bar{z}^k.
\end{split}
\end{equation}
This make application of \Cref{ass:AsymPrecon:Minty:Struct} unproblematic, but affects the choice of the bias-correction term, since the analysis will now apply to $s^k$.
If we instead of the careful choice of $h^k$ in \eqref{eq:AFBA:h} had made the choice
\begin{equation}
h^k = \SPC[z^k][z^k][\xi_k]-\hat{F}(z^k,\xi_k) + (1 - \alpha_{k})(h^{k-1} - \SPC[z^{k-1}][z^{k-1}][\xi_k]+\hat{F}(z^{k-1},\xi_k))
\end{equation}
then 
\begin{equation*}
s^k = \SPC[z^k][z^k][\xi_k]-\hat{F}(z^k,\xi_k)-\SPC[z^k][\bar{z}^k][\xi_k'] + (1 - \alpha_{k})(s^{k-1} - \SPC[z^{k-1}][z^{k-1}][\xi_k]+\hat{F}(z^{k-1},\xi_k)-\SPC[z^{k-1}][\bar{z}^{k-1}][\xi_{k-1}']).
\end{equation*}
Notice how the latter term is evaluated under $\xi_{k-1}'$ instead of $\xi_{k}'$.
The choice in \eqref{eq:AFBA:h} resolves this issue.

\section{Negative weak Minty variational inequality}\label{app:negF}

In this section we consider the problem of finding a zero of the single-valued operator \(F\) (with the set-valued operator \(A\equiv 0\)).  Observe that the weak MVI in \cref{ass:AsymPrecon:Minty:Struct}, 
\(
\langle Fz,z - z^{\star}\rangle \geq \rho\|Fz\|^2, 
\)
for all $z\in \R^n$, 
is not symmetric and one may instead consider that the assumption holds for \(-F\).  As we will see below this simple observation leads to nontrivial problem classes extending the reach of extragradient-type methods both in the deterministic and stochastic settings.  
\begin{ass}[negative weak MVI]
\label{ass:basic:negative} 
     There exists a nonempty set $\mathcal S^{\star}\subseteq \zer T$ such that for all $z^\star\in \mathcal S^{\star}$ and some $\bar\rho\in(-\infty,\nicefrac1{2L})$
    \begin{equation}\label{eq:neg:WMVI}
        \langle Fz,z - z^{\star}\rangle \leq \bar\rho\|Fz\|^2, \quad \text{for all $z\in \R^n$.}
    \end{equation}
\end{ass}
Under this assumption the algorithm of \cite{pethick2022escaping} leads to the following modified iterates: 
\begin{align}\label{eq:iter:constant}
    \bar{z}^k 
        {}={}&
    z^k {+}  \gamma_k Fz^k,
    \\
    z^{k+1}
        {}={}&
    z^k
        {}+{}
    \lambda_k\alpha_k (H_k\z^k - H_kz^k) 
        {}={}
      z^k
        {}{+}{}
    \lambda_k\alpha_k \gamma_k F\z^k 
    ,\quad \text{where}\quad H_k\coloneqq \id {+} \gamma_k F
\end{align}

We next consider the lower bound example of \cite[Ex. 5]{pethick2022escaping} to
show that despite the condition for weak MVI being violated for \(b\) smaller than a certain threshold, the negative weak MVI in \cref{ass:basic:negative} holds for any negative \(b\) and thus the extragradient method applied to \(-F\) is guaranteed to converge. 
\begin{example}
    Consider \cite[Ex. 5]{pethick2022escaping}
    \begin{equation}
\label{eq:lowerbound-example}
\minimize_{x \in \mathbb R} \maximize_{y \in \mathbb R} f(x,y) := a x y+\frac{b}{2}(x^2-y^2),
\end{equation}
where $b<0$ and $a>0$. The associated \(F\) is a linear mapping. For a linear mapping M, \cref{ass:basic:negative} holds if 
    \[
        \tfrac12(M + M^\top) - \bar \rho M^{\top}M \preceq 0, \qquad \bar \rho\in(-\infty, \nicefrac1{2L})
    \]
    While \cref{ass:AsymPrecon:Minty:Struct} holds if 
    \[
        \tfrac12(M + M^\top)- \rho M^{\top}M \succeq 0,\qquad \rho\in(-\nicefrac1{2L}, \infty).
    \] 
For this example \(L={\sqrt{a^2+b^2}}\) and 
\[
F(z) = \overbrace*{(bx + ay, -ax + by)}^{\eqqcolon Mz}. 
\] 
Since \(M\) is a bisymmetric linear mapping, \(M^\top M = (a^2+b^2)\I\) which according to the above characterizations implies  
\[
\rho \in (-\tfrac1{2L}, \tfrac{b}{a^2+b^2}], \qquad \bar \rho \in [\tfrac{b}{a^2+b^2}, \tfrac1{2L}).
\]
The range for \(\rho\) is nonempty if \(b > -\tfrac{a}{\sqrt{3}}\) while this is not an issue for \(\bar \rho\) which allows any negative \(b\).

\end{example}

We complete this section with a corollary to \cref{thm:BiasCorr} when replacing weak MVI assumption with \cref{ass:basic:negative}. 
\begin{cor}
    Suppose that \cref{ass:A:Struct,ass:AsymPrecon:M:Lip}, \cref{ass:AsymPrecon:stoch,ass:AsymPrecon:stoch:stocLips,ass:basic:negative} hold. Let \(\seq{z^k}\) denote the sequence generated by \Cref{alg:WeakMinty:Sto:Struct} applied to \(-F\). Then, the claims of \cref{thm:BiasCorr} hold true. 
\end{cor}

\section{Experiments}\label{app:experiments}
    \begin{figure}[t]
\centering
\includegraphics[width=0.8\textwidth]{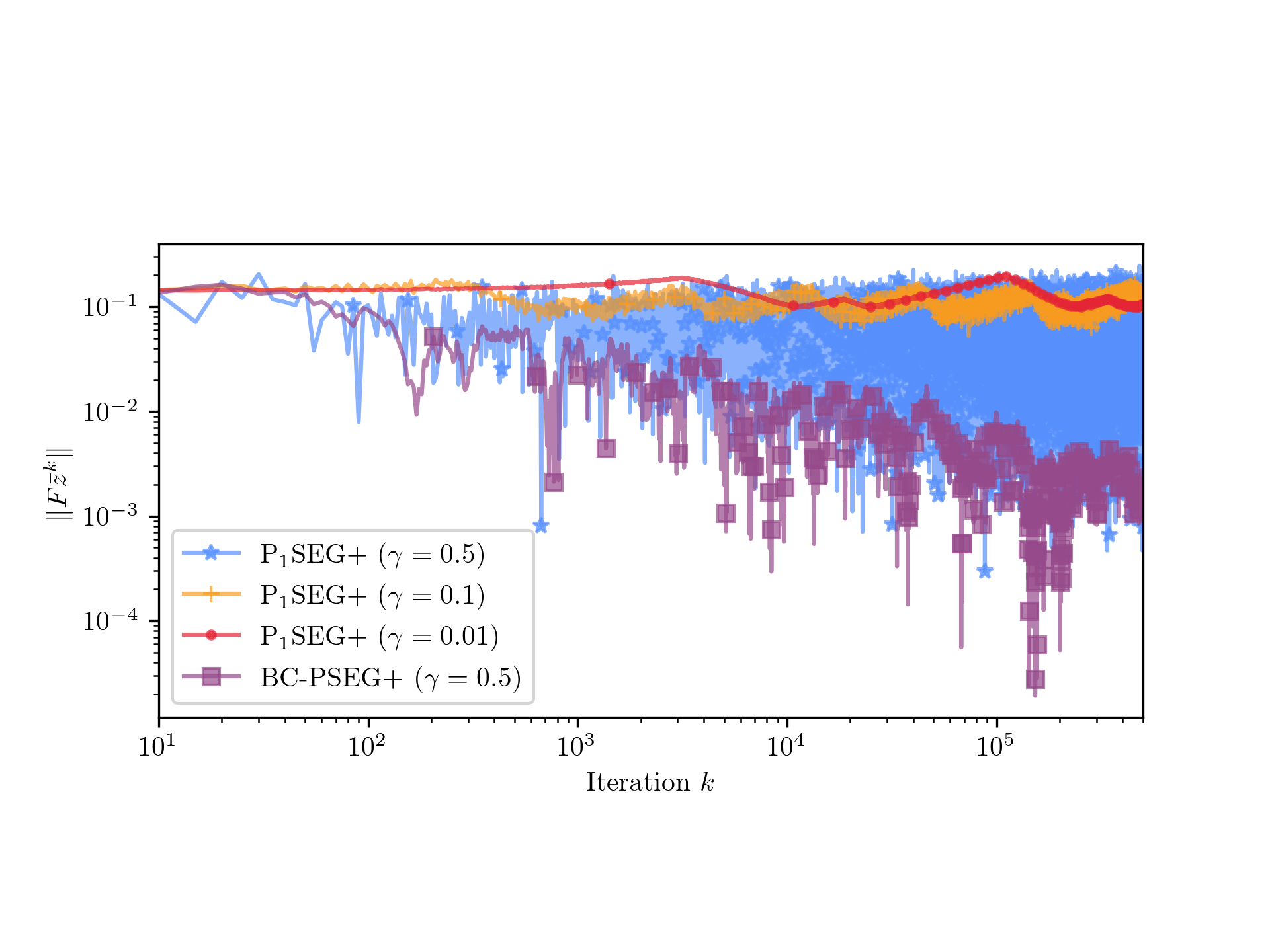}
\vspace{-2em}
\caption{
  \rbl{
The (projected) \eqref{eq:seg+} method needs to take $\gamma$ arbitrarily small to guarantee convergence to an arbitrarily small neighborhood. 
We show an instance satisfying the weak MVI where $\gamma$ cannot be taken arbitrarily small.
The objective is $\psi(x,y)=\phi(x-0.9,y-0.9)$ under box constraints $\|(x,y)\|_\infty\leq 1$ with $\phi$ from \Cref{ex:quadratic} where $L=1$ and $\rho = -\nicefrac{1}{10L}$.
The unique stationary point $(x^\star,y^\star)=(0.9,0.9)$ lies in the interior, so even $\|Fz\|$ can be driven to zero.
Taking $\gamma$ smaller does \emph{not} make the neighborhood smaller as oppose to the monotone case in \Cref{fig:monotone}.}
}
\label{fig:SEG+:counterexample}
\end{figure}

\begin{figure}[t]
\centering
\includegraphics[width=0.5\textwidth]{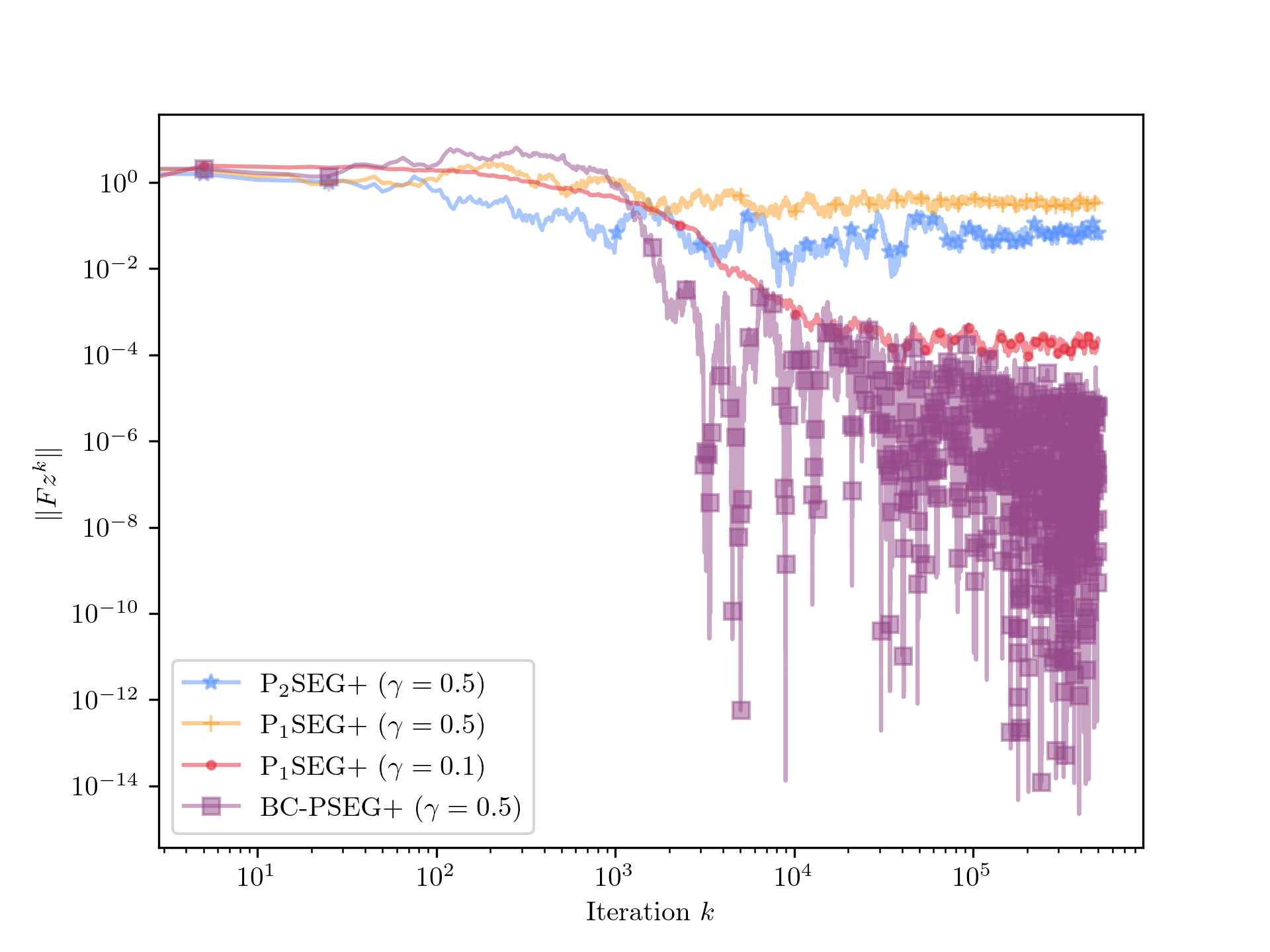}%
\includegraphics[width=0.5\textwidth]{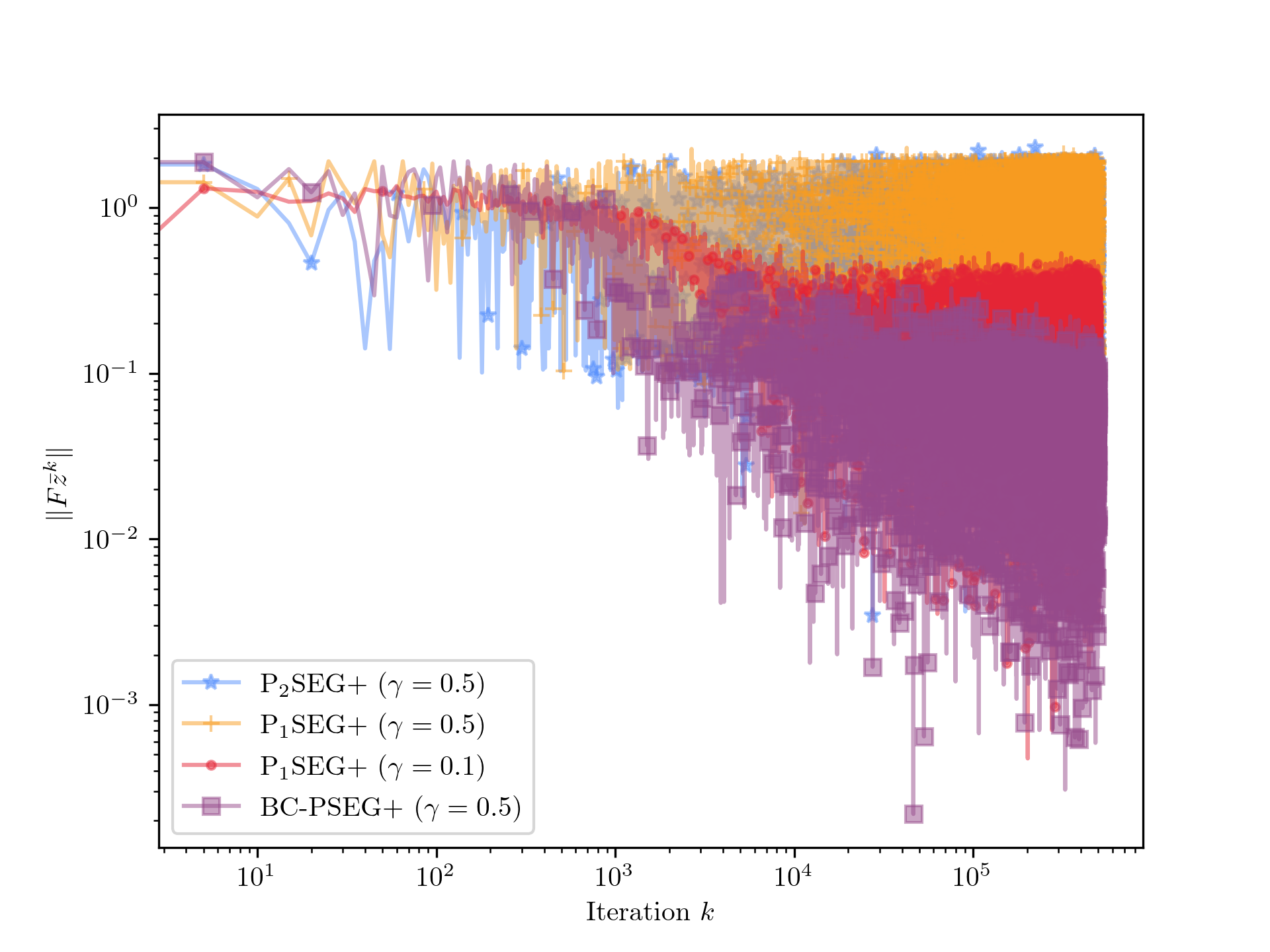}\\
\includegraphics[width=0.5\textwidth]{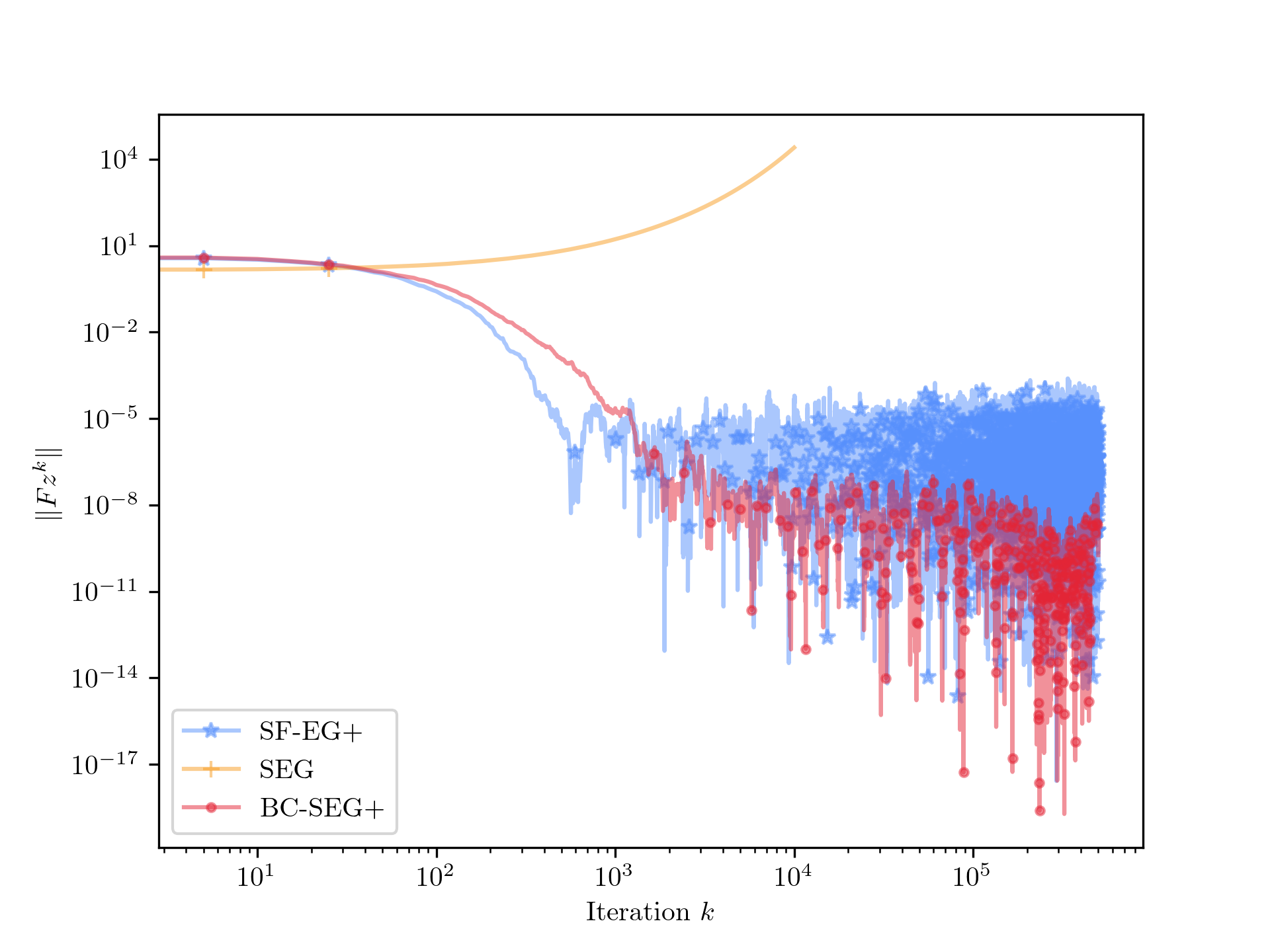}%
\includegraphics[width=0.5\textwidth]{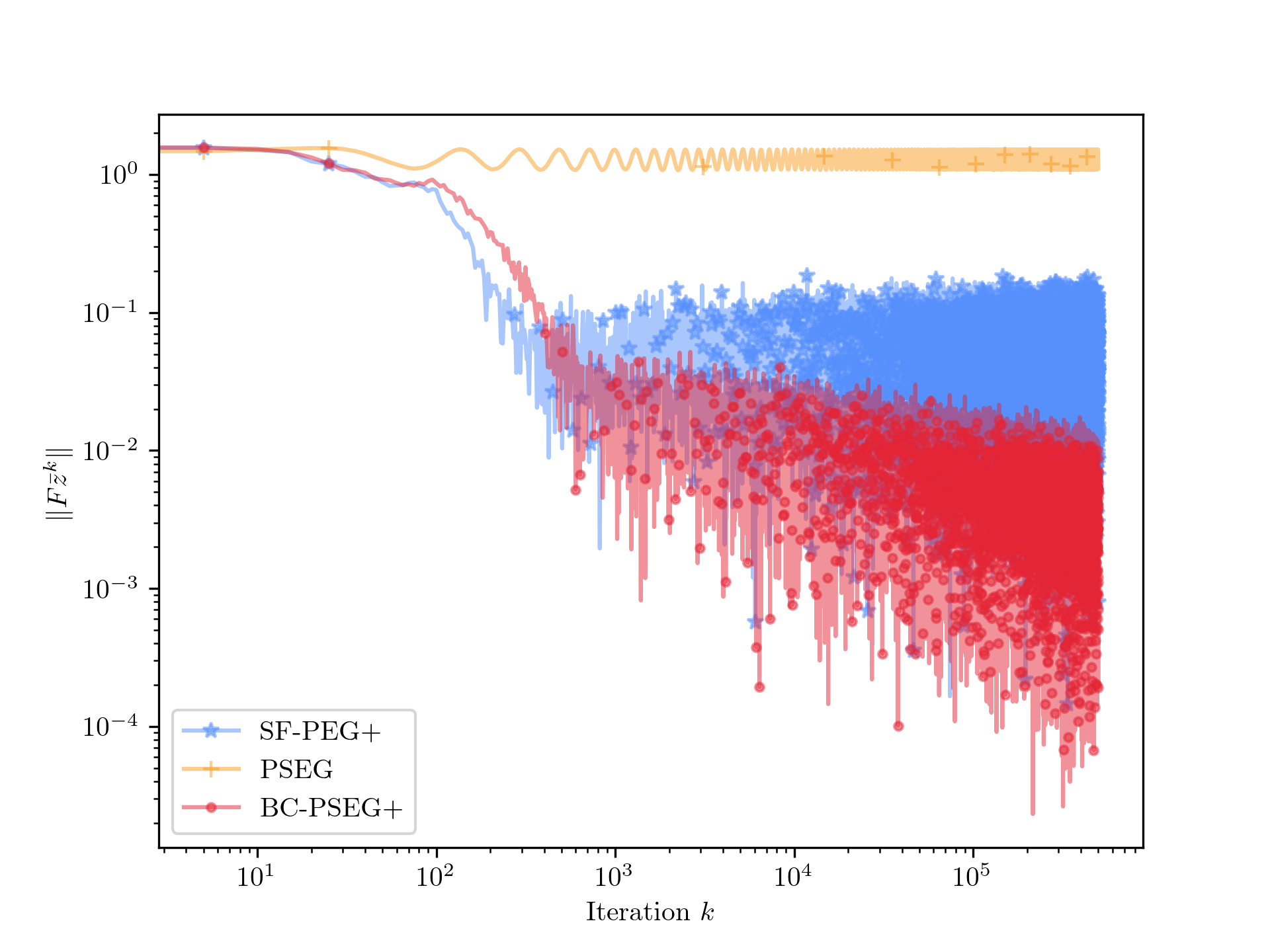}
\caption{\rbl{Instead of taking $\alpha_k \propto \nicefrac{1}{k}$ (for which almost sure convergence is established through \Cref{thm:BiasCorr,app:thm:AFBA:almostsure}) we take $\alpha_k \propto \nicefrac{1}{\sqrt{k}}$ as permitted in \Cref{thm:BiasCorr:2,thm:const:convergence}.
  We consider the example provided in \Cref{fig:monotone} (top row) and the two examples from \Cref{fig:simulations} (bottom row).
  Under this more aggressive stepsize schedule the guarantee is only in expectation over the iterates which is also apparent from the relatively large volatility in comparison with \Cref{fig:monotone,fig:simulations}.}
}
\label{fig:aggressive-step}
\end{figure}

\subsection{Synthetic example}

\begin{example}[{Unconstrained quadratic game \cite[Ex. 5]{pethick2022escaping}}]
\label{ex:quadratic}
Consider,
\begin{equation}
\minimize_{x \in \mathbb R} \maximize_{y  \in \mathbb R} \phi(x,y) := axy + \frac{b}{2}x^2 - \frac{b}{2}y^2,
\end{equation}
where $a \in \mathbb R_+$ and $b \in \mathbb R$.
\end{example}
The problem constants in \Cref{ex:quadratic} can easily be computed as $\rho = \frac{b}{a^2+b^2}$ and $L = \sqrt{a^2+b^2}$.
We can rewrite \Cref{ex:quadratic} in terms of $L$ and $\rho$ by choosing
$a=\sqrt{L^2-L^4 \rho ^2}$ and $b = L^2 \rho$.

\begin{example}[{Constrained minimax \cite[Ex. 4]{pethick2022escaping}}]
\label{ex:globalforsaken}
Consider
\begin{equation}
\label{eq:globalforsaken}
\tag{GlobalForsaken}
\minimize_{|x|\leq\nicefrac{4}{3}} \maximize_{|y|\leq\nicefrac{4}{3}} \phi(x,y):=xy+\psi(x)-\psi(y),
\end{equation}
where $\psi(z) = \frac{2 z^6}{21}-\frac{z^4}{3}+\frac{z^2}{3}$.
\end{example}

In both \Cref{ex:quadratic} and \Cref{ex:globalforsaken} the operator $F$ is defined as $Fz = (\nabla_x \phi(x,y), -\nabla_y \phi(x,y))$.

To simulate a stochastic setting in all examples, we consider additive Gaussian noise, i.e.
$\hat F(z,\xi) = Fz + \xi$ where $\xi \sim \mathcal N(0,\sigma^2 I)$.
We choose $\sigma = 0.1$ and initialize with $z^0=1$ if not specified otherwise.
The default configuration is $\gamma=\nicefrac{1}{2L_F}$ with $\alpha_k=\nicefrac{1}{18 \cdot (k/c+1)}$, $c=100$ and $\beta_k = \alpha_k$ for diminishing stepsize schemes \rbl{and $\alpha=\nicefrac{1}{18}$ for fixed stepsize schemes.
We make two exceptions: \Cref{fig:monotone} uses the slower decay $c=1000$ when $\gamma=0.1$ and \Cref{fig:SEG+:counterexample} uses $c=5000$ for $\gamma=0.01$ (and otherwise $c=1000$) to ensure fast enough convergence.
When the aggressive stepsize schedule is used then $\alpha_k=\nicefrac{1}{18 \cdot \sqrt{k/100+1}}$.}

\subsection{Additional algorithmic details}\label{app:algo}

For the constrained setting in \Cref{fig:monotone}, we consider two extensions of \eqref{eq:seg+}.
One variant uses a single application of the resolvent as suggested by \citet{pethick2022escaping},
\begin{equation}
\label{eq:p1seg+}
\tag{$P_1$SEG+}
\begin{split}
\bar{z}^k &= (\id+\gamma A)^{-1}(z^k - \gamma \hat F(z^k, \xi_k)) \ \quad \text{with}\quad \xi_k \sim \mathcal{P}  \\
z^{k+1}&=z^{k}+\alpha_{k}\left((\bar{z}^{k}-z^{k})-\gamma(\hat F(\bar z^k, \bar\xi_k)-\hat F(z^k, \xi_k))\right) \quad \text{with}\quad \bar{\xi}_k \sim \mathcal{P}  \\
\end{split}
\end{equation}
The other variant applies the resolvent twice as in stochastic Mirror-Prox \citep{juditsky2011solving},
\begin{equation}
\label{eq:p2seg+}
\tag{$P_2$SEG+}
\begin{split}
\bar{z}^k &= (\id+\gamma A)^{-1}(z^k - \gamma \hat F(z^k, \xi_k)) \ \quad \text{with}\quad \xi_k \sim \mathcal{P}  \\
z^{k+1} &= (\id+\alpha_k\gamma A)^{-1}(z^k - \alpha_k \gamma \hat F(\bar{z}^k, \bar{\xi}_k)) \quad \text{with}\quad \bar{\xi}_k \sim \mathcal{P}  \\
\end{split}
\end{equation}
When applying \eqref{eq:seg} to constrained settings we similarly use the following projected variants:
\begin{equation}
\label{eq:PSEG}
\tag{PSEG}
\begin{split}
\bar{z}^k &= (\id+\beta_k\gamma A)^{-1}(z^k - \beta_k \gamma \hat F(z^k, \xi_k)) \ 
\quad \text{with}\quad \xi_k \sim \mathcal{P}  \\
z^{k+1} &= (\id+\alpha_k\gamma A)^{-1}(z^k - \alpha_k \gamma \hat F(\bar{z}^k, \bar{\xi}_k)) 
\quad \text{with}\quad \bar{\xi}_k \sim \mathcal{P}  \\
\end{split}
\end{equation}
and \eqref{eq:eg+} \rbl{(using stochastic feedback denoted SF)}
\begin{equation}
\label{eq:SF-PEG+}
\tag{SF-PEG+}
\begin{split}
\bar{z}^k &= (\id+\gamma A)^{-1}(z^k - \gamma \rbl{\hat F(z^k, \xi_k)) 
\quad \text{with}\quad \xi_k \sim \mathcal{P}} \\
z^{k+1} &= (\id+\alpha\gamma A)^{-1}(z^k - \alpha\gamma \rbl{\hat F(\z^k, \bar \xi_k))
\quad \text{with}\quad \bar{\xi}_k \sim \mathcal{P}}
\end{split}
\end{equation}
\rbl{which we in the unconstrained case ($A\equiv 0$) refer to as \eqref{eq:SF-EG+} as defined below.
\begin{equation}
\label{eq:SF-EG+}
\tag{SF-EG+}
\begin{split}
\bar{z}^k &= z^k - \gamma \hat F(z^k, \xi_k) 
\quad \text{with}\quad \xi_k \sim \mathcal{P} \\
z^{k+1} &= z^k - \alpha\gamma \hat F(\z^k, \bar \xi_k)
\quad \text{with}\quad \bar{\xi}_k \sim \mathcal{P}
\end{split}
\end{equation}}

    \section{Comparison with variance reduction}\label{sec:finiteSum}
        Consider the case where the expectation comes in the form a finite sum,
\begin{equation}
Fz = \frac 1N\sum_{\xi = 1}^N \hat F(z, \xi).
\end{equation}
In the worst case the averaged Lipschitz constant $F_{\hat F}$ scales proportionally to the number of elements $N$ squared, i.e. $L_{\hat F} = \Omega(\sqrt{N} L_F)$.
It is easy to construct such an example by taking one elements to have Lipschitz constant $N L$ while letting the remaining elements have Lipschitz constant $L$. 
Recalling the definition in \Cref{ass:AsymPrecon:stoch:stocLips}, $L_{\hat F}^2 = \frac{N^2 L^2}{N} + \frac{N-1}{N}L^2\geq N L^2$ while the average becomes $L_F=\frac{N L}{N} + \frac{N-1}{N}L \leq 2L$ so $L_{\hat F} \geq \nicefrac{\sqrt{N}}{2L_F}$. 
Thus, $L_{\hat F}$ can be $\sqrt{N}$ times larger than $L_F$, leading to a potentially strict requirement on the weak MVI parameter $\rho > -\nicefrac{L_{\hat F}}{2}$ for variance reduction methods.

\end{document}